\documentclass[11pt]{amsart}

\usepackage{amsmath}
\usepackage{amssymb}
\usepackage{amscd}
\usepackage{color}
\usepackage{hyperref}
\usepackage{mathrsfs}
\usepackage{tikz}

\usepackage{xy}
\xyoption{all}

\newcommand{\nc}{\newcommand}

\renewcommand{\AA}{{\mathbb{A}}}
\nc{\CC}{{\mathbb{C}}}
\nc{\EE}{{\mathbb{E}}}
\nc{\LL}{{\mathbb{L}}}
\nc{\RR}{{\mathbb{R}}}
\nc{\PP}{{\mathbb{P}}}
\renewcommand{\P}{{\mathbb{P}}}
\nc{\OO}{{\mathbb{O}}}

\nc{\QQ}{{\mathbb{Q}}}
\nc{\VV}{{\mathbb{V}}}
\nc{\WW}{{\mathbb{W}}}
\nc{\XX}{{\mathbb{X}}}
\nc{\YY}{{\mathbb{Y}}}
\nc{\ZZ}{{\mathbb{Z}}}

\nc{\cA}{{\mathcal{A}}}
\nc{\cB}{{\mathscr{B}}}
\nc{\cC}{{\mathscr{C}}}
\nc{\cD}{{\mathscr{D}}}
\nc{\cE}{{\mathscr{E}}}
\nc{\cF}{{\mathscr{F}}}
\nc{\cG}{{\mathscr{G}}}
\nc{\cH}{{\mathscr{H}}}
\nc{\cI}{{\mathscr{I}}}
\nc{\cJ}{{\mathscr{J}}}
\nc{\cK}{{\mathscr{K}}}
\nc{\cL}{{\mathscr{L}}}
\nc{\cM}{{\mathscr{M}}}
\nc{\cN}{{\mathscr{N}}}
\nc{\cO}{{\mathscr{O}}}
\nc{\cP}{{\mathscr{P}}}
\nc{\cQ}{{\mathscr{Q}}}
\nc{\cR}{{\mathscr{R}}}
\nc{\cS}{{\mathscr{S}}}
\nc{\cT}{{\mathscr{T}}}
\nc{\cU}{{\mathscr{U}}}
\nc{\cV}{{\mathscr{V}}}
\nc{\cW}{{\mathscr{W}}}
\nc{\cX}{{\mathscr{X}}}
\nc{\cY}{{\mathscr{Y}}}
\nc{\cZ}{{\mathscr{Z}}}

\nc{\bA}{{\mathbf{A}}}
\nc{\bB}{{\mathbf{B}}}
\nc{\bC}{{\mathbf{C}}}
\nc{\bD}{{\mathbf{D}}}
\nc{\bE}{{\mathbf{E}}}
\nc{\bF}{{\mathbf{F}}}
\nc{\bG}{{\mathbf{G}}}
\nc{\bH}{{\mathbf{H}}}
\nc{\bI}{{\mathbf{I}}}
\nc{\bJ}{{\mathbf{J}}}
\nc{\bK}{{\mathbf{K}}}
\nc{\bL}{{\mathbf{L}}}
\nc{\bM}{{\mathbf{M}}}
\nc{\bN}{{\mathbf{N}}}
\nc{\bO}{{\mathbf{O}}}
\nc{\bP}{{\mathbf{P}}}
\nc{\bQ}{{\mathbf{Q}}}
\nc{\bR}{{\mathbf{R}}}
\nc{\bS}{{\mathbf{S}}}
\nc{\bT}{{\mathbf{T}}}
\nc{\bU}{{\mathbf{U}}}
\nc{\bV}{{\mathbf{V}}}
\nc{\bW}{{\mathbf{W}}}
\nc{\bX}{{\mathbf{X}}}
\nc{\bY}{{\mathbf{Y}}}
\nc{\bZ}{{\mathbf{Z}}}

\nc{\ba}{{\mathbf{a}}}
\nc{\bb}{{\mathbf{b}}}
\nc{\bc}{{\mathbf{c}}}
\nc{\bd}{{\mathbf{d}}}
\nc{\be}{{\mathbf{e}}}
\nc{\bg}{{\mathbf{g}}}
\nc{\bh}{{\mathbf{h}}}
\nc{\bi}{{\mathbf{i}}}
\nc{\bj}{{\mathbf{j}}}
\nc{\bk}{{\mathbf{k}}}
\nc{\bl}{{\mathbf{l}}}
\nc{\bm}{{\mathbf{m}}}
\nc{\bn}{{\mathbf{n}}}
\nc{\bo}{{\mathbf{o}}}
\nc{\bp}{{\mathbf{p}}}
\nc{\bq}{{\mathbf{q}}}
\nc{\br}{{\mathbf{r}}}
\nc{\bs}{{{s}}}
\nc{\bt}{{\mathbf{t}}}
\nc{\bu}{{\mathbf{u}}}
\nc{\bv}{{\mathbf{v}}}
\nc{\bw}{{\mathbf{w}}}
\nc{\bx}{{\mathbf{x}}}
\nc{\by}{{\mathbf{y}}}
\nc{\bz}{{\mathbf{z}}}

\nc{\fA}{{\mathfrak{A}}}
\nc{\fB}{{\mathfrak{B}}}
\nc{\fC}{{\mathfrak{C}}}
\nc{\fD}{{\mathfrak{D}}}
\nc{\fE}{{\mathfrak{E}}}
\nc{\fF}{{\mathfrak{F}}}
\nc{\fG}{{\mathfrak{G}}}
\nc{\fH}{{\mathfrak{H}}}
\nc{\fI}{{\mathfrak{I}}}
\nc{\fJ}{{\mathfrak{J}}}
\nc{\fK}{{\mathfrak{K}}}
\nc{\fL}{{\mathfrak{L}}}
\nc{\fM}{{\mathfrak{M}}}
\nc{\fN}{{\mathfrak{N}}}
\nc{\fO}{{\mathfrak{O}}}
\nc{\fP}{{\mathfrak{P}}}
\nc{\fQ}{{\mathfrak{Q}}}
\nc{\fR}{{\mathfrak{R}}}
\nc{\fS}{{\mathfrak{S}}}
\nc{\fT}{{\mathfrak{T}}}
\nc{\fU}{{\mathfrak{U}}}
\nc{\fV}{{\mathfrak{V}}}
\nc{\fW}{{\mathfrak{W}}}
\nc{\fX}{{\mathfrak{X}}}
\nc{\fY}{{\mathfrak{Y}}}
\nc{\fZ}{{\mathfrak{Z}}}

\nc{\fa}{{\mathfrak{a}}}
\nc{\fb}{{\mathfrak{b}}}
\nc{\fc}{{\mathfrak{c}}}
\nc{\fd}{{\mathfrak{d}}}
\nc{\fe}{{\mathfrak{e}}}
\nc{\ff}{{\mathfrak{f}}}
\nc{\fg}{{\mathfrak{g}}}
\nc{\fh}{{\mathfrak{h}}}
\nc{\fj}{{\mathfrak{j}}}
\nc{\fk}{{\mathfrak{k}}}
\nc{\fl}{{\mathfrak{l}}}
\nc{\fm}{{\mathfrak{m}}}
\nc{\fn}{{\mathfrak{n}}}
\nc{\fo}{{\mathfrak{o}}}
\nc{\fp}{{\mathfrak{p}}}
\nc{\fq}{{\mathfrak{q}}}
\nc{\fr}{{\mathfrak{r}}}
\nc{\fs}{{\mathfrak{s}}}
\nc{\ft}{{\mathfrak{t}}}
\nc{\fu}{{\mathfrak{u}}}
\nc{\fv}{{\mathfrak{v}}}
\nc{\fw}{{\mathfrak{w}}}
\nc{\fx}{{\mathfrak{x}}}
\nc{\fy}{{\mathfrak{y}}}
\nc{\fz}{{\mathfrak{z}}}

\nc{\sA}{{\mathsf{A}}}
\nc{\sB}{{\mathsf{B}}}
\nc{\sC}{{\mathsf{C}}}
\nc{\sD}{{\mathsf{D}}}
\nc{\sE}{{\mathsf{E}}}
\nc{\sF}{{\mathsf{F}}}
\nc{\sG}{{\mathsf{G}}}
\nc{\sH}{{\mathsf{H}}}
\nc{\sI}{{\mathsf{I}}}
\nc{\sJ}{{\mathsf{J}}}
\nc{\sK}{{\mathsf{K}}}
\nc{\sL}{{\mathsf{L}}}
\nc{\sM}{{\mathsf{M}}}
\nc{\sN}{{\mathsf{N}}}
\nc{\sO}{{\mathsf{O}}}
\nc{\sP}{{\mathsf{P}}}
\nc{\sQ}{{\mathsf{Q}}}
\nc{\sR}{{\mathsf{R}}}
\nc{\sS}{{\mathsf{S}}}
\nc{\sT}{{\mathsf{T}}}
\nc{\sU}{{\mathsf{U}}}
\nc{\sV}{{\mathsf{V}}}
\nc{\sW}{{\mathsf{W}}}
\nc{\sX}{{\mathsf{X}}}
\nc{\sY}{{\mathsf{Y}}}
\nc{\sZ}{{\mathsf{Z}}}

\nc{\sa}{{\mathsf{a}}}
\nc{\sd}{{\mathsf{d}}}
\nc{\se}{{\mathsf{e}}}
\nc{\sg}{{\mathsf{g}}}
\nc{\sh}{{\mathsf{h}}}
\nc{\si}{{\mathsf{i}}}
\nc{\sj}{{\mathsf{j}}}
\nc{\sk}{{\mathsf{k}}}
\nc{\sm}{{\mathsf{m}}}
\nc{\sn}{{\mathsf{n}}}
\nc{\so}{{\mathsf{o}}}
\nc{\sq}{{\mathsf{q}}}
\nc{\sr}{{\mathsf{r}}}
\nc{\st}{{\mathsf{t}}}
\nc{\su}{{\mathsf{u}}}
\nc{\sv}{{\mathsf{v}}}
\nc{\sw}{{\mathsf{w}}}
\nc{\sx}{{\mathsf{x}}}
\nc{\sy}{{\mathsf{y}}}
\nc{\sz}{{\mathsf{z}}}

\nc{\oA}{{\overline{A}}}
\nc{\oB}{{\overline{B}}}
\nc{\oC}{{\overline{C}}}
\nc{\oD}{{\overline{D}}}
\nc{\oE}{{\overline{E}}}
\nc{\oF}{{\overline{F}}}
\nc{\oG}{{\overline{G}}}
\nc{\oH}{{\overline{H}}}
\nc{\oI}{{\overline{I}}}
\nc{\oJ}{{\overline{J}}}
\nc{\oK}{{\overline{K}}}
\nc{\oL}{{\overline{L}}}
\nc{\oM}{{\overline{M}}}
\nc{\oN}{{\overline{N}}}
\nc{\oO}{{\overline{O}}}
\nc{\oP}{{\overline{P}}}
\nc{\oQ}{{\overline{Q}}}
\nc{\oR}{{\overline{R}}}
\nc{\oS}{{\overline{S}}}
\nc{\oT}{{\overline{T}}}
\nc{\oU}{{\overline{U}}}
\nc{\oV}{{\overline{V}}}
\nc{\oW}{{\overline{W}}}
\nc{\oX}{{\overline{X}}}
\nc{\oY}{{\overline{Y}}}
\nc{\oZ}{{\overline{Z}}}

\nc{\oa}{{\overline{a}}}
\nc{\ob}{{\overline{b}}}
\nc{\oc}{{\overline{c}}}
\nc{\od}{{\overline{d}}}
\nc{\of}{{\overline{f}}}
\nc{\og}{{\overline{g}}}
\nc{\oh}{{\overline{h}}}
\nc{\oi}{{\overline{i}}}
\nc{\oj}{{\overline{j}}}
\nc{\ok}{{\overline{k}}}
\nc{\ol}{{\overline{l}}}
\nc{\om}{{\overline{m}}}
\nc{\on}{{\overline{n}}}
\nc{\oo}{{\overline{o}}}
\nc{\op}{{\overline{p}}}
\nc{\oq}{{\overline{q}}}
\nc{\os}{{\overline{s}}}
\nc{\ot}{{\overline{t}}}
\nc{\ou}{{\overline{u}}}
\nc{\ov}{{\overline{v}}}
\nc{\ow}{{\overline{w}}}
\nc{\ox}{{\overline{x}}}
\nc{\oy}{{\overline{y}}}
\nc{\oz}{{\overline{z}}}

\nc{\tA}{{\tilde{A}}}
\nc{\tB}{{\tilde{B}}}
\nc{\tC}{{\tilde{C}}}
\nc{\tD}{{\tilde{D}}}
\nc{\tE}{{\tilde{E}}}
\nc{\tF}{{\tilde{F}}}
\nc{\tG}{{\tilde{G}}}
\nc{\tH}{{\tilde{H}}}
\nc{\tI}{{\tilde{I}}}
\nc{\tJ}{{\tilde{J}}}
\nc{\tK}{{\tilde{K}}}
\nc{\tL}{{\tilde{L}}}
\nc{\tM}{{\tilde{M}}}
\nc{\tN}{{\tilde{N}}}
\nc{\tO}{{\tilde{O}}}
\nc{\tP}{{\widetilde{P}}}
\nc{\tQ}{{\tilde{Q}}}
\nc{\tR}{{\tilde{R}}}
\nc{\tS}{{\tilde{S}}}
\nc{\tT}{{\tilde{T}}}
\nc{\tU}{{\tilde{U}}}
\nc{\tV}{{\tilde{V}}}
\nc{\tW}{{\tilde{W}}}
\nc{\tX}{{\widetilde{X}}}
\nc{\tY}{{\tilde{Y}}}
\nc{\tZ}{{\tilde{Z}}}

\nc{\hcX}{{\widehat{\cX}}}
\nc{\tcX}{{\widetilde{\cX}}}

\nc{\hpi}{{\hat{\pi}}}
\nc{\hrho}{{\hat{\rho}}}

\nc{\ta}{{\tilde{a}}}
\nc{\tb}{{\tilde{b}}}
\nc{\tc}{{\tilde{c}}}
\nc{\td}{{\tilde{d}}}
\nc{\te}{{\tilde{e}}}
\nc{\tf}{{\tilde{f}}}
\nc{\tg}{{\tilde{g}}}
\nc{\ti}{{\tilde{i}}}
\nc{\tj}{{\tilde{j}}}
\nc{\tk}{{\tilde{k}}}
\nc{\tl}{{\tilde{l}}}
\nc{\tm}{{\tilde{m}}}
\nc{\tn}{{\tilde{n}}}
\nc{\tp}{{\tilde{p}}}
\nc{\tq}{{\tilde{q}}}
\nc{\tr}{{\tilde{r}}}
\nc{\ts}{{\tilde{s}}}
\nc{\tu}{{\tilde{u}}}
\nc{\tv}{{\tilde{v}}}
\nc{\tw}{{\tilde{w}}}
\nc{\tx}{{\tilde{x}}}
\nc{\ty}{{\tilde{y}}}
\nc{\tz}{{\tilde{z}}}

\nc{\hA}{{\hat{A}}}
\nc{\hB}{{\hat{B}}}
\nc{\hC}{{\hat{C}}}
\nc{\hD}{{\hat{D}}}
\nc{\hE}{{\hat{E}}}
\nc{\hF}{{\hat{F}}}
\nc{\hG}{{\hat{G}}}
\nc{\hH}{{\hat{H}}}
\nc{\hI}{{\hat{I}}}
\nc{\hJ}{{\hat{J}}}
\nc{\hK}{{\hat{K}}}
\nc{\hL}{{\hat{L}}}
\nc{\hM}{{\hat{M}}}
\nc{\hN}{{\hat{N}}}
\nc{\hO}{{\hat{O}}}
\nc{\hP}{{\hat{P}}}
\nc{\hQ}{{\hat{Q}}}
\nc{\hR}{{\hat{R}}}
\nc{\hS}{{\hat{S}}}
\nc{\hT}{{\hat{T}}}
\nc{\hU}{{\hat{U}}}
\nc{\hV}{{\hat{V}}}
\nc{\hW}{{\hat{W}}}
\nc{\hX}{{\hat{X}}}
\nc{\hY}{{\hat{Y}}}
\nc{\hZ}{{\hat{Z}}}

\nc{\ha}{{\hat{a}}}
\nc{\hb}{{\hat{b}}}
\nc{\hc}{{\hat{c}}}
\nc{\hd}{{\hat{d}}}
\nc{\he}{{\hat{e}}}
\nc{\hf}{{\hat{f}}}
\nc{\hg}{{\hat{g}}}
\nc{\hh}{{\hat{h}}}
\nc{\hi}{{\hat{i}}}
\nc{\hj}{{\hat{j}}}
\nc{\hk}{{\hat{k}}}
\nc{\hl}{{\hat{l}}}
\nc{\hm}{{\hat{m}}}
\nc{\hn}{{\hat{n}}}
\nc{\ho}{{\hat{o}}}
\nc{\hp}{{\hat{p}}}
\nc{\hq}{{\hat{q}}}
\nc{\hr}{{\hat{r}}}
\nc{\hs}{{\hat{s}}}
\nc{\hu}{{\hat{u}}}
\nc{\hv}{{\hat{v}}}
\nc{\hw}{{\hat{w}}}
\nc{\hx}{{\hat{x}}}
\nc{\hy}{{\hat{y}}}
\nc{\hz}{{\hat{z}}}

\nc{\rO}{{\mathrm{O}}}
\nc{\rT}{{\mathrm{T}}}
\nc{\rD}{{\mathrm{D}}}
\nc{\rG}{{\mathrm{G}}}

\nc{\tcA}{{\widetilde{\cA}}}
\nc{\tcE}{{\widetilde{\cE}}}
\nc{\tfE}{{\widetilde{\fE}}}
\nc{\tcF}{{\widetilde{\cF}}}

\nc{\eps}{\varepsilon}
\nc{\lan}{\big\langle}
\nc{\ran}{\big\rangle}
\nc{\kk}{{\mathsf{k}}}
\nc{\bkk}{{\bar\kk}}

\nc{\bDelta}{{\boldsymbol\Delta}}

\nc{\quand}{\quad\text{and}\quad}
\nc{\qquand}{\qquad\text{and}\qquad}

\def\bw#1#2{\textstyle{\bigwedge\hskip-0.9mm^{#1}}\hskip0.2mm{#2}}

\nc{\perf}{\mathrm{perf}}

\DeclareMathOperator{\pr}{\mathrm{pr}}

\DeclareMathOperator{\Hom}{\mathrm{Hom}}
\DeclareMathOperator{\Ext}{\mathrm{Ext}}
\DeclareMathOperator{\End}{\mathrm{End}}
\DeclareMathOperator{\RHom}{\mathrm{RHom}}
\DeclareMathOperator{\Tor}{\mathrm{Tor}}

\DeclareMathOperator{\CExt}{\mathscr{E}\mathit{xt}}

\DeclareMathOperator{\RCHom}{\mathrm{R}\mathscr{H}\mathit{om}}

\DeclareMathOperator{\Hilb}{\mathrm{Hilb}}
\DeclareMathOperator{\Spec}{\mathrm{Spec}}

\DeclareMathOperator{\Bl}{\mathrm{Bl}}

\DeclareMathOperator{\Pic}{\mathrm{Pic}}

\DeclareMathOperator{\coh}{\mathrm{coh}}

\DeclareMathOperator{\Ker}{\mathrm{Ker}}

\DeclareMathOperator{\Ima}{\mathrm{Im}}
\DeclareMathOperator{\Cone}{\mathrm{Cone}}

\DeclareMathOperator{\Gr}{\mathrm{Gr}}

\DeclareMathOperator{\Fl}{\mathrm{Fl}}

\DeclareMathOperator{\Gm}{{\mathrm{G}_{\mathrm{m}}}}

\DeclareMathOperator{\PGL}{\mathrm{PGL}}

\DeclareMathOperator{\id}{\mathrm{id}}
\DeclareMathOperator{\rank}{\mathrm{rk}}

\DeclareMathOperator{\ch}{\mathrm{ch}}

\DeclareMathOperator{\Br}{\mathrm{Br}}

\DeclareMathOperator{\mmod}{\mathrm{-mod}}

\DeclareMathOperator{\DP}{{\mathfrak{DP}_6}}

\theoremstyle{plain}

\newtheorem{theorem}{Theorem}[section]

\newtheorem{lemma}[theorem]{Lemma}
\newtheorem{proposition}[theorem]{Proposition}
\newtheorem{corollary}[theorem]{Corollary}

\theoremstyle{definition}

\newtheorem{definition}[theorem]{Definition}

\theoremstyle{remark}

\newtheorem{remark}[theorem]{Remark}

\title{Derived categories of families of sextic del Pezzo surfaces}
\author{Alexander Kuznetsov}
\address{{\sloppy
\parbox{0.9\textwidth}{
Algebraic Geometry Section, Steklov Mathematical Institute of Russian Academy of Sciences,\\
8 Gubkin str., Moscow 119991 Russia
\\[5pt]
The Poncelet Laboratory, Independent University of Moscow
\hfill\\[5pt]
Laboratory of Algebraic Geometry, National Research University Higher School of Economics, Russian Federation
}\bigskip}}
\email{akuznet@mi.ras.ru}
\date{}
\thanks{I was partially supported by the Russian Academic Excellence Project ``5-100'', by RFBR grant 15-01-02164, and by the Simons foundation.}

\begin{document}

\begin{abstract}
We construct a natural semiorthogonal decomposition for the derived category of an arbitrary flat family of sextic del Pezzo surfaces with at worst du Val singularities.
This decomposition has three components equivalent to twisted derived categories of finite flat schemes of degrees 1, 3, and 2 over the base of the family.
We provide a modular interpretation for these schemes and compute them explicitly in a number of standard families.
For two such families the computation is based on a symmetric version of homological projective duality 
for $\P^2 \times \P^2$ and $\P^1 \times \P^1 \times \P^1$, which we explain in an appendix.
\end{abstract}

\maketitle


\section{Introduction}

In this paper we describe the bounded derived category of coherent sheaves on an arbitrary flat family 
of del Pezzo surfaces of canonical degree~6 with du Val singularities.
Our description  clarifies and makes more precise results of~\cite{blunk2011derived} and~\cite{auel2015semiorthogonal}
about \emph{smooth} sextic del Pezzo surfaces over non-closed fields.
We expect our results to be useful for a description of derived categories of varieties, 
that admit a structure of a family of sextic del Pezzo surfaces.
There are at least two interesting examples of this sort.

One example is provided by special cubic fourfolds of discriminant 18. 
In~\cite{addington2016cubic} it was shown, that a general such cubic fourfold contains an elliptic scroll, 
and after its blowup the cubic fourfold acquires a structure of a family of sextic del Pezzo surfaces over $\P^2$
(this is quite similar to the case of cubic fourfolds containing a plane, when blowing up the plane one gets a family of two-dimensional quadrics over $\P^2$).
Another example is provided by Gushel--Mukai fourfolds (\cite{debarre-kuznetsov, kuznetsov-perry}) containing a Veronese surface.
The results of this paper should have a direct application in these two cases and provide a description of the derived categories of cubic and Gushel--Mukai fourfolds of these types,
and in particular, a geometric interpretation of their K3 categories (see~\cite{kuznetsov2010cubic,kuznetsov2016perry}).


Our main result (Theorem~\ref{theorem:sod-dp6}) proves, 
that given a flat family $\cX \to S$ all of whose fibers are sextic del Pezzo surfaces with at worst du Val singularities,
there are two finite flat morphisms $\cZ_2 \to S$ and~$\cZ_3 \to S$ of degrees $3$ and $2$ respectively, 
with Brauer classes $\beta_{\cZ_2}$ and $\beta_{\cZ_3}$ of order $2$ and $3$ respectively,
and an~$S$-linear semiorthogonal decomposition of the bounded derived category of coherent sheaves
\begin{equation*}
\bD(\cX) = \langle \bD(S), \bD(\cZ_2,\beta_{\cZ_2}), \bD(\cZ_3,\beta_{\cZ_3}) \rangle,
\end{equation*}
where the second and the third components are the twisted derived categories.

\medskip

To construct the semiorthogonal decomposition, we first investigate in detail the case 
when $S$ is the spectrum of an algebraically closed field $\kk$, 
and so $\cX$ is just a single sextic del Pezzo surface $X$ over~$\kk$ with du Val singularities.
In this case, to describe $\bD(X)$ we first consider the minimal resolution of singularities~\mbox{$\pi \colon \tX \to X$}.
Here $\tX$ is a weak del Pezzo surface of degree 6; it has at most three $(-2)$-curves contracted by $\pi$, 
which in the worst case form two chains of lengths 2 and 1.
The category $\bD(X)$ then can be identified with the Verdier quotient of the category $\bD(\tX)$ 
by the subcategory generated by the sheaves~$\cO_\Delta(-1)$, 
for $\Delta$ running through the set of all $(-2)$-curves.

On the other hand, the surface $\tX$ can be realized as an iterated blowup of $\P^2$, and so comes with a natural exceptional collection.
We mutate this collection slightly (Proposition~\ref{proposition:tx-excol}) 
to get a semiorthogonal decomposition $\bD(\tX) = \langle \tcA_1, \tcA_2, \tcA_3 \rangle$ 
such that for each $(-2)$-curve $\Delta$ the sheaf $\cO_\Delta(-1)$ is contained in one of the components $\tcA_i$ (Lemma~\ref{lemma:ker-pis});
when $X$ is smooth this semiorthogonal decomposition coincides with the three-block semiorthogonal decomposition from~\cite{karpov-nogin}.
After that we prove (Theorem~\ref{theorem:dbx}) that~$\bD(X)$ has a semiorthogonal decomposition $\bD(X) = \langle \cA_1, \cA_2, \cA_3 \rangle$, 
whose components are Verdier quotients of the categories $\tcA_i$ by the subcategories generated by appropriate sheaves~$\cO_\Delta(-1)$.
An explicit computation shows that the categories~$\tcA_i$ are equivalent to products of derived categories of so-called Auslander algebras,
and their Verdier quotients $\cA_i$ are equivalent to derived categories of zero-dimensional schemes of lengths~$1$, $3$, and~$2$ respectively.


This approach, however, does not generalize to families of del Pezzo surfaces, since one cannot construct a relative minimal resolution.
To deal with this problem, we go back to the case of a single del Pezzo surface $X$ (still over an algebraically closed field), 
and provide a \emph{modular interpretation} for the components of the constructed semiorthogonal decomposition. 
Namely, we show that the zero-dimensional schemes associated with the nontrivial components $\cA_2$ and $\cA_3$ 
(the component~$\cA_1$ is generated by the structure sheaf $\cO_X$ and has a natural counterpart in any family)
can be identified with the moduli spaces of semistable sheaves on $X$ with Hilbert polynomials $h_d(t) = (3t + d)(t+1)$ 
for~$d = 2$ and~$d = 3$ respectively, see Theorem~\ref{theorem:moduli}.
These moduli spaces turn out to be fine, and the corresponding universal families provide 
fully faithful Fourier--Mukai functors from derived categories of the moduli spaces into $\bD(X)$.

This description, of course, can be easily used in a family $\cX \to S$.
We consider the relative moduli spaces $\cM_d(\cX/S)$ of semistable sheaves on fibers of $\cX$ over $S$ with Hilbert polynomials $h_d(t)$.
Now, however, the moduli spaces need not to be fine, so we consider their coarse moduli spaces $\cZ_d$ and the Brauer obstruction classes $\beta_{\cZ_d}$ on them.
Then the universal families are well defined as $\beta_{\cZ_d}^{-1}$-twisted sheaves on $\cX \times_S \cZ_d$ and define Fourier--Mukai functors from the twisted derived categories $\bD(\cZ_d,\beta_{\cZ_d})$ to $\bD(\cX)$.
Using the results over an algebraically closed field described earlier, we show in Theorem~\ref{theorem:sod-dp6} that these functors are fully faithful, 
and together with the pullback functor $\bD(S) \to \bD(\cX)$ form the required semiorthogonal decomposition.


The question of understanding the derived category of a family $\cX \to S$ of sextic del Pezzo surfaces thus reduces to understanding the schemes $\cZ_2 \to S$ and $\cZ_3 \to S$ together with their Brauer classes.
We provide a Hilbert scheme interpretation of these. 
Namely, we show in Proposition~\ref{proposition:fd-md} that 
the relative Hilbert scheme $F_2(\cX/S)$ of conics in the fibers of $\cX \to S$ is a $\P^1$-bundle over $\cZ_2$ with associated Brauer class $\beta_{\cZ_2}$, and
the relative Hilbert scheme $F_3(\cX/S)$ of twisted cubic curves is a~$\P^2$-bundle over $\cZ_3$ with associated Brauer class $\beta_{\cZ_3}$.
We also prove that the relative Hilbert scheme of lines $F_1(\cX/S)$ can be written as $F_1(\cX/S) \cong \cZ_2 \times_S \cZ_3$.

Another useful result is the following regularity criterion.
We show that the total space $\cX$ of a flat family $\cX \to S$ of sextic del Pezzo surfaces with du Val singularities is regular 
if and only if the three schemes $S$, $\cZ_2$, and~$\cZ_3$, associated with it, are all regular (Proposition~\ref{proposition:smoothness});
in the same vein, the morphism $\cX \to S$ is smooth 
if and only if the morphisms $\cZ_2 \to S$ and $\cZ_3 \to S$ are (Remark~\ref{remark:smoothness}).

This leads to the following description of the schemes~$\cZ_2$ and~$\cZ_3$ in case of regular~$\cX$ --- 
the schemes~$\cZ_2$ and~$\cZ_3$ are isomorphic to the normal closures of their generic fibers over~$S$.
This shows that to understand~$\cZ_2$ and~$\cZ_3$ globally, it is enough to understand them over any dense open subset, 
or even over the general point of~$S$ if~$S$ is integral.
In particular, if $\cX \to S$ and $\cX' \to S$ are two families with regular total spaces 
and~$F_d(\cX/S)$ is birational (over~$S$) to~$F_d(\cX'/S)$ for some $d \in \{2, 3\}$, 
then~$\cZ_d(\cX/S) \cong \cZ_d(\cX'/S)$ and~$\beta_{\cZ_d(\cX/S)} = \beta_{\cZ_d(\cX'/S)}$ 
(Corollary~\ref{corollary:birational-isomorphism-fd}).
We expect this property to be very useful in geometric applications mentioned at the beginning of the Introduction.


We finish the paper by an explicit description of the schemes $\cZ_2$ and $\cZ_3$ for some ``standard'' families of sextic del Pezzo surfaces.

The first standard family is the family of codimension 2 linear sections of~$\P^2 \times \P^2$.
In this case we show that $\cZ_3 = S \sqcup S$, $\cZ_2$ is the scheme of ``degenerate linear equations'' of the fibers of $\cX$, 
and both Brauer classes are trivial, see Proposition~\ref{proposition:dp-p2p2}.

The second standard family is the family of hyperplane sections of~$\P^1 \times \P^1  \times \P^1$.
In this case we show that $\cZ_2 = S \sqcup S \sqcup S$, $\cZ_3$ is the double cover of $S$ branched over the divisor of ``degenerate linear equations'' of the fibers of $\cX$, 
and again both Brauer classes are trivial, see Proposition~\ref{proposition:dp-p1p1p1}.

In both cases we deduce the required description of $\bD(\cX)$ from a symmetric version of homological projective duality 
for $\P^2 \times \P^2$ and $\P^1 \times \P^1 \times \P^1$ respectively that we discuss in Appendices~\ref{appendix:hpd-p2p2}
and~\ref{appendix:hpd-p1p1p1}.
Note that the description via the homological projective duality allows to extend the general description of~$\bD(\cX)$ 
to a wider class of families of sextic del Pezzo surfaces, allowing in particular non-integral degenerations. 
In these cases the schemes~$\cZ_2$ and $\cZ_3$ controlling the components of $\bD(\cX)$ become non-flat over $S$ (see Remark~\ref{remark:extended-family}).

We also consider families of relative anticanonical models of the blowups of $\P^2$ (resp.\ of $\P^1 \times \P^1$) in length 3 (resp.\ length 2) subschemes.
We show that in these cases one of the schemes $\cZ_2$ and $\cZ_3$ coincides with the family of the blowup centers, while the other is obtained by gluing appropriate number of copies of $S$,
see Propositions~\ref{proposition:blowup-p2} and~\ref{proposition:blowup-p1p1} for details.


Of course, the approach used in this paper can be applied to del Pezzo families of other degree.
In case of a single del Pezzo surface over an algebraically closed field one should analyze 
possible configurations of $(-2)$-curves on its weak del Pezzo resolution and find a semiorthogonal decomposition
such that for any $(-2)$-curve $\Delta$ the sheaf $\cO_\Delta(-1)$ is contained in one of its components.
Most probably, (weak del Pezzo analogues of) the three-block collections of Karpov and Nogin~\cite{karpov-nogin} should be used here.
This approach definitely should work for del Pezzo surfaces of degrees $d \ge 5$, and we leave it to the readers to check the results it leads to.

For $d \le 4$, however, the results of~\cite[Theorem~5.1]{auel2015semiorthogonal} show 
that an $S$-linear semiorthogonal decomposition for a general flat family~$\cX \to S$ 
whose components are twisted derived categories of finite coverings of~$S$ does not exist.
So, a new idea is needed to treat this case.
We hope that the approach developed in~\cite{KKS} will be useful.

%


The paper is organized as follows.
In Section~\ref{section:preliminaries} we discuss the geometry of sextic del Pezzo surfaces with du Val singularities 
over an algebraically closed field and remind some general results about resolutions of rational singularities and Grothendieck duality.
In Section~\ref{section:dp6-single} we describe the derived category of a single del Pezzo surface with du Val singularities 
over an algebraically closed field.
In Section~\ref{section:moduli-spaces} we provide a modular interpretation for this description 
and discuss the relation of the corresponding moduli spaces to Hilbert schemes of curves.
In Section~\ref{section:families} we prove the main result of the paper --- the semiorthogonal decomposition 
of the derived category for a family of sextic del Pezzo surfaces with du Val singularities,
and discuss some properties of this decomposition.
In particular, we relate regularity of the total space~$\cX$ of the family to that of~$S$, $\cZ_2$, and~$\cZ_3$.
In Section~\ref{section:special-families} we describe the schemes $\cZ_2$ and $\cZ_3$ for standard families of sextic del Pezzo surfaces.

In Appendix~\ref{appendix:auslander} we discuss the derived categories of Auslander algebras and their relation to derived categories of zero-dimensional schemes.
In Appendix~\ref{appendix:moduli-stack} we show that the moduli stack of sextic del Pezzo surfaces is smooth.
Finally, in Appendices~\ref{appendix:hpd-p2p2} and~\ref{appendix:hpd-p1p1p1} we describe the symmetric homological projective duality for~$\P^2 \times \P^2$, $\Fl(1,2;3)$, and $\P^1 \times \P^1 \times \P^1$ respectively.


{\bf Conventions.}
Throughout the paper we work over a field $\kk$, whose characteristic is assumed to be distinct from~2 and~3.
In sections~\ref{section:preliminaries}, \ref{section:dp6-single} and~\ref{section:moduli-spaces} we assume 
that $\kk$ is algebraically closed, while in sections~\ref{section:families} and~\ref{section:special-families} we leave this assumption.
For a scheme $X$ we denote by $\bD(X)$ the bounded derived category of coherent sheaves on $X$, 
and unless something else is specified explicitly, this is what we mean by a derived category.
All functors that we consider are derived --- for instance $\otimes$ stands for the derived tensor product, 
$f^*$ and~$f_*$ stand for the derived pullback and pushforward functors.
If we want to consider the classical pullback or pushforward, we write $L_0f^*$ and $R^0f_*$ respectively 
(and similarly for other classical functors).
We think of the Brauer group of a scheme as of the group of Morita-equivalence classes of Azumaya algebras on it.
For a Brauer class $\beta$ on a scheme $X$ we denote by $\bD(X,\beta)$ the twisted bounded derived category of coherent sheaves.
We refer to~\cite{huybrechts} for an introduction into derived categories, and to~\cite{icm2014} and references therein 
for an introduction into semiorthogonal decompositions.


{\bf Acknowledgement.}
During the work on the paper I benefited from discussions with many people.
Let me mention 
Nick Addington,
Valery Alexeev, 
Marcello Bernardara,
Sasha Efimov,
Shinnosuke Okawa, 
Alex Perry, 
Yura Prokhorov,
Jenya Tevelev
--- I thank them all for their help.
I am also grateful to the referee for useful comments and suggestions.


\section{Preliminaries}
\label{section:preliminaries}

In this section $\kk$ is an algebraically closed field of characteristic distinct from~2 and~3.

\subsection{Sextic del Pezzo surfaces}
\label{subsection:dp6}

For purposes of this paper we adopt the following definition.

\begin{definition}\label{definition:dp6}
A {\sf sextic du Val del Pezzo surface} is a normal integral projective surface $X$ over a field $\kk$ with at worst du Val singularities and ample anticanonical class such that $K_X^2 = 6$.
\end{definition}

Recall that du Val surface singularities are just canonical singularities or, equivalently, rational double points.
In particular, any surface $X$ with du Val singularities is Gorenstein, so $\omega_X$ is a line bundle, $K_X$ is a Cartier divisor, and its square is well-defined.

Let $\pi \colon \tX \to X$ be the minimal (in particular crepant) resolution of singularities of $X$.
It is well-known (see, e.g.~\cite[Section~8.4.2]{dolgachev2012cag}) that the surface $\tX$ is rational and can be obtained from $\P^2$ by a sequence of three blowups of a point, i.e., we have a diagram
\begin{equation*}
X \xleftarrow{\ \pi\ } \tX = X_3 \to X_2 \to X_1 \to X_0 = \P^2,
\end{equation*}
where each map $X_i \to X_{i-1}$ is the blowup of a point $P_i \in X_{i-1}$.
We denote by $h$ the hyperplane class on $\P^2$ and its pullback to $\tX$.
We denote by $E_i \subset \tX$ the pullback (i.e., the total preimage) to $\tX$ of the exceptional divisor of $X_i \to X_{i-1}$ and by $e_i$ its class in $\Pic(\tX)$.
The following result is standard.

\begin{lemma}
\label{lemma:pic-tx}
We have $\Pic(\tX) \cong \ZZ\langle h, e_1,e_2,e_3 \rangle$, with $h^2 = 1$, $e_i^2 = -1$, $he_i = e_ie_j = 0$ for all $i \ne j$.
Moreover, 
\begin{equation}\label{eq:ktx}
K_\tX = -3h + e_1 + e_2 + e_3 = \pi^*K_X.
\end{equation}
\end{lemma}

The surface $X$ is the anticanonical model of $\tX$.
In other words, $X$ is obtained from $\tX$ by contraction of all $(-2)$-curves.
By~\cite[Section~8.4.2]{dolgachev2012cag} there are six possibilities for configurations of the blowup centers and $(-2)$-curves on $\tX$
(see Table~\ref{table:dp6} below for a picture).
\begin{description}
\item[Type 0] 
Neither of $P_i$ lies on the exceptional divisor in $X_{i-1}$ and their images in $\P^2$ do not lie on a common line.
Then $\tX$ contains no $(-2)$-curves and $X = \tX$ is smooth.
\item[Type 1] 
Neither of $P_i$ lies on the exceptional divisor in $X_{i-1}$ but their images in $\P^2$ lie on a common line.
Then $\tX$ contains a unique $(-2)$-curve (the strict transform $\Delta_{123}$ of that line) and~$X$ has one~$A_1$ singularity.
\item[Type 2] 
The point $P_2$ lies on the exceptional divisor of $X_1 \to X_0$, the point $P_3$ is away from the exceptional divisors, 
and the line through $P_1$ in the direction of $P_2$ on $\P^2$ does not pass through the image of $P_3$.
Then $\tX$ contains a unique $(-2)$-curve (the strict transform $\Delta_{12}$ of the exceptional divisor of $X_1 \to X_0$) and $X$ has one $A_1$ singularity.
\item[Type 3] 
The point $P_2$ lies on the exceptional divisor of $X_1 \to X_0$, the point $P_3$ is away from the exceptional divisors, 
but the line through $P_1$ in the direction of $P_2$ on $\P^2$ passes through the image of $P_3$.
Then $\tX$ contains two disjoint $(-2)$-curves 
(the strict transforms $\Delta_{123}$ and $\Delta_{12}$ of the line on~$\P^2$ and of the exceptional divisor of $X_1 \to X_0$, respectively) 
and $X$ has two $A_1$ singularities.
\item[Type 4] 
The point $P_2$ lies on the exceptional divisor of $X_1 \to X_0$, the point $P_3$ lies on the exceptional divisor of $X_2 \to X_1$, 
and the strict transform $L_{12}$ of the line through $P_1$ in the direction of $P_2$ does not contain $P_3$.
Then $\tX$ contains a 2-chain of $(-2)$-curves 
(the strict transforms $\Delta_{12}$ and $\Delta_{23}$ of the exceptional divisors of $X_1 \to X_0$ and $X_2 \to X_1$, respectively) 
and $X$ has one $A_2$ singularity.
\item[Type 5] 
The point $P_2$ lies on the exceptional divisor of $X_1 \to X_0$, the point $P_3$ lies on the exceptional divisor of $X_2 \to X_1$, 
and the strict transform of the line through $P_1$ in the direction of $P_2$ contains $P_3$.
Then $\tX$ contains a 2-chain of $(-2)$-curves and one more $(-2)$-curve disjoint from the chain (the strict transforms $\Delta_{123}$, $\Delta_{12}$ and $\Delta_{23}$ of the line 
and the exceptional divisors of~$X_1 \to X_0$ and $X_2 \to X_1$ respectively) and $X$ has one $A_2$ singularity and one $A_1$ singularity.
\end{description}

For reader's convenience we draw the configurations of exceptional curves on sextic del Pezzo surfaces of all types. 
Red thick lines are the $(-2)$-curves, while the thin lines are $(-1)$-curves. 
We denote by~$\bDelta = \bDelta(X)$ the set of all $(-2)$-curves on $\tX$. 
Note that the $(-2)$ curves (when they exist) on $\tX$ are contained in the following linear systems:
\begin{equation*}
\Delta_{12} = E_1 - E_2 \in |e_1 - e_2|,
\qquad 
\Delta_{23} = E_2 - E_3 \in |e_2 - e_3|,
\qquand
\Delta_{123} \in |h - e_1 - e_2 - e_3|.
\end{equation*}
We denote by $L_{ij}$ the strict transform of the line connecting (the images on $\P^2$ of) the points $P_i$ and $P_j$.

\begin{table}[h]
\nc{\xfactor}{.3}
\nc{\yfactor}{.25}
\begin{tabular}{|c|c|c|}
\hline
\begin{tikzpicture}[xscale = \xfactor, yscale = \yfactor]  
\draw (2,0) -- (1,2) node [left] {$E_1$} -- (0,4) -- (1,6) node [left] {$L_{13}$} -- (2,8) -- (4,8) node [above] {$E_3$} -- (6,8) -- (7,6) node [right] {$L_{23}$} -- (8,4) -- (7,2) node [right] {$E_2$} -- (6,0) -- (4,0) node [above] {$L_{12}$} -- (2,0);
\end{tikzpicture}
 
&

\begin{tikzpicture}[xscale = \xfactor, yscale = \yfactor]
\draw[ultra thick,  color = red] (0,1) -- (2,2) node [below] {$\Delta_{12}$} -- (4,3);
\draw (4,3) -- (6,2) node [below] {$E_2$} -- (8,1);
\draw (0,0) -- (0,4) node [left] {$L_{13}$} -- (0,8);
\draw (0,7) -- (4,7) node [above] {$E_3$} -- (6,7) node [right] {\hphantom{$\Delta_{123}$}};
\draw (8,0) -- (8,4) node [right] {$L_{12}$} -- (8,8);
\end{tikzpicture}

&

\begin{tikzpicture}[xscale = \xfactor, yscale = \yfactor]
\draw[ultra thick,  color = red] (0,1)  -- (2,2) node[above left] {$\Delta_{12}$} -- (4,3);
\draw[ultra thick,  color = red] (4,3) -- (6,2) node [above right] {$\Delta_{23}$} -- (8,1);
\draw (8,1)  -- (10,2) node[below right] {$E_{3}$} -- (12,3);
\draw (6,0) -- (6,8) node [left] {$L_{12}$};
\end{tikzpicture}

\\
\hline 
Type 0, $\bDelta = \varnothing$
&
Type 2, $\bDelta = \{ \Delta_{12} \}$
&
Type 4, $\bDelta = \{ \Delta_{12}, \Delta_{23} \}$

\\
\hline 
\hline

\begin{tikzpicture}[xscale = \xfactor, yscale = \yfactor]
\draw node [left] {\hphantom{$E_1$}} (0,1) -- (8,1) node [right] {$E_1$};
\draw (0,4) -- (8,4) node [right] {$E_2$};
\draw (0,7) -- (8,7) node [right] {$E_3$};
\draw[ultra thick,  color = red] (4,0) -- (4,9) node [right] {$\Delta_{123}$};
\end{tikzpicture}

&

\begin{tikzpicture}[xscale = \xfactor, yscale = \yfactor]
\draw[ultra thick,  color = red] (0,1) -- (2,2) node [below] {$\Delta_{12}$} -- (4,3);
\draw (4,3) -- (6,2) node [below] {$E_2$} -- (8,1);
\draw[ultra thick,  color = red] (8,0) -- (8,4) node [right] {$\Delta_{123}$} -- (8,8);
\draw (0,7) node [left] {\hphantom{$L_{13}$}} -- (4,7) node [above] {$E_3$} -- (8,7);
\end{tikzpicture}

&

\begin{tikzpicture}[xscale = \xfactor, yscale = \yfactor]
\draw[ultra thick,  color = red] (0,1)  -- (2,2) node[above left] {$\Delta_{12}$} -- (4,3);
\draw[ultra thick,  color = red] (4,3) -- (6,2) node [above right] {$\Delta_{23}$} -- (8,1);
\draw (8,1)  -- (10,2) node[below] {$E_{3}$} -- (12,3);
\draw[ultra thick,  color = red] (12,0) -- (12,8) node [left] {$\Delta_{123}$};
\end{tikzpicture}

\\
\hline

Type 1, $\bDelta = \{ \Delta_{123} \}$
&
Type 3, $\bDelta = \{ \Delta_{12}, \Delta_{123} \}$
&
Type 5, $\bDelta = \{ \Delta_{12}, \Delta_{23}, \Delta_{123} \}$

\\
\hline
\end{tabular}
\medskip
\caption{Configurations of exceptional curves on sextic del Pezzo surfaces}
\label{table:dp6}
\end{table}
\begin{table}[h]\tabcolsep=2em
\nc{\xfactor}{.7}
\nc{\yfactor}{.7}
\begin{tabular}{cccc}
\begin{tikzpicture}[xscale = \xfactor, yscale = \yfactor]
\filldraw[black] (0,0) circle (.2em) -- (1,0) circle (.2em) -- (2,1) circle (.2em) -- (2,2) circle (.2em) -- (1,2) circle (.2em) -- (0,1) circle (.2em) -- (0,0);
\filldraw (1,1) circle (.2em);
\end{tikzpicture}
&
\begin{tikzpicture}[xscale = \xfactor, yscale = \yfactor]
\filldraw[black] (1,0) circle (.2em) -- (2,1) circle (.2em) -- (2,2) circle (.2em) -- (1,2);
\filldraw[red] (1,2) circle (.2em);
\filldraw[black] (1,2) -- (0,2) circle (.2em) -- (0,1) circle (.2em) -- (1,0);
\filldraw[black] (1,1) circle (.2em);
\end{tikzpicture}
&
\begin{tikzpicture}[xscale = \xfactor, yscale = \yfactor]
\filldraw[black] (0,0) circle (.2em) -- (1,0) circle (.2em) -- (2,2) circle (.2em) -- (1,2);
\filldraw[red] (1,2) circle (.2em);
\filldraw[black] (1,2) -- (0,2) circle (.2em) -- (0,1);
\filldraw[red] (0,1) circle (.2em);
\filldraw[black] (0,1) -- (0,0);
\filldraw[black] (1,1) circle (.2em);
\end{tikzpicture}
&
\begin{tikzpicture}[xscale = \xfactor, yscale = \yfactor]
\filldraw[black] (0,0) circle (.2em) -- (3,2) circle (.2em) -- (2,2);
\filldraw[red] (2,2) circle (.2em);
\filldraw[black] (2,2) -- (1,2);
\filldraw[red] (1,2) circle (.2em);
\filldraw[black] (1,2) -- (0,2) circle (.2em) -- (0,1);
\filldraw[red] (0,1) circle (.2em);
\filldraw[black] (0,1) -- (0,0);
\filldraw[black] (1,1) circle (.2em);
\end{tikzpicture}
\\
Type 0 & Type 2 & Type 3 & Type 5
\end{tabular}
\medskip
\caption{Polygons of toric sextic del Pezzo surfaces (types 1 and 4 are not toric)}
\end{table}
In each of these types there is a unique (up to isomorphism) sextic del Pezzo surface.
Moreover, the surfaces of types $0$, $2$, $3$, and $5$ are toric
(in particular, the surface of type~5 is the weighted projective plane~$\P(1,2,3)$, 
see~\cite[Example~5.8]{kawamata2015multi} for an alternative description of its derived category).
The surfaces of type~1 and~4 are not toric. 

\subsection{Resolutions of rational surface singularities}
\label{subsection:resolutions}

In the next section we investigate the derived category of a singular del Pezzo surface~$X$ through its minimal resolution $\tX$.
In this subsection we collect some facts about resolutions of surface singularities we are going to use.

Let $X$ be a normal surface with rational singularities and let $\pi \colon \tX \to X$ be its resolution.
The derived categories of $X$ and $\tX$ are related by the (derived) pushforward functor $\pi_* \colon \bD(\tX) \to \bD(X)$.
The (derived) pullback functor does not preserve boundedness, but is well defined on the bounded from above derived category $\pi^* \colon \bD^-(X) \to \bD^-(\tX)$.
Denote by $\bDelta$ the set of irreducible components of the exceptional divisor of $\pi$;
each of these is a smooth rational curve on $\tX$.

\begin{lemma}
\label{lemma:pi-properties}
Let $X$ be a normal surface with rational singularities and let $\pi \colon \tX \to X$ be its resolution.
The functor $\pi^* \colon \bD^-(X) \to \bD^-(\tX)$ is fully faithful.
The functor $\pi_* \colon \bD^-(\tX) \to \bD^-(X)$ is its right adjoint, 
it preserves boundedness, and there is an isomorphism of functors
\begin{equation}\label{eq:pis-pis-id}
\pi_* \circ \pi^* \cong \id.
\end{equation}
Moreover, 
\begin{equation*}
\Ima \pi^* = {}^\perp\langle \cO_\Delta(-1) \rangle_{\Delta \in \bDelta}
\qquand
\Ker \pi_* = \langle \cO_\Delta(-1) \rangle^\oplus_{\Delta \in \bDelta},
\end{equation*}
where $\langle - \rangle^\oplus$ denotes the minimal triangulated subcategory closed under infinite direct sums defined in $\bD^-$.
\end{lemma}
\begin{proof}
The pullback-pushforward adjunction is standard.
By projection formula we have
\begin{equation*}
\pi_*(\pi^*\cF) \cong \cF \otimes \pi_*\cO_\tX,
\end{equation*}
and since $X$ has rational singularities, the canonical morphism $\cO_X \to \pi_*\cO_\tX$ is an isomorphism, hence~\eqref{eq:pis-pis-id} holds.
By adjunction it follows that $\pi^*$ is fully faithful.
Finally, by~\cite[Lemma~2.1]{bodzenta2015flops} (see also~\cite[Lemma~3.1]{bridgeland} and Lemma~\ref{lemma:spectral} below) and~\cite[Proposition~9.14 and Theorem~9.15]{bodzenta2015flops}
the category $\Ker\pi_*$ is generated by sheaves $\cO_\Delta(-1)$.
The description of $\Ima\pi^*$ follows by adjunction.
\end{proof}

The following Bridgeland's spectral sequence argument is quite useful, so we remind it here.

\begin{lemma}[\protect{\cite[Lemma~3.1]{bridgeland}}]
\label{lemma:spectral}
Let $\pi \colon \tX \to X$ be a proper morphism with fibers of dimension at most~$1$.
Let $\cF$ be a possibly unbounded complex of quasicoherent sheaves on $\tX$ and let $\cH^i(\cF)$ be its cohomology sheaf in degree $i$.
The spectral sequence $R^i\pi_*(\cH^j(\cF)) \Rightarrow \cH^{i+j}(\pi_*(\cF))$ degenerates at the second page, 
and gives for each $i$ an exact sequence
\begin{equation*}
0 \to R^1\pi_*\cH^{i-1}(\cF) \to \cH^i(\pi_*(\cF)) \to R^0\pi_*\cH^i(\cF) \to 0.
\end{equation*}
In particular, if $\cH^i(\pi_*(\cF)) = 0$ for $i \le p$ for some integer $p$, then $\pi_*(\cH^i(\cF)) = 0$ for $i \le p - 1$ and
\begin{equation*}
\pi_*\cF \cong \pi_*(\tau^{\ge p}\cF),
\end{equation*}
where $\tau$ stands for the truncation functor with respect to the canonical filtration.
\end{lemma}
\begin{proof}
The fibers of $\pi$ are at most 1-dimensional, hence $R^{\ge 2}\pi_* = 0$, and the second page of the spectral sequence looks like
\begin{equation*}
\xymatrix@R=1ex{
\dots &
0 &
0 &
0 &
0 &
\dots
\\
\dots &
R^1\pi_*\cH^{i-2}(\cF) & 
R^1\pi_*\cH^{i-1}(\cF) & 
R^1\pi_*\cH^{i}(\cF) & 
R^1\pi_*\cH^{i+1}(\cF) &
\dots
\\
\dots &
R^0\pi_*\cH^{i-2}(\cF) & 
R^0\pi_*\cH^{i-1}(\cF) \ar[uul]_{d_2} & 
R^0\pi_*\cH^{i}(\cF) \ar[uul]_{d_2} & 
R^0\pi_*\cH^{i+1}(\cF) \ar[uul]_{d_2} & 
\dots
}
\end{equation*}
It follows that the spectral sequence degenerates at the second page, and gives the required exact sequences.
The vanishing of $\pi_*(\cH^i(\cF))$ for all $i \le p-1$ follows immediately from the exact sequences, and in its turn implies $\pi_*(\tau^{\le p-1}\cF) = 0$.
Applying the pushforward to the canonical truncation triangle~$\tau^{\le p-1}\cF \to \cF \to \tau^{\ge p}\cF$ we obtain the required isomorphism.
\end{proof}

The following consequences of this observation will be used later.

\begin{corollary}\label{corollary:pis-epi}
The functor $\pi_* \colon \bD(\tX) \to \bD(X)$ is essentially surjective.
\end{corollary}
\begin{proof}
Let $\cF \in \bD(X)$ and assume that $p$ is such that $\tau^{\le p}(\cF) = 0$.
Then $\cF \cong \pi_*(\pi^*\cF) \cong \pi_*(\tau^{\ge p}\pi^*\cF)$, and clearly $\tau^{\ge p}\pi^*\cF \in \bD(\tX)$.
\end{proof}

\subsection{Grothendieck and Serre duality}
\label{subsection:duality}

Let $f \colon X \to Y$ be a proper morphism. 
The {\sf Grothendieck duality} is a bifunctorial isomorphism
\begin{equation}\label{eq:Gd}
\RHom(f_*\cF,\cG) \cong \RHom(\cF,f^!\cG),
\end{equation}
where $f^!$ is the twisted pullback functor (if $\cG$ is perfect, $f^!\cG \cong \cG \otimes \omega^\bullet_{X/Y}$, where $\omega^\bullet_{X/Y} = f^!\cO_Y$ is the relative dualizing complex).
In other words, the twisted pullback functor is right adjoint to the (derived) pushforward.

Grothendieck duality has many consequences.
One of them --- Serre duality for Gorenstein schemes --- will be very useful for our purposes.

\begin{proposition}\label{proposition:serre-gorenstein}
Let $X$ be a projective Gorenstein $\kk$-scheme of dimension $n$.
If either $\cF$ or $\cG$ is a perfect complex, there is a natural Serre duality isomorphism
\begin{equation*}
\Ext^i(\cF,\cG)^\vee \cong \Ext^{n-i}(\cG,\cF \otimes \omega_X).
\end{equation*}
\end{proposition}
\begin{proof}
Let $f \colon X \to \Spec(\kk)$ be the structure morphism.
If $\cF$ is a locally free sheaf, then we have $\Ext^i(\cF,\cG) \cong H^i(X,\cF^\vee \otimes \cG)$, and this is a cohomology group of $f_*(\cF^\vee \otimes \cG)$.
By Grothendieck duality
\begin{equation*}
\RHom(f_*(\cF^\vee \otimes \cG),\kk) \cong \RHom(\cF^\vee \otimes \cG,f^!(\kk)).
\end{equation*}
Since $X$ is Gorenstein, $f^!(\kk) \cong \omega_X[n]$, hence the right hand side equals $\RHom(\cF^\vee \otimes \cG, \omega_X[n])$.
Since~$\cF$ is locally free, this can be rewritten as $\RHom(\cG,\cF \otimes \omega_X[n])$.
Computing the cohomology groups in degree $-i$, we deduce the required duality isomorphism.

For arbitrary perfect $\cF$ the Serre duality follows by using the stupid filtration.

Finally, when $\cG$ is perfect, we replace $\cF$ by $\cG$, $\cG$ by $\cF \otimes \omega_X$, and $i$ by $n - i$, and deduce the required isomorphism from the previous case.
\end{proof}

Let us also discuss a contravariant duality functor or a projective $\kk$-scheme~$X$. 
It follows from sheafified Grothendieck duality that 
\begin{equation*}
\RCHom(-,\omega^\bullet_{X/\kk}) \colon \bD(X)^{\mathrm{opp}} \xrightarrow{\ \sim\ } \bD(X).
\end{equation*}
is an equivalence of categories.
In case when the scheme $X$ is Gorenstein, the dualizing complex $\omega^\bullet_{X/\kk}$ 
is a shift of the canonical line bundle, $\omega^\bullet_{X/\kk} \cong \omega_X[\dim X]$, 
and it follows that the usual duality functor
\begin{equation*}
\cF \mapsto \cF^\vee := \RCHom(\cF,\cO_X) \cong \RCHom(\cF,\omega_X^\bullet) \otimes \omega_X^{-1}[-\dim X]
\end{equation*}
is also an equivalence of categories $\bD(X)^{\mathrm{opp}} \xrightarrow{\ \sim\ } \bD(X)$.

\section{Derived category of a single sextic del Pezzo surface}
\label{section:dp6-single}

Let $X$ be a sextic du Val del Pezzo surface (Definition~\ref{definition:dp6}) over an algebraically closed field $\kk$, and let~$\pi \colon \tX \to X$ be its minimal resolution of singularities.
We use freely notation introduced in Section~\ref{subsection:dp6}.

\subsection{Derived category of the resolution}\label{subsection:derived-tx}

We start by describing the derived category of $\tX$.
In the case when $X$ is smooth (and so $\tX = X$), the following result has been proved in~\cite[Proposition~4.2(3)]{karpov-nogin}
and in~\cite[proof of Proposition~9.1]{auel2015semiorthogonal}.
We leave it to the readers to check that the same arguments work for any du Val del Pezzo surface of degree~6.

\begin{proposition}
\label{proposition:tx-excol}
Let $X$ be a sextic du Val del Pezzo surface over an algebraically closed field $\kk$ 
and let~$\pi \colon \tX \to X$ be its minimal resolution of singularities. 
There is a semiorthogonal decomposition
\begin{equation}
\label{eq:dbtx}
\bD(\tX) = \langle \tcA_1, \tcA_2, \tcA_3  \rangle,
\end{equation}
whose components are generated by the following exceptional collections of line bundles
\begin{equation}
\begin{aligned}
\label{eq:tx-excol}
\tcA_1 &= \langle \cO_\tX \rangle,\\
\tcA_2 &= \langle \cO_\tX(h - e_1), \cO_\tX(h - e_2), \cO_\tX(h - e_3) \rangle,\\
\tcA_3 &= \langle \cO_\tX(h), \cO_\tX(2h - e_1 - e_2 - e_3) \rangle. 
\end{aligned}
\end{equation} 
\end{proposition}

%

If $X$ is smooth (hence $\tX = X$) the exceptional line bundles in each of the components $\tcA_i$ of~\eqref{eq:tx-excol} are pairwise orthogonal.
However, for singular $X$ this is no longer true.
We describe the structure of the categories~$\tcA_i$ below, but before that we observe a self-duality property of~\eqref{eq:dbtx}.

As it is explained in Section~\ref{subsection:duality} the derived dualization functor $\cF \mapsto \cF^\vee$ 
provides an anti-autoequiva\-lence of~$\bD(\tX)$. 
When applied to~\eqref{eq:dbtx} it produces another semiorthogonal decomposition
\begin{equation}\label{eq:dbtx-dual}
\bD(\tX) = \langle \tcA_3^\vee, \tcA_2^\vee, \tcA_1^\vee \rangle.
\end{equation} 
It turns out that it is also the right mutation-dual of~\eqref{eq:dbtx}, i.e., is obtained from~\eqref{eq:dbtx} by a standard sequence of mutations.
Below we denote by $\LL_\cA$ the left mutation functor through an admissible subcategory $\cA$.

\begin{lemma}\label{lemma:dual-sod-tx}
The semiorthogonal decomposition~\eqref{eq:dbtx-dual} is right mutation-dual to~\eqref{eq:dbtx}, i.e., 
\begin{equation*}
\tcA_1^\vee = \tcA_1,
\qquad
\tcA_2^\vee = \LL_{\tcA_1}(\tcA_2),
\quad\text{and}\quad 
\tcA_3^\vee = \LL_{\tcA_1}(\LL_{\tcA_2}(\tcA_3)) = \tcA_3 \otimes \omega_\tX.
\end{equation*}
\end{lemma}
\begin{proof}
The claim is trivial for the first component, since $\tcA_1^\vee  = \langle \cO_\tX^\vee  \rangle = \langle \cO_\tX \rangle = \tcA_1$,
and is easy for the last component, since $\LL_{\tcA_1}(\LL_{\tcA_2}(\tcA_3)) = \tcA_3  \otimes \omega_\tX$ and
\begin{equation}\label{eq:ce3-dual}
\begin{aligned}
& \cO_\tX(h) \otimes \omega_\tX && \cong \cO_\tX(-2h + e_1 + e_2 + e_3) && \cong \cO_\tX(2h - e_1 - e_2 - e_3)^\vee,\\
& \cO_\tX(2h - e_1 - e_2 - e_3) \otimes \omega_\tX && \cong \cO_\tX(-h) && \cong \cO_\tX(h)^\vee.
\end{aligned}
\end{equation} 
Finally, for the second component we have 
$\LL_{\tcA_1}(\tcA_2) = \tcA_1^\perp \cap \tcA_3^\perp = (\tcA_1^\vee)^\perp \cap  {}^\perp(\tcA_3 \otimes \omega_X) = \tcA_2^\vee$
by Serre duality.
\end{proof}

The structure of the components $\tcA_i$ of the decomposition~\eqref{eq:dbtx} depends on the type of $X$.
We explain this dependence in Proposition~\ref{proposition:tcai-auslander}.
For $m = 2$ and $m = 3$ we denote by $\tR_m$ the Auslander algebra defined by~\eqref{eq:trm}, and refer to Appendix~\ref{appendix:auslander} for its basic properties, 
especially note the definition~\eqref{eq:projective-resolution-ei} of the {\sf standard exceptional modules} and Proposition~\ref{proposition:rm-equivalence}.

\begin{proposition}
\label{proposition:tcai-auslander}
The components $\tcA_i$ of~\eqref{eq:dbtx} are equivalent to products of derived categories of Auslander algebras as indicated in the next table:
\begin{equation*}\def\arraystretch{1.7}%
\begin{array}{|c|c|c|c|}
\hline
\text{Type of $X$} & \tcA_1 & \tcA_2 & \tcA_3 \\
\hline 
0 & \bD(\kk) & \bD(\kk) \times \bD(\kk) \times \bD(\kk) & \bD(\kk) \times \bD(\kk) \\
\hline
1 & \bD(\kk) & \bD(\kk) \times \bD(\kk) \times \bD(\kk) & \bD(\tR_2) \\
\hline
2 & \bD(\kk) & \bD(\tR_2) \times \bD(\kk) & \bD(\kk)  \times \bD(\kk) \\
\hline
3 & \bD(\kk) & \bD(\tR_2) \times \bD(\kk) & \bD(\tR_2) \\
\hline
4 & \bD(\kk) & \bD(\tR_3) & \bD(\kk)  \times \bD(\kk) \\
\hline
5 & \bD(\kk) & \bD(\tR_3) & \bD(\tR_2) \\
\hline
\end{array}
\end{equation*}
This equivalence takes the exceptional line bundles in~\eqref{eq:tx-excol} 
to the standard exceptional modules over the corresponding Auslander algebra.
\end{proposition}
\begin{proof}
The component $\tcA_1$ is generated by a single exceptional object, hence is equivalent to $\bD(\kk)$.
So, in view of Proposition~\ref{proposition:rm-equivalence} to prove the proposition it is enough to compute $\Ext$-spaces 
between the exceptional line bundles forming the components $\tcA_2$ and $\tcA_3$ of the semiorthogonal decomposition~\eqref{eq:dbtx}.

First of all, we have
\begin{equation*}
\Ext^p(\cO_\tX(h-e_i), \cO_\tX(h-e_j)) \cong H^p(\tX, \cO_\tX(e_i - e_j)).
\end{equation*}
Assuming $i \ne j$ and using exact sequences (note that $e_i\cdot e_j = 0$ and $e_i^2 = -1$ by Lemma~\ref{lemma:pic-tx})
\begin{equation*}
0 \to \cO_\tX(e_i - e_j) \to \cO_\tX(e_i) \to \cO_{E_j} \to 0,
\qquad\qquad
0 \to \cO_\tX \to \cO_\tX(e_i) \to \cO_{E_i}(-1) \to 0,
\end{equation*}
we obtain an exact sequence
\begin{equation*}
0 \to H^0(\tX, \cO_\tX(e_i - e_j)) \to \kk \to \kk \to H^1(\tX, \cO_\tX(e_i - e_j)) \to 0,
\end{equation*}
where the middle map is given by the restriction to $E_j$ of the equation of $E_i$.
This restriction vanishes if and only if $E_j$ is a component of $E_i$ --- in this case we deduce that $\Ext^\bullet(\cO_\tX(h-e_i), \cO_\tX(h-e_j)) \cong \kk \oplus \kk[-1]$,
and otherwise $\Ext^\bullet(\cO_\tX(h-e_i), \cO_\tX(h-e_j)) = 0$.
By Proposition~\ref{proposition:rm-equivalence} this gives the required description of $\tcA_2$ in types from~0 to 3.
In the last two types (4 and 5) it remains to check that multiplication map
\begin{equation}\label{eq:multiplication-ext-tx}
\Ext^p(\cO_\tX(h-e_1),\cO_\tX(h-e_2)) \otimes \Ext^q(\cO_\tX(h-e_2),\cO_\tX(h-e_3)) \to \Ext^{p+q}(\cO_\tX(h-e_1),\cO_\tX(h-e_3))
\end{equation}
is an isomorphism when $p = 0$ or $q = 0$.
For this consider exact sequences
\begin{equation}\label{eq:e12}
0 \to \cO_\tX(h - e_1) \xrightarrow{\ \Delta_{12}\ } \cO_\tX(h - e_2) \to \cO_{\Delta_{12}}(-1) \to 0,
\end{equation}
\begin{equation}\label{eq:e23}
0 \to \cO_\tX(h - e_2) \xrightarrow{\ \Delta_{23}\ } \cO_\tX(h - e_3) \to \cO_{\Delta_{23}}(-1) \to 0,
\end{equation}
where $\Delta_{12} = E_1 - E_2$ and $\Delta_{23} = E_2 - E_3$.
Using Lemma~\ref{lemma:pic-tx}, we compute
\begin{equation*}
\Ext^\bullet(\cO_\tX(h-e_1),\cO_{\Delta_{23}}(-1)) \cong
H^\bullet(\Delta_{23},\cO_\tX(e_1 - h)\vert_{\Delta_{23}} \otimes \cO_{\Delta_{23}}(-1)) =
H^\bullet(\Delta_{23},\cO_{\Delta_{23}}(-1)) = 0.
\end{equation*}
Applying the functor $\Ext^\bullet(\cO_\tX(h-e_1),-)$ to~\eqref{eq:e23}, we deduce that~\eqref{eq:multiplication-ext-tx} is an isomorphism for~$q = 0$. 
Similarly, by Serre duality $\Ext^\bullet(\cO_{\Delta_{12}}(-1),\cO_\tX(h-e_3))$ is dual to $\Ext^\bullet(\cO_\tX(e_1 + e_2 - 2h),\cO_{\Delta_{12}}(-1))$, and
a computation similar to the above shows it is zero.
Applying the functor $\Ext^\bullet(-,\cO_\tX(h-e_3))$ to~\eqref{eq:e12}, we deduce that~~\eqref{eq:multiplication-ext-tx} is an isomorphism for $p = 0$. 

To describe $\tcA_3$ we only need to compute 
\begin{equation*}
\Ext^p(\cO_\tX(h), \cO_\tX(2h - e_1 - e_2 - e_3)) \cong H^p(\tX, \cO_\tX(h - e_1 - e_2 - e_3)).
\end{equation*}
Using exact sequences
\begin{align*}
0 \to \cO_\tX(h - e_1) \to \cO_\tX(h) &\to \cO_{E_1} \to 0,\\
0 \to \cO_\tX(h - e_1 - e_2) \to \cO_\tX(h - e_1) &\to \cO_{E_2} \to 0,\\
0 \to \cO_\tX(h - e_1 - e_2 - e_3) \to \cO_\tX(h - e_1 - e_2) &\to \cO_{E_3} \to 0,
\end{align*}
we conclude that $H^\bullet(\tX, \cO_\tX(h - e_1 - e_2 - e_3)) = 0$ 
if the three centers of the blowups are not contained on a line (i.e., in types~0, 2, and 4), and 
$H^0(\tX, \cO_\tX(h - e_1 - e_2 - e_3)) = H^1(\tX, \cO_\tX(h - e_1 - e_2 - e_3)) = \kk$ otherwise.
By Proposition~\ref{proposition:rm-equivalence} this describes $\tcA_3$.
\end{proof}

The nontrivial morphisms between the line bundles $\cO_\tX(h-e_1)$, $\cO_\tX(h-e_2)$, and $\cO_\tX(h-e_3)$ when exist are realized by exact sequences~\eqref{eq:e12} and~\eqref{eq:e23}.
Analogously, the nontrivial morphism between the line bundles $\cO_\tX(h)$ and $\cO_\tX(2h - e_1 - e_2 - e_3)$ is realized by the exact sequence
\begin{equation}\label{eq:e123}
0 \to \cO_\tX(h) \xrightarrow{\ \Delta_{123}\ } \cO_\tX(2h - e_1 - e_2 - e_3) \to \cO_{\Delta_{123}}(-1) \to 0.
\end{equation}
This shows that each category $\tcA_i$ is a product of categories of the form described in~\cite[Theorem~2.5]{hille2017tilting},
and in fact, the Auslander algebras $\tR_m$ appearing in Proposition~\ref{proposition:tcai-auslander},
are nothing but the particular forms of the algebras from \emph{loc.\ cit.}, cf.~\cite[\S3.2]{hille2017tilting}.

\begin{remark}
\label{remark:delta-simples}
Comparing exact sequences~\eqref{eq:resolution-simple} of standard exceptional modules over an Auslander algebra with the exact sequences~\eqref{eq:e12}, \eqref{eq:e23} and~\eqref{eq:e123},
and taking into account that the equivalence of Proposition~\ref{proposition:tcai-auslander} takes the exceptional line bundles generating $\tcA_i$ 
to the standard exceptional modules over the corresponding Auslander algebras, we conclude that the sheaves~$\cO_\Delta(-1)$ go to the corresponding simple modules.
For instance, for a surface of type 5, when we have the maximal number of $(-2)$-curves on $\tX$, and when $\tcA_2 \cong \bD(\tR_3)$, $\tcA_3 \cong \bD(\tR_2)$, 
the sheaves $\cO_{\Delta_{12}}(-1)$ and~$\cO_{\Delta_{23}}(-1)$ go to $S_1$ and $S_2$ in $\bD(\tR_3)$, and the sheaf $\cO_{\Delta_{123}}(-1)$ goes to $S_1$ in $\bD(\tR_2)$.
\end{remark}

\subsection{Semiorthogonal decomposition for a sextic del Pezzo surface}
\label{subsection:derived-x}

Now we apply the above computations to describe the derived category of a sextic du Val del Pezzo surface $X$.
The main result of this section is the next theorem.

\begin{theorem}
\label{theorem:dbx}
Let $X$ be a sextic du Val del Pezzo surface over an algebraically closed field~$\kk$, and let $\pi \colon \tX \to X$ be its minimal resolution of singularities.
Then there is a unique semiorthogonal decomposition
\begin{equation}\label{eq:dbx}
\bD(X) = \langle \cA_1, \cA_2, \cA_3 \rangle,
\end{equation}
such that $\pi_*(\tcA_i) = \cA_i$ and $\pi^*(\cA_i \cap \bD^\perf(X)) \subset \tcA_i$, where $\tcA_i$ are the components of~\eqref{eq:dbtx}.
The components~$\cA_i$ are admissible with projection functors of finite cohomological amplitude, 
and are equivalent to products of derived categories of finite-dimensional algebras as indicated 
in the next table:
\begin{equation*}\def\arraystretch{1.2}%
\begin{array}{|c|c|c|c|}
\hline
\text{Type of $X$} & \cA_1 & \cA_2 & \cA_3 \\
\hline 
0 & \bD(\kk) & \bD(\kk) \times \bD(\kk) \times \bD(\kk) & \bD(\kk) \times \bD(\kk) \\
\hline
1 & \bD(\kk) & \bD(\kk) \times \bD(\kk) \times \bD(\kk) & \bD(\kk[t]/t^2) \\
\hline
2 & \bD(\kk) & \bD(\kk[t]/t^2) \times \bD(\kk) & \bD(\kk)  \times \bD(\kk) \\
\hline
3 & \bD(\kk) & \bD(\kk[t]/t^2) \times \bD(\kk) & \bD(\kk[t]/t^2) \\
\hline
4 & \bD(\kk) & \bD(\kk[t]/t^3) & \bD(\kk)  \times \bD(\kk) \\
\hline
5 & \bD(\kk) & \bD(\kk[t]/t^3) & \bD(\kk[t]/t^2) \\
\hline
\end{array}
\end{equation*}
The categories $\cA_i^\perf := \cA_i \cap \bD^\perf(X)$ form a semiorthogonal decomposition of the perfect category
\begin{equation*}
\bD^\perf(X) = \langle \cA_1^\perf, \cA_2^\perf, \cA_3^\perf \rangle.
\end{equation*}
\end{theorem}

The proof takes the rest of this subsection.
We start with some preparations.

Consider the semiorthogonal decomposition~\eqref{eq:dbtx}.
Since $\tX$ is smooth, every component $\tcA_i$ of $\bD(\tX)$ is admissible, hence~\eqref{eq:dbtx} is a strong semiorthogonal decomposition 
in the sense of~\cite[Definition~2.6]{kuznetsov2011base}.
Therefore, by~\cite[Proposition~4.2]{kuznetsov2011base} it extends to a semiorthogonal decomposition of the bounded above derived category
\begin{equation}\label{eq:dbtx-}
\bD^-(\tX) = \langle \tcA_1^-, \tcA_2^-, \tcA_3^- \rangle.
\end{equation}
The next lemma describes the intersections of its components with the kernel category of the pushforward functor $\pi_*$.
Recall that $\bDelta$ denotes the set of all $(-2)$-curves defined on $\tX$. 
We denote by $\bDelta_2 \subset \bDelta$ the subset formed by those of the curves $\Delta_{12}$ and $\Delta_{23}$ 
that are defined on $\tX$ and by $\bDelta_3 \subset \bDelta$ its complement.

\begin{lemma}
\label{lemma:ker-pis}
We have 
\begin{equation*}
\tcA_1^- \cap \Ker \pi_* = 0,
\qquad 
\tcA_2^- \cap \Ker \pi_* = \langle \cO_{\Delta}(-1) \rangle^\oplus_{\Delta \in \bDelta_2},
\quand
\tcA_3^- \cap \Ker \pi_* = \langle \cO_{\Delta}(-1) \rangle^\oplus_{\Delta \in \bDelta_3}.
\end{equation*}
Moreover, for any $\cF \in \Ker\pi_* \subset \bD^-(\tX)$ there is a canonical direct sum decomposition 
\begin{equation*}
\cF = \cF_2 \oplus \cF_3
\end{equation*}
with $\cF_j \in \tcA_j^- \cap \Ker \pi_*$.
\end{lemma}
\begin{proof}
By Lemma~\ref{lemma:pi-properties}
an object $\cF \in \bD^-(\tX)$ is in $\Ker\pi_*$ if and only if every its cohomology sheaf~$\cH^j(\cF)$ 
is an iterated extension of sheaves $\cO_\Delta(-1)$
where $\Delta$ run over the set $\bDelta$ of all $(-2)$-curves on~$\tX$.
Note that $\cO_\Delta(-1) \in \tcA_2$ for $\Delta \in \bDelta_2$ by~\eqref{eq:e12} and~\eqref{eq:e23}, while $\cO_\Delta(-1) \in \tcA_3$ for $\Delta \in \bDelta_3$ by~\eqref{eq:e123}.
Moreover, the subcategories generated by the sheaves $\cO_\Delta(-1)$ with $\Delta \in \bDelta_2$ and $\Delta \in \bDelta_3$ are completely orthogonal, since the supports of these sheaves do not intersect.
This last observation shows that any $\cF \in \Ker\pi_*$ decomposes as $\cF_2 \oplus \cF_3$ with the required properties (just take $\cF_2$ and $\cF_3$ to be the components of $\cF$ 
supported on the union of the curves $\Delta$ from $\bDelta_2$ and $\bDelta_3$ respectively).
\end{proof}

Another useful observation is the following.

\begin{lemma}\label{lemma:ker-pis-orthogonality}
We have $\Hom(\tcA_i^-,\tcA_j^- \cap \Ker \pi_*) = 0$ for any $i \ne j$.
\end{lemma}
\begin{proof}
Since $\tcA_i^-$ is generated by an exceptional collection of line bundles and $\tcA_j^- \cap \Ker \pi_*$ is generated by sheaves $\cO_\Delta(-1)$ for $\Delta \in \bDelta_j$,
it is enough to check that any of the line bundles generating $\tcA_i^-$ restricts trivially to any curve $\Delta$ from $\bDelta_j$.

For $i = 1$ this is evident; for $i = 2$ we have
$(h-e_k)\cdot \Delta_{123} = (h - e_k)(h - e_1 - e_2 - e_3) = 1 - 1 = 0$,
and for $i = 3$ we have 
$h\cdot (e_k - e_l) = 0$
and
$(2h - e_1 - e_2 - e_3) \cdot (e_k - e_l) = 1 - 1 = 0$.
\end{proof}

Now we can show that~\eqref{eq:dbtx-} induces a semiorthogonal decomposition of~$\bD^-(X)$.
Denote by $\tilde\alpha_i$ the projection functors of the decomposition~\eqref{eq:dbtx-}.
By~\cite[Proposition~4.2]{kuznetsov2011base} the projection functors of~\eqref{eq:dbtx} are given 
by the restrictions of $\tilde\alpha_i$ to $\bD(\tX)$.

\begin{proposition}\label{proposition:sod-dbx-}
The subcategories $\cA_i^- = \{ \cF \in \bD^-(X) \mid  \pi^*(\cF) \in \tcA_i^- \} \subset \bD^-(X)$ form a semiorthogonal decomposition
\begin{equation*}
\bD^-(X) = \langle \cA_1^-, \cA_2^-, \cA_3^- \rangle
\end{equation*}
with projection functors given by
\begin{equation}
\label{eq:alpha}
\alpha_i = \pi_* \circ \tilde\alpha_i \circ \pi^*.
\end{equation}
Moreover, we have
\begin{equation}\label{eq:pis-ai}
\pi_*(\tcA_i^-) = \cA_i^-.
\end{equation}
\end{proposition}
\begin{proof}
By Lemma~\ref{lemma:pi-properties} the functor $\pi^* \colon \bD^-(X) \to \bD^-(\tX)$ is fully faithful, hence the subcategories $\cA_i^-$ are semiorthogonal.

Let us prove~\eqref{eq:pis-ai}.
For this take any object $\cF \in \tcA_i^-$ and consider the standard triangle
\begin{equation*}
\pi^*(\pi_*\cF) \to \cF \to \cG.
\end{equation*}
Then, of course, $\cG \in \Ker\pi_*$.
By Lemma~\ref{lemma:ker-pis} we have $\cG = \cG_2 \oplus \cG_3$ with~$\cG_j \in \tcA_j^- \cap \Ker \pi_*$.
If~$j \ne i$ then~$\Ext^\bullet(\cF,\cG_j) = 0$ by Lemma~\ref{lemma:ker-pis-orthogonality} and $\Ext^\bullet(\pi^*(\pi_*\cF),\cG_j) = 0$ since $\cG_j \in \Ker  \pi_*$.
It follows from the above triangle that $\Hom(\cG,\cG_j) = 0$, hence $\cG_j = 0$ since it is a direct summand of $\cG$.
Therefore we have $\cG \in \tcA_i^-$, hence $\pi^*(\pi_*\cF) \in \tcA_i^-$, hence $\pi_*\cF \in \cA_i^-$ by definition of the latter.
This proves an inclusion $\pi_*(\tcA_i^-) \subset \cA_i^-$.
The other inclusion follows from~\eqref{eq:pis-pis-id}.

Now let us decompose any $\cF \in \bD^-(X)$.
For this take $\tcF := \pi^*(\cF) \in \bD^-(\tX)$ and consider its decomposition with respect to~\eqref{eq:dbtx-}.
It is given by a chain of morphisms 
\begin{equation*}
0 = \tcF_3 \rightarrow \tcF_2 \rightarrow \tcF_1 \rightarrow \tcF_0 = \tcF,
\end{equation*}
whose cones are $\tilde\alpha_i(\tcF) \in \tcA_i^-$.
Pushing it forward to $X$, we obtain a chain of morphisms 
\begin{equation*}
0 = \pi_*(\tcF_3) \rightarrow \pi_*(\tcF_2) \rightarrow \pi_*(\tcF_1) \rightarrow \pi_*(\tcF_0) = \pi_*(\tcF) \cong \cF,
\end{equation*}
whose cones are $\pi_*(\tilde\alpha_i(\tcF)) \cong \pi_*(\tilde\alpha_i(\pi^*\cF)) \in \pi_*(\tcA_i^-) = \cA_i^-$.
This proves the semiorthogonal decomposition and shows that its projection functors are given by~\eqref{eq:alpha}.
\end{proof}

And now we can construct the required semiorthogonal decomposition of~$\bD(X)$.

\begin{proposition}
\label{proposition:sod-dbx}
The subcategories 
\begin{equation*}
\cA_i := \cA_i^- \cap \bD(X) = \{ \cF \in \bD(X) \mid \pi^*(\cF) \in \tcA_i^- \} \subset \bD(X)
\end{equation*}
provide a semiorthogonal decomposition~\eqref{eq:dbx}.
Its projection functors $\alpha_i$ are given by~\eqref{eq:alpha}; they preserve boundedness and have finite cohomological amplitude.
Finally, $\cA_i = \pi_*(\tcA_i)$.
\end{proposition}
\begin{proof}
For the first claim it is enough to check that the projection functors $\alpha_i$ preserve boundedness.
Take any $\cF \in \bD^{[a,b]}(X)$ and consider $\pi^*(\cF) \in \bD^{(-\infty,b]}(\tX)$.
By projection formula $\pi_*(\pi^*(\cF)) \cong \cF$, hence by Lemma~\ref{lemma:spectral} we have $\tau^{\le a-1}(\pi^*(\cF)) \in \Ker \pi_*$.
Consider the triangle
\begin{equation*}
\tilde\alpha_i(\tau^{\le a-1}(\pi^*(\cF))) \to \tilde\alpha_i(\pi^*(\cF)) \to \tilde\alpha_i(\tau^{\ge a}(\pi^*(\cF)))
\end{equation*}
obtained by applying the projection functor $\tilde\alpha_i$ to the canonical truncation triangle.
By Lemma~\ref{lemma:ker-pis} the functor $\tilde\alpha_i$ preserves~$\Ker\pi_*$, hence the first term of the triangle is in $\Ker \pi_*$.
Therefore, applying the pushforward we obtain an isomorphism 
\begin{equation*}
\alpha_i(\cF) = \pi_*(\tilde\alpha_i(\pi^*(\cF))) \cong \pi_*(\tilde\alpha_i(\tau^{\ge a}(\pi^*(\cF))).
\end{equation*}
So, it remains to note that $\tau^{\ge a}(\pi^*(\cF)) \in \bD^{[a,b]}(\tX)$,
hence $\pi_*(\tilde\alpha_i(\tau^{\ge a}(\pi^*(\cF)))$ is bounded, since both $\tilde\alpha_i$ and $\pi_*$ preserve boundedness.
Moreover, if the cohomological amplitude of $\tilde\alpha_i$ is $(p,q)$ 
(it is finite since~$\tX$ is smooth, see~\cite[Proposition~2.5]{kuznetsov2008resolution}), then
$\pi_*(\tilde\alpha_i(\tau^{\ge a}(\pi^*(\cF))) \in \bD^{[a+p,b+q+1]}(X)$.
In particular, $\alpha_i$ has finite cohomological amplitude.

Let us prove the last claim.
By~\eqref{eq:pis-ai} we have $\pi_*(\tcA_i) \subset \cA_i^-$, and since $\pi_*$ preserves boundedness, we have $\pi_*(\tcA_i) \subset \cA_i$.
To check that this inclusion is an equality, take any $\cF \in \cA_i$.
By Corollary~\ref{corollary:pis-epi} there exists $\tcF \in \bD(\tX)$ such that $\cF \cong \pi_*(\tcF)$.
Let $\cG$ be the cone of the natural morphism $\pi^*\cF \to \tcF$.
Then $\cG \in \Ker  \pi_*$.
Moreover, $\tilde\alpha_i(\cG) \in \Ker \pi_*$ by Lemma~\ref{lemma:ker-pis}, hence applying the functor $\pi_* \circ \tilde\alpha_i$ 
to the distinguished triangle $\pi^*\cF \to \tcF \to \cG$, we deduce an isomorphism $\cF \cong \alpha_i(\cF) \cong \pi_*(\tilde\alpha_i(\tcF))$,
and it remains to note that $\tilde\alpha_i(\tcF) \in \tcA_i$.
\end{proof}

Next, we identify the categories $\cA_1$, $\cA_2$, and $\cA_3$ constructed in Proposition~\ref{proposition:sod-dbx} 
with the corresponding products of $\bD(\kk[t]/t^m)$.
By Proposition~\ref{proposition:tcai-auslander} each category $\tcA_i$ is equivalent to a product of derived categories $\bD(\tR_m)$ of Auslander algebras $\tR_m$.
Take one of these and denote by $\tilde\gamma \colon \bD(\tR_m) \to \bD(\tX)$ its embedding functor.
Let 
\begin{equation*}
\pi_{m*} \colon \bD^-(\tR_m) \to \bD^-(\kk[t]/t^m)
\qquad \text{and} \qquad
\pi_m^* \colon \bD^-(\kk[t]/t^m) \to \bD^-(\tR_m)
\end{equation*}
be the functors described in Appendix~\ref{appendix:auslander}, see equation~\eqref{eq:pms}.
Note that $\pi_m^*$ is left adjoint to~$\pi_{m*}$.

\begin{proposition}
\label{proposition:gamma}
The functor 
\begin{equation}\label{def:gamma}
\gamma := \pi_* \circ \tilde\gamma \circ \pi_m^* \colon \bD^-(\kk[t]/t^m) \to \bD^-(X).
\end{equation}
is fully faithful and preserves boundedness.
Moreover, the diagrams
\begin{equation}
\label{eq:gamma-tgamma}
\vcenter{\xymatrix@C=5em{
\bD^-(\tR_m) \ar[r]^-{\tilde\gamma} \ar[d]_{\pi_{m*}} &
\bD^-(\tX) \ar[d]^{\pi_{*}} 
\\
\bD^-(\kk[t]/t^m) \ar[r]^{\gamma} &
\bD^-(X)
}}
\qquad\text{and}\qquad
\vcenter{\xymatrix@C=5em{
\bD^-(\tR_m) \ar[r]^-{\tilde\gamma} &
\bD^-(\tX)  
\\
\bD^-(\kk[t]/t^m) \ar[r]^{\gamma} \ar[u]^{\pi_m^{*}} &
\bD^-(X) \ar[u]_{\pi^{*}}
}}
\end{equation}
are both commutative.
\end{proposition}
\begin{proof}
By definition $\gamma \circ \pi_{m*} = \pi_* \circ \tilde\gamma \circ \pi_m^* \circ \pi_{m*}$, so for commutativity of the first diagram it is enough to check that for any $M \in \bD^-(\tR_m)$ 
the cone of the canonical morphism $\pi_m^*(\pi_{m*}M) \to M$ is killed by the functor $\pi_* \circ \tilde\gamma$.
But this cone, by Proposition~\ref{proposition:auslander-localization} is contained in the subcategory generated by the simple modules $S_1,\dots,S_{m-1}$ over $\tR_m$.
By Remark~\ref{remark:delta-simples} the functor $\tilde\gamma$ takes these modules to sheaves $\cO_\Delta(-1)$ 
for appropriate $(-2)$-curves $\Delta$ on $\tX$, and the functor $\pi_*$ kills every $\cO_\Delta(-1)$ by Lemma~\ref{lemma:pi-properties}.

For commutativity of the second diagram note that by Proposition~\ref{proposition:auslander-localization} the image of the functor~$\pi_m^*$ 
is the left orthogonal ${}^\perp\langle S_1,\dots, S_{m-1} \rangle \subset \bD^-(\tR_m)$ of the subcategory generated by the simple modules.
Therefore, using Lemma~\ref{lemma:ker-pis-orthogonality} and Remark~\ref{remark:delta-simples} we deduce that the image of $\tilde\gamma \circ \pi_m^*$ is contained 
in the left orthogonal ${}^\perp \langle \cO_\Delta(-1) \rangle \subset \bD^-(\tX)$ of the subcategory generated by all $(-2)$-curves $\Delta$.
Hence, by Lemma~\ref{lemma:pi-properties} it is contained in the image of $\pi^* \colon \bD^-(X) \to \bD^-(\tX)$.
Thus, the functor $\pi^* \circ \pi_*$ is identical on the image of $\tilde\gamma \circ \pi_m^*$, hence 
$\tilde\gamma \circ \pi_m^* \cong 
\pi^* \circ \pi_* \circ \tilde\gamma \circ \pi_m^* \cong 
\pi^* \circ \gamma$,
and so the second diagram commutes.

Let us show that $\gamma$ is fully faithful.
Indeed, by commutativity of the second diagram in~\eqref{eq:gamma-tgamma}, this follows from full faithfulness of the functors $\pi_m^*$ (Proposition~\ref{proposition:auslander-localization}), 
$\tilde\gamma$ (Proposition~\ref{proposition:tcai-auslander}), and $\pi^*$ (Lemma~\ref{lemma:pi-properties}).

Finally, let us show that $\gamma$ preserves boundedness.
Indeed, take any $N \in \bD(\kk[t]/t^m)$.
By Proposition~\ref{proposition:auslander-localization} there exists $M \in \bD(\tR_m)$ such that $N \cong \pi_{m*}(M)$.
Then $\gamma(N) = \gamma(\pi_{m*}(M))$ and by commutativity of the first diagram this is the same as $\pi_*(\tilde\gamma(M))$.
We know that $\tilde\gamma(M)$ is bounded by Proposition~\ref{proposition:tcai-auslander}, hence so is $\pi_*(\tilde\gamma(M)) = \gamma(N)$.
\end{proof}

Now we are ready to describe the components $\cA_i$ of~\eqref{eq:dbx}.
Let $\tcA_i = \bD(\tR_{m_1}) \times \dots \times \bD(\tR_{m_s})$ be the corresponding component of~\eqref{eq:dbtx}.
Let $\tilde\gamma_j \colon \bD(\tR_{m_j}) \to \bD(\tX)$ be the corresponding embedding functors and 
let $\gamma_j \colon \bD(\kk[t]/t^{m_j}) \to \bD(X)$ be the embeddings constructed in Proposition~\ref{proposition:gamma}.

\begin{proposition}
\label{proposition:sod-cai}
The fully faithful functors $\gamma_1,\dots,\gamma_s$ induce a completely orthogonal decomposition 
\begin{equation*}
\cA_i = \bD(\kk[t]/t^{m_1}) \times \dots \times \bD(\kk[t]/t^{m_s}).
\end{equation*}
\end{proposition}
\begin{proof}
By definition $\gamma_j = \pi_* \circ \tilde\gamma_j \circ \pi_{m_j}^*$ and its image is contained in $\pi_*(\tcA_i^-) = \cA_i^-$.
Moreover, since $\gamma_j$ preserves boundedness, it actually is in $\cA_i$.

To see that the images of $\gamma_j$ are orthogonal, it is enough to check orthogonality of the images of functors $\pi^* \circ \gamma_j = \tilde\gamma_j \circ \pi_{m_j}^*$,
which follows immediately from Proposition~\ref{proposition:tcai-auslander}.

Finally, let us show the generation.
Assume $\cF \in \cA_i$ and let $\cG \in \tcA_i$ be such that $\pi_*\cG \cong \cF$ (it exists by Proposition~\ref{proposition:sod-dbx}). 
Then $\cG$ has a direct sum decomposition $\cG \cong \oplus \cG_j$, 
where $\cG_j \in \tilde\gamma_j(\bD(\tR_{m_j}))$ (Proposition~\ref{proposition:tcai-auslander}).
Therefore, $\cF \cong \pi_*(\cG) \cong \oplus \pi_*(\cG_j)$.
Moreover, $\pi_*(\cG_j) \in \pi_*(\tilde\gamma_j(\bD(\tR_{m_j}))) = \gamma_j(\bD(\kk[t]/t^m))$,
hence objects $\cF_j = \pi_*(\cG_j)$ give the required decomposition of $\cF$.
\end{proof}

Now we can finish the proof of Theorem~\ref{theorem:dbx}.

\begin{lemma}
\label{lemma:cai-admissible}
The subcategories $\cA_i \subset \bD(X)$ in~\eqref{eq:dbx} are admissible.
\end{lemma}
\begin{proof}
Consider the semiorthogonal decompositions
\begin{equation*}
\bD(\tX) = \langle \tcA_3 \otimes \omega_\tX, \tcA_1, \tcA_2 \rangle
\qquand
\bD(\tX) = \langle \tcA_2 \otimes \omega_\tX, \tcA_3 \otimes \omega_\tX, \tcA_1 \rangle,
\end{equation*}
obtained from~\eqref{eq:dbtx} by mutations.
Since $\omega_\tX \cong \pi^*\omega_X$ and $\omega_X$ is a line bundle on $X$, we have 
\begin{equation*}
\pi_*(\tcA_i \otimes \omega_\tX) = \pi_*(\tcA_i) \otimes \omega_X = \cA_i \otimes \omega_X,
\end{equation*}
hence the arguments of this subsection also prove semiorthogonal decompositions
\begin{equation}\label{eq:dbx-other}
\bD(X) = \langle \cA_3 \otimes \omega_X, \cA_1, \cA_2 \rangle
\qquand
\bD(X) = \langle \cA_2 \otimes \omega_X, \cA_3 \otimes \omega_X, \cA_1 \rangle.
\end{equation}
Since the twist by $\omega_X$ is an autoequivalence of $\bD(X)$, these decompositions together with~\eqref{eq:dbx} 
show that each $\cA_i$ is admissible.
Indeed, (up to a twist) each $\cA_i$ appears in one of the three decompositions on the left, and in one on the right,
hence it is both left and right admissible.
\end{proof}

\begin{proof}[Proof of Theorem~\textup{\ref{theorem:dbx}}]
The semiorthogonal decomposition~\eqref{eq:dbx} is constructed in Proposition~\ref{proposition:sod-dbx}.
The equality $\cA_i = \pi_*(\tcA_i)$, which implies uniqueness of the decomposition, 
and finiteness of cohomological amplitude of the projection functors $\alpha_i$ are also proved there.
Admissibility of $\cA_i$ is proved in Lemma~\ref{lemma:cai-admissible}.
The structure of the components $\cA_i$ is described in Proposition~\ref{proposition:sod-cai}.
The required semiorthogonal decomposition of~$\bD^\perf(X)$ is obtained by~\cite[Proposition~4.1]{kuznetsov2011base}
and the embedding $\pi^*(\cA_i^\perf) \subset \tcA_i$ is evident from the definition
of the categories $\cA_i$ (see Propositions~\ref{proposition:sod-dbx-} and~\ref{proposition:sod-dbx}).
\end{proof}

For further convenience, we would like to rewrite the semiorthogonal decomposition~\eqref{eq:dbx} geometrically.

\begin{corollary}
\label{corollary:dbx-z2-z3}
For every sextic du Val del Pezzo surface $X$ over an algebraically closed field $\kk$ there are 
zero-dimensional Gorenstein schemes $Z_1$, $Z_2$, $Z_3$ of lengths $1$, $3$, and $2$ respectively, such that $\cA_i \cong \bD(Z_i)$ and
\begin{equation}
\label{eq:dbx-dz}
\bD(X) = \langle \bD(Z_1), \bD(Z_2), \bD(Z_3) \rangle.
\end{equation}
The scheme structure of $Z_i$ depends on the type of $X$ as follows:
\begin{equation*}\def\arraystretch{1.2}%
\begin{array}{|c|c|c|c|}
\hline
\text{Type of $X$} & Z_1 & Z_2 & Z_3 \\
\hline 
0 & \Spec(\kk) & \Spec(\kk) \sqcup \Spec(\kk) \sqcup \Spec(\kk) & \Spec(\kk) \sqcup \Spec(\kk) \\
\hline
1 & \Spec(\kk) & \Spec(\kk) \sqcup \Spec(\kk) \sqcup \Spec(\kk) & \Spec(\kk[t]/t^2) \\
\hline
2 & \Spec(\kk) & \Spec(\kk[t]/t^2) \sqcup \Spec(\kk) & \Spec(\kk)  \sqcup \Spec(\kk) \\
\hline
3 & \Spec(\kk) & \Spec(\kk[t]/t^2) \sqcup \Spec(\kk) & \Spec(\kk[t]/t^2) \\
\hline
4 & \Spec(\kk) & \Spec(\kk[t]/t^3) & \Spec(\kk)  \sqcup \Spec(\kk) \\
\hline
5 & \Spec(\kk) & \Spec(\kk[t]/t^3) & \Spec(\kk[t]/t^2) \\
\hline
\end{array}
\end{equation*}
\end{corollary}

Since $X$ is Gorenstein, we can produce yet another semiorthogonal decomposition by dualization.
Recall that $\cF^\vee$ denotes the derived dual of $\cF$, see Section~\ref{subsection:duality}.

\begin{proposition}\label{proposition:dbx-dual}
If $X$ is a sextic del Pezzo surface over an algebraically closed field, there is a semiorthogonal decomposition 
\begin{equation}
\label{eq:dbx-dual}
\bD(X) = \langle \cA_3^\vee, \cA_2^\vee, \cA_1^\vee \rangle,
\end{equation}
where
\begin{equation*}
\cA_i^\vee := \{ \cF \in \bD(X) \mid \cF^\vee \in \cA_i \} \cong \bD(Z_i),
\end{equation*}
and $\cA_i$ are the components of~\eqref{eq:dbx}.
Moreover, this semiorthogonal decomposition 
is right mutation-dual to~\eqref{eq:dbx}, i.e., $\cA_1^\vee = \cA_1$, $\cA_2^\vee = \LL_{\cA_1}(\cA_2)$ and $\cA_3^\vee = \LL_{\cA_1}(\LL_{\cA_2}(\cA_3))$.
\end{proposition}
\begin{proof}
As we discussed in Section~\ref{subsection:duality}, the functor $\cF \mapsto \cF^\vee$ is an equivalence $\bD(X)^{\mathrm{opp}} \to \bD(X)$ (since~$X$ is Gorenstein).
When applied to the semiorthogonal decomposition of Corollary~\ref{corollary:dbx-z2-z3}, it gives~\eqref{eq:dbx-dual} 
with~$\cA_i^\vee \cong \bD(Z_i)^{\mathrm{opp}}$.
Since $Z_1$, $Z_2$ and $Z_3$ are themselves Gorenstein, we have $\bD(Z_i)^{\mathrm{opp}} \cong \bD(Z_i)$.

Finally, let us prove mutation-duality of~\eqref{eq:dbx-dual} and~\eqref{eq:dbx}.
It follows from~\eqref{eq:alpha} that $\LL_{\cA_i} \circ \pi_* \cong \pi_* \circ \LL_{\tcA_i}$ and $\LL_{\tcA_i} \circ \pi^* \cong \pi^* \circ \LL_{\cA_i}$.
So, the required result follows easily from Lemma~\ref{lemma:dual-sod-tx}.
\end{proof}

\subsection{Generators of the components}

In this section we construct generators of the subcategories $\cA_2$ and $\cA_3$ of $\bD(X)$ 
and describe the equivalences $\bD(Z_2) \xrightarrow{\ \sim\ } \cA_2$ and $\bD(Z_3) \xrightarrow{\ \sim\ } \cA_3$ of Corollary~\ref{corollary:dbx-z2-z3} as Fourier--Mukai functors.

We will use the following lemma.

\begin{lemma}\label{lemma:sheaves-y-z}
Let $Z = \Spec(\kk[t]/t^m)$.
For any scheme $Y$ we have an equivalence 
\begin{equation*}
\coh(Y \times Z) \cong \coh(Y,\cO_Y[t]/t^m),
\end{equation*}
where the right hand side is the category of coherent sheaves $\cF$ on $Y$ with an operator $t \colon \cF \to \cF$ such that $t^m = 0$.
An object $(\cF,t) \in \coh(Y,\cO_Y[t]/t^m)$ considered as a sheaf on $Y \times Z$ is flat over $Z$ if and only if all the maps in the next chain of epimorphisms
\begin{equation}\label{eq:t-quotients}
\xymatrix@1{\cF/(t\cF) \ar@{->>}[r]^-t & (t\cF)/(t^2\cF) \ar@{->>}[r]^-t & \ \dots\  \ar@{->>}[r]^-t & (t^{m-1}\cF)/(t^m\cF)}
\end{equation}
are isomorphisms.
\end{lemma}
\begin{proof}
The first part follows immediately from an identification $Y \times Z = \Spec_Y(\cO_Y[t]/t^m)$.
To verify flatness, we should compute the derived pullback functors for the embedding $Y \to Y \times Z$ corresponding to the unique closed point of~$Z$.
Using the standard resolution
\begin{equation*}
\dots \xrightarrow{\ t^{m-1}\ } \kk[t]/t^m \xrightarrow{\ t\ } \kk[t]/t^m \xrightarrow{\ t^{m-1}\ } \kk[t]/t^m \xrightarrow{\ t\ } \kk[t]/t^m \xrightarrow{\ \ \ } \kk \to 0
\end{equation*}
we see that $(\cF,t)$ is flat over $Z$ if and only if the complex
\begin{equation*}
\dots \xrightarrow{\ t^{m-1}\ } \cF \xrightarrow{\ t\ } \cF \xrightarrow{\ t^{m-1}\ } \cF \xrightarrow{\ t\ } \cF \xrightarrow{\ t^{m-1}\ } \dots 
\end{equation*}
is exact.
This means that $\Ker(\cF \xrightarrow{\ t\ } \cF) = t^{m-1}\cF$ and $\Ker(\cF \xrightarrow{\ t^{m-1}\ } \cF) = t\cF$.
The second equality implies that the composition of all maps in~\eqref{eq:t-quotients} is an isomorphism, hence so is each of them.
On the other hand, if all the maps in~\eqref{eq:t-quotients} are isomorphisms then both equalities above easily follow.
This proves the required criterion of flatness.
\end{proof}

In what follows for a connected component $Z$ of $Z_d$ we denote by $\gamma_Z$ the embedding functor 
(constructed in Proposition~\ref{proposition:gamma})
\begin{equation*}
\gamma_Z \colon \bD(Z) \hookrightarrow \bD(X),
\end{equation*}
By definition~\eqref{def:gamma} we have $\gamma_Z = \pi_* \circ \tilde\gamma_Z \circ \pi_m^*$, 
where $\pi_m^* \colon \bD^-(Z) \to \bD^-(\tR_m)$ is the categorical resolution of the scheme $Z = \Spec(\kk[t]/t^m)$ described in~\eqref{eq:pms}, 
and $\tilde\gamma_Z \colon \bD(\tR_m) \to \bD(\tX)$ is the embedding of Proposition~\ref{proposition:tcai-auslander}.
Then 
\begin{equation*}
\gamma_{Z_2} = \bigoplus_{Z \subset Z_2} \gamma_Z
\qquand 
\gamma_{Z_3} = \bigoplus_{Z \subset Z_3} \gamma_Z,
\end{equation*}
(Proposition~\ref{proposition:sod-cai}) with the sum over all connected components in both cases.

In the next lemma we describe the images in $\bD(X)$ of the structure sheaves of points of the schemes~$Z_d$ under their embeddings $\gamma_{Z_d}$.
We use the convention of Section~\ref{subsection:dp6} on numbering the blowup centers.

\begin{lemma}\label{lemma:e-z-small}
$(1)$ 
If $z$ is the unique closed $\kk$-point of $Z_1$ then $\gamma_{Z_1}(\cO_z) \cong \cO_X \cong \pi_*\cO_\tX$.

\noindent$(2)$
Let $z \in Z_2$ be a closed $\kk$-point and let $\cE_z := \gamma_{Z_2}(\cO_z) \in \bD(X)$ be the corresponding object.
Then
\begin{itemize}
\item[(a)] 
if the component of $Z_2$ containing $z$ is reduced and corresponds to the $i$-th blowup center, then 
\begin{equation*}
\cE_z \cong \pi_*\cO_\tX(h-e_i);
\end{equation*}
\item[(b)] 
if the component of $Z_2$ containing $z$ has length $2$
then 
\begin{equation*}
\cE_z \cong \pi_*\cO_\tX(h-e_1) \cong \pi_*\cO_\tX(h - e_2);
\end{equation*}
\item[(c)] 
if the component of $Z_2$ containing $z$ has length $3$, then
\begin{equation*}
\cE_z \cong \pi_*\cO_\tX(h-e_1) \cong \pi_*\cO_\tX(h - e_2) \cong \pi_*\cO_\tX(h - e_3).
\end{equation*}
\end{itemize}
\noindent$(3)$
Let $z \in Z_3$ be a closed $\kk$-point and let $\cE_z := \gamma_{Z_3}(\cO_z) \in \bD(X)$ be the corresponding object.
Then
\begin{itemize}
\item[(a)] 
if the component of $Z_3$ containing $z$ is reduced then 
\begin{equation*}
\cE_z \cong \pi_*\cO_\tX(h)
\qquad\text{or}\qquad 
\cE_z \cong \pi_*\cO_\tX(2h - e_1 - e_2 - e_3);
\end{equation*}
\item[(b)] 
if the component of $Z_3$ containing $z$ has length $2$ then 
\begin{equation*}
\cE_z \cong \pi_*\cO_\tX(h) \cong \pi_*\cO_\tX(2h - e_1 - e_2 - e_3).
\end{equation*}
\end{itemize}
In all these cases, $\cE_z$ is a globally generated sheaf, and if $z \in Z_d$ then
\begin{equation*}
\dim H^0(X,\cE_z) = d
\qquand
H^{i}(X,\cE_z) = 0
\quad\text{for $i \ne 0$}.
\end{equation*}
\end{lemma}

\begin{proof}
The statements (1), (2), and~(3) follow
directly from the first commutative diagram in~\eqref{eq:gamma-tgamma},
isomorphism~$\pi_{m*}(S_0) \cong \cO_z$ proved in Proposition~\ref{proposition:auslander-localization}, 
and the fact that the simple module $S_0$ coincides with the standard exceptional module $E_0$ (see~\eqref{eq:resolution-simple}), so that by Proposition~\ref{proposition:tcai-auslander}
the functor $\tilde\gamma$ takes it to one of the line bundles in~\eqref{eq:tx-excol}.
Going over the possible cases gives for~$\cE_z$ the first isomorphisms.
The other isomorphisms of sheaves in case of non-reduced components follow from exact sequences~\eqref{eq:e12}, \eqref{eq:e23} and~\eqref{eq:e123} respectively.

In cases $(1)$, $(2a)$ and $(3a)$ the global generation is easy (the corresponding sheaves are already globally generated on $\tX$; 
extending their evaluation homomorphisms to exact sequences and pushing them forward to $X$ it is easy to see that their pushforwards are also globally generated).
In cases $(2b)$, $(2c)$, and~$(3b)$ the same argument shows that $\pi_*\cO_\tX(h-e_1)$, $\pi_*\cO_\tX(h-e_3)$ and $\pi_*\cO_\tX(h)$ are globally generated, 
and for the remaining two sheaves we can use the corresponding isomorphisms.

The cohomology computation reduces to a computation on $\tX$ which is straightforward.
\end{proof}

Next we determine the images 
\begin{equation}
\label{def:fez}
\fE_Z := \gamma_Z(\cO_Z) \in \bD(X)
\qquad\text{and}\qquad 
\tfE_Z := \pi^*(\fE_Z) \in \bD(\tX)
\end{equation}
of the structure sheaves of connected components of~$Z_d$.
We denote by~$\ell(Z)$ the length of~$Z$.

\begin{proposition}
\label{proposition:e-z-big}
Both $\fE_Z$ and $\tfE_Z$ are vector bundles of rank~$\ell(Z)$ on~$X$ and~$\tX$ respectively, and
\begin{equation}
\label{eq:ez-pullback}
\fE_Z \cong \pi_*(\tfE_Z).
\end{equation}
Moreover, $\fE_Z$ is an iterated extension of the sheaf $\cE_z$, where $z$ is the closed point of~$Z$.
In particular, the bundle $\fE_Z$ is globally generated with $\dim H^0(X,\fE_Z) = d\ell(Z)$ and $H^i(X,\fE_Z) = 0$ for $i \ne 0$.
\end{proposition}
\begin{proof}
First, by commutativity of~\eqref{eq:gamma-tgamma} we have $\tfE_Z \cong \tilde\gamma_Z(\pi_m^*(\cO_Z))$.
Furthermore, $\pi_m^*(\cO_Z) \cong P_0$, the principal projective module, see~\eqref{eq:pms}, so that
\begin{equation}\label{eq:tcez-p0}
\tfE_Z \cong \tilde\gamma_Z(P_0).
\end{equation}
By~\eqref{eq:projective-resolution-ei} the module $P_0$ is an iterated extension of the standard exceptional modules $E_i$, 
hence $\tilde\gamma_Z(P_0)$ is an iterated extension of $\tilde\gamma_Z(E_i)$, 
and as it is explained in the proof of Proposition~\ref{proposition:tcai-auslander} these are line bundles on $\tX$.
More precisely, if $Z$ is reduced, then 
\begin{equation*}
\tfE_Z \cong \cO_\tX(h-e_i)\quad\text{for some $i$},\quad\text{or}\quad 
\tfE_Z \cong \cO_\tX(h),\quad\text{or}\quad
\tfE_Z \cong \cO_\tX(2h-e_1-e_2-e_3),
\end{equation*}
depending on whether $Z$ is a component of $Z_2$ or of $Z_3$.
Furthermore, if $Z$ has length~2 then we have exact sequences
\begin{equation}\label{eq:tcez-12}
0 \to \cO_\tX(h-e_2) \to \tfE_Z \to \cO_\tX(h-e_1) \to 0,
\end{equation}
when $Z$ is a component of $Z_2$, or
\begin{equation}\label{eq:tcez-123}
0 \to \cO_\tX(2h-e_1-e_2-e_3) \to \tfE_Z \to \cO_\tX(h) \to 0,
\end{equation}
when $Z$ is a component of $Z_3$.
Finally, if $Z$ has length~3 we have two exact sequences
\begin{equation}\label{eq:tcez-12-23}
0 \to \cO_\tX(h-e_3) \to \tfE_Z \to \tfE'_Z \to 0,
\qquad\text{where}\qquad
0 \to \cO_\tX(h-e_2) \to \tfE'_Z \to \cO_\tX(h-e_1) \to 0.
\end{equation}
In all cases $\tfE_Z$ is evidently a vector bundle of rank~$\ell(Z)$, 
and since $\tfE_Z \cong \pi^*\fE_Z$, it follows that $\fE_Z$ is also a vector bundle of the same rank~$\ell(Z)$.
Finally, since $Z \cong \Spec(\kk[t]/t^m)$, the sheaf $\cO_Z$ is an iterated extension of the sheaf $\cO_z$,
where $z$ is the closed point of $Z$, hence $\fE_Z = \gamma_Z(\cO_Z)$ is an iterated extension of the sheaf $\cE_z = \gamma_Z(\cO_z)$. 
Now the last part of Lemma~\ref{lemma:e-z-small} implies the last part of the proposition. 
\end{proof}

The vector bundles constructed above allow to present the equivalence of $\bD(Z_d)$ and $\cA_d \subset \bD(X)$ as a Fourier--Mukai functor.
Set 
\begin{equation}\label{eq:cezd}
\fE_{Z_1} := \cO_X,
\qquad
\fE_{Z_2} := \bigoplus_{Z \subset Z_2} \fE_Z,
\qquand
\fE_{Z_3} := \bigoplus_{Z \subset Z_3} \fE_Z,
\end{equation}
where the sums are taken over all connected components of $Z_2$ and $Z_3$ respectively.
By Proposition~\ref{proposition:e-z-big} these are globally generated vector bundles on $X$ of ranks $1$, $3$, and $2$ respectively.
We denote by~\mbox{$p_X \colon X \times Z_d \to X$} the natural projection.

\begin{proposition}
\label{proposition:fm-zd}
For each $d \in \{1,2,3\}$ there is a sheaf 
\begin{equation*}
\cE_{Z_d} \in \coh(X \times Z_d),
\end{equation*}
flat over $Z_d$, such that $p_{X*}\cE_{Z_d} \cong \fE_{Z_d}$, and the equivalence $\gamma_{Z_d} \colon \bD(Z_d) \to \cA_d \subset \bD(X)$ of Corollary~\textup{\ref{corollary:dbx-z2-z3}} can be written as
\begin{equation*}
\gamma_{Z_d} \cong \Phi_{\cE_{Z_d}} \colon \bD(Z_d) \to \bD(X),
\end{equation*}
where $\Phi_{\cE_{Z_d}}$ is the Fourier--Mukai functor with kernel $\cE_{Z_d}$.
Moreover, the equivalence \mbox{$\bD(Z_d) \cong \cA_d^\vee \subset \bD(X)$} of Proposition~\textup{\ref{proposition:dbx-dual}} is given by the Fourier--Mukai functor $\Phi_{\cE_{Z_d}^\vee}$ with kernel $\cE_{Z_d}^\vee$.
\end{proposition}
\begin{proof}
Let $Z \cong \Spec(\kk[t]/t^m) \subset Z_d$ be a connected component.
We apply Lemma~\ref{lemma:sheaves-y-z} to construct a sheaf $\cE_Z$ on~$X \times Z$ such that $\fE_Z \cong p_{X*}\cE_Z$.
For this, first, note that the filtration on $\cO_Z$ induced by the natural action of the ring $\kk[t]/t^m$ 
has quotients isomorphic to $\cO_z$, where $z$ is the closed point of $Z$,
hence the induced filtration on $\fE_Z \cong \gamma_Z(\cO_Z)$ has quotients isomorphic to $\gamma(\cO_z) = \cE_z$, i.e., 
\begin{equation}\label{eq:fez-quotients}
(t^i\fE_Z)/(t^{i+1}\fE_Z) \cong \cE_z.
\end{equation}
It is clear that the epimorphisms between these quotients induced by $t$ have to be isomorphisms,
hence the sheaf $\cE_Z$ associated with $\fE_Z$ is flat over~$Z$ by Lemma~\ref{lemma:sheaves-y-z}.


Next, we define $\cE_{Z_d}$ as the sum $\cE_{Z_d} := \bigoplus \cE_Z$ over connected components $Z \subset Z_d$.
Then the first part of the proposition holds.
It remains to show that the functors $\gamma_{Z_d}$ are Fourier--Mukai.

The functor $\tilde\gamma_Z \colon \bD(\tR_m) \to \bD(\tX)$ by construction 
(which is explained in Proposition~\ref{proposition:tcai-auslander} and is based on Proposition~\ref{proposition:rm-equivalence}) 
is a Fourier--Mukai functor.
By~\eqref{def:gamma} the functor $\gamma_Z \colon \bD(Z) \to \bD(X)$ is also a Fourier--Mukai functor.
Since the scheme $Z$ is affine, the kernel object defining the functor $\gamma_Z$ can be identified with the sheaf  $\gamma_Z(\cO_Z) \cong \fE_Z \in \bD(X)$ with its natural module structure over $\kk[Z] = \kk[t]/t^m$, 
which corresponds to the sheaf $\cE_Z$ by its definition above.

The last claim follows from this by dualization (note that Grothendieck duality implies that the equivalence of Lemma~\ref{lemma:sheaves-y-z} is compatible with dualization, since the dualizing complex of $Z$ is trivial).
\end{proof}

\begin{corollary}\label{corollary:a-as-orthogonals}
The components $\cA_2$ and $\cA_3$ of $\bD(X)$ are compactly generated by the bundles $\fE_{Z_2}$ and~$\fE_{Z_3}$ respectively.
In particular, we have $\cA_2 = {}^\perp \cO_X \cap \fE_{Z_3}^\perp$ and $\cA_3 = {}^\perp \cO_X \cap {}^\perp\fE_{Z_2}$.
\end{corollary}
\begin{proof}
Indeed, the derived category of an affine scheme $Z_d$ is compactly generated by its structure sheaf~$\cO_{Z_d}$.
Therefore, the component $\cA_d$ of $\bD(X)$ is compactly generated by the bundle $\fE_{Z_d} = \gamma_{Z_d}(\cO_{Z_d})$.
For the second claim use~\eqref{eq:dbx}.
\end{proof}

As we will see in the next section, for verification of the orthogonality the following numerical result is useful.
Denote 
\begin{equation*}
\chi(\cF,\cG) = \sum (-1)^i \dim \Ext^i(\cF,\cG)
\qquad\text{and}\qquad 
r(\cF) = \sum (-1)^i \rank(\cH^i(\cF))
\end{equation*}
(assuming the sums are actually finite).
Note that $r(\cF) = \chi(\cO_P,\cF)$ for any smooth point $P$ of $X$.

\begin{lemma}\label{lemma:chi-ez-f}
For any $\cF \in \bD(X)$ we have
\begin{alignat*}{2}
\chi(\fE_{Z_2},\cF) &= 2 &&\chi(\cO_X,\cF) + \chi(\cO_X,\cF(K_X)) - 3r(\cF),\\
\chi(\fE_{Z_3},\cF) &=   &&\chi(\cO_X,\cF) + \chi(\cO_X,\cF(K_X)) - 2r(\cF).
\end{alignat*}
\end{lemma}
\begin{proof}
By Corollary~\ref{corollary:pis-epi} we can write $\cF = \pi_*\tcF$ for some $\tcF \in \bD(\tX)$, hence by adjunction
\begin{equation*}
\chi(\fE_{Z_d}, \cF) = \chi(\fE_{Z_d},\pi_*\tcF) = \chi(\pi^*\fE_{Z_d},\tcF) = \chi(\tfE_{Z_d},\tcF).
\end{equation*}
As it is explained in Proposition~\ref{proposition:e-z-big} the bundle $\tfE_{Z_2}$ 
is an extension (possibly trivial) of the line bundles~$\cO_\tX(h-e_i)$ for $i = 1,2,3$, hence 
\begin{equation*}
\ch(\tfE_{Z_2}) = \sum \ch(\cO_\tX(h-e_i)) = 3 + (3h - e_1 - e_2 - e_3) = 3 - K_\tX = 2\ch(\cO_\tX) + \ch(\omega_\tX^{-1}) - 3\ch(\cO_\tP),
\end{equation*}
where $\tP$ is a general point on $\tX$.
Therefore, by Riemann--Roch and adjunction we have
\begin{equation*}
\chi(\tfE_{Z_2},\tcF) = 
2\chi(\cO_\tX,\tcF) + \chi(\omega_\tX^{-1},\tcF) - 3\chi(\cO_\tP,\tcF) =
2\chi(\cO_X,\cF) + \chi(\omega_X^{-1},\cF) - 3\chi(\cO_{\pi(\tP)},\cF),
\end{equation*}
which gives the first equality.
Similarly, for $\tfE_{Z_3}$ we have
\begin{equation*}
\ch(\tfE_{Z_3}) = \ch(\cO_\tX(h)) + \ch(\cO_\tX(2h - e_1 - e_2 - e_3)) = 
\ch(\cO_\tX) + \ch(\omega_\tX^{-1}) - 2\ch(\cO_\tP),
\end{equation*}
and using Riemann--Roch and adjunction in the same way as before, we finish the proof.
\end{proof}

Denote by $p_2 \colon X \times Z_2 \to Z_2$, $p_3 \colon X \times Z_3 \to Z_3$, and $p_{23} \colon X \times Z_2 \times Z_3 \to Z_2 \times Z_3$ the projections.
Similarly, consider the projections $p_{X2} \colon X \times Z_2 \times Z_3 \to X \times Z_2$ and $p_{X3} \colon X \times Z_2 \times Z_3 \to X \times Z_3$.

\begin{lemma}
\label{lemma:ps-ez}
The pushforwards $p_{2*}\cE_{Z_2}$ and $p_{3*}\cE_{Z_3}$ are vector bundles on $Z_2$ and $Z_3$ of rank $2$ and $3$ respectively.
Similarly, 
the pushforward $p_{23*}(p_{X2}^*\cE_{Z_2}^\vee \otimes p_{X3}^*\cE_{Z_3})$ is a line bundle on $Z_2 \times Z_3$.
\end{lemma}
\begin{proof}
Let $Z \cong \Spec(\kk[t]/t^m)$ be a connected component of $Z_d$ with $d \in \{2,3\}$ with closed point~$z$.
Consider the fiber square
\begin{equation*}
\xymatrix{
X \times Z \ar[r]^-{p_Z} \ar[d]_{p_X} & Z \ar[d] \\
X \ar[r] & \Spec(\kk)
}
\end{equation*}
The pushforward to $\Spec(\kk)$ of $p_{Z*}\cE_{Z}$ equals the pushforward of $p_{X*}(\cE_{Z}) \cong \fE_Z$, 
i.e.~$H^\bullet(X,\fE_{Z})$ with its natural $\kk[t]/t^m$-module structure.
So, by Lemma~\ref{lemma:sheaves-y-z} to check that $p_{Z*}\cE_{Z}$ is a vector bundle
it is enough to show that the natural epimorphisms between the quotient spaces $(t^iH^0(X,\fE_{Z}))/(t^{i+1}H^0(X,\fE_{Z}))$ are isomorphisms
(the other cohomology groups of $\fE_Z$ vanish by Proposition~\ref{proposition:e-z-big}).
By~\eqref{eq:fez-quotients} all these quotients are isomorphic to $H^0(X,\cE_z)$, 
hence $d$-dimensional by Lemma~\ref{lemma:e-z-small}, 
hence the epimorphisms $t$ between them are isomorphisms.
This proves that $p_{Z*}\cE_{Z}$ is locally free of rank~$d$ on $Z$.

For the second statement, let $Z \subset Z_2$ and $Z' \subset Z_3$ be connected components, 
$Z \cong \Spec(\kk[t]/t^m)$, and~$Z' \cong \Spec(\kk[t']/(t')^{m'})$.
Under an analogue of Lemma~\ref{lemma:sheaves-y-z} the sheaf $p_{23*}(p_{X2}^*\cE_{Z_2}^\vee \otimes p_{X3}^*\cE_{Z_3})$ 
corresponds to the vector space $H^\bullet(X,\fE_{Z_2}^\vee \otimes \fE_{Z_3})$ with its natural bifiltration.
So, we have to check that the operators $t$ and $t'$ induce isomorphisms between the quotients of this bifiltration.
Since
\begin{equation*}
H^\bullet(X,\fE_{Z_2}^\vee \otimes \fE_{Z_3}) 
\cong \Ext^\bullet(\fE_{Z_2},\fE_{Z_3})
\cong \Ext^\bullet(\pi^*\fE_{Z_2},\pi^*\fE_{Z_3})
\cong \Ext^\bullet(\tfE_{Z_2},\tfE_{Z_3}),
\end{equation*}
and the bifiltration is induced by the defining exact sequences~\eqref{eq:tcez-12}, \eqref{eq:tcez-12-23} and~\eqref{eq:tcez-123} of $\tfE_Z$ and $\tfE_{Z'}$, 
it is enough to compute $\Ext$-spaces between the corresponding line bundles on $\tX$.
A direct computation gives 
\begin{alignat*}{3}
& \Ext^\bullet(\cO_\tX(h-e_i),\cO_\tX(h))			&&\cong 	H^\bullet(\tX,\cO_\tX(e_i)) 		&&= \kk,\\
& \Ext^\bullet(\cO_\tX(h-e_i),\cO_\tX(2h - e_1 - e_2 - e_3)) 	&&\cong 	H^\bullet(\tX,\cO_\tX(h - e_j - e_k)) 	&&= \kk,
\end{alignat*}
therefore, the epimorphisms $t$ and $t'$ between them are isomorphisms, and the sheaf $p_{23*}(p_{X2}^*\cE_{Z_2}^\vee \otimes p_{X3}^*\cE_{Z_3})$ is locally free of rank 1 on $Z \times Z'$.

Summing up over all connected components completes the proof of the lemma.
\end{proof}

\begin{remark}
The above lemma can be interpreted as a computation of the gluing bimodules (cf.~\cite[Section~2.2]{kuznetsov2015categorical}) between the components of~\eqref{eq:dbx}.
It says that $\bD(X)$ is the gluing of $\bD(\kk)$, $\bD(Z_2)$, and $\bD(Z_3)$ with the gluing bimodules being $\kk[Z_2]^{\oplus 2}$, $\kk[Z_2 \times Z_3]$, and $\kk[Z_3]^{\oplus 3}$.
\end{remark}

\section{Moduli spaces interpretation}\label{section:moduli-spaces}

In this section we provide a modular interpretation for the finite length schemes $Z_2$ and $Z_3$ that appeared in the semiorthogonal decomposition~\eqref{eq:dbx-dz} of $\bD(X)$
and for the Fourier--Mukai kernels $\cE_{Z_2}$ and~$\cE_{Z_3}$ of Proposition~\ref{proposition:fm-zd}.
This interpretation is essential for the description of the derived category of a family of sextic del Pezzo surfaces in Section~\ref{section:families}.

All through this section $X$ is a sextic du Val del Pezzo surface (as defined in Definition~\ref{definition:dp6}) over an algebraically closed field $\kk$, 
and we use freely the notation introduced in Section~\ref{subsection:dp6}.

\subsection{Moduli of rank 1 sheaves}\label{subsection:moduli}

For a sheaf $\cF$ on $X$ we denote by 
\begin{equation*}
h_\cF(t) := \chi(\cF(-tK_X)) \in \ZZ[t],
\end{equation*}
the Hilbert polynomial of $\cF$ with respect to the anticanonical polarization of $X$.
This is a quadratic polynomial with the leading coefficient equal to $r(\cF)\cdot K_X^2/2 = 3r(\cF)$.
Note that 
\begin{equation}\label{eq:hox}
h_{\cO_X}(t) = 3t(t+1) + 1.
\end{equation}

For each $d \in \ZZ$ we consider the polynomial
\begin{equation}\label{eq:hd}
h_d(t) := (3t + d)(t+1) 
\in \ZZ[t].
\end{equation}
An elementary verification shows
\begin{equation}\label{eq:hd-inequalities}
h_4(t) > h_3(t) > h_2(t) > h_{\cO_X}(t) > h_4(t-1) > h_3(t-1) > h_2(t-1)
\qquad\text{for all $t \gg 0$.}
\end{equation} 
Below we will be interested in semistable sheaves on $X$ with Hilbert polynomial $h_d(t)$.

\begin{lemma}
\label{lemma:hd-stable}
A sheaf $\cF$ on $X$ with Hilbert polynomial $h_d(t)$ is Gieseker semistable if and only if it is Gieseker stable and if and only if it is torsion-free.
\end{lemma}
\begin{proof}
As we observed above, the leading monomial of $h_d(\cF)$ being $3t^2$ means that $r(\cF) = 1$.
Therefore, such a sheaf $\cF$ could (and will) be destabilized only by a subsheaf of rank~0, i.e., by a torsion sheaf.
Thus, $\cF$ is (semi)stable if and only if it is torsion-free.
\end{proof}

Recall the sheaves $\cE_z$, $\fE_Z$, and $\fE_{Z_d}$ on $X$ introduced 
in Lemma~\ref{lemma:e-z-small}, \eqref{def:fez}, and~\eqref{eq:cezd} respectively.

\begin{lemma}\label{lemma:hilbert-fz}
Let $d \in \{2, 3\}$.
For any closed point $z \in Z_d$ the sheaf $\cE_z$ is stable with $h_{\cE_z}(t) = h_d(t)$.
Its derived dual~$\cE_z^\vee$ is a stable sheaf with $h_{\cE_z^\vee}(t) = h_d(-t-1) = h_{6-d}(t-1)$.
Furthermore, the sheaf $\fE_{Z_d}$ is semistable with $h_{\fE_{Z_d}} = \ell(Z_d)h_d$.
\end{lemma}
\begin{proof}
As we already mentioned, $r(\cE_z) = 1$ implies the leading monomial of $h_{\cE_z}$ equals $3t^2$.
So, to show an equality of polynomials $h_{\cE_z} = h_d$ it is enough to check that they take the same values at points~$t = 0$ and~$t = -1$.
In other words, we have to check that 
\begin{equation*}
\chi(X,\cE_z) = h_d(0) = d
\qquand
\chi(X,\cE_z \otimes \omega_X) = h_d(-1) = 0.
\end{equation*}
The first is proved in Lemma~\ref{lemma:e-z-small}, and the second follows from $\cE_z \in \cA_d$ and decompositions~\eqref{eq:dbx-other}.

The sheaf $\cE_z$ is torsion free since by Lemma~\ref{lemma:e-z-small} it is the direct image of a line bundle 
under a dominant map $\pi$, hence is stable by Lemma~\ref{lemma:hd-stable}.
Furthermore, it follows from Proposition~\ref{proposition:e-z-big} that $\fE_{Z_d}$ is semistable with the same reduced Hilbert polynomial.

To show that $\cE_z^\vee$ is a sheaf, consider the connected component $Z$ of~$Z_d$ containing~$z$ and
note that the sheaf $\fE_Z$ is filtered by $\cE_z$, hence $\fE_Z^\vee$ is filtered by $\cE_z^\vee$.
But $\fE_Z$ is locally free by Proposition~\ref{proposition:e-z-big}, 
hence~$\fE_Z^\vee$ is a sheaf, hence $\cE_z^\vee$ is a sheaf as well.
%
%
%
%
Stability of $\cE_z^\vee$ is proved in the same way as that of $\cE_z$; and the fact that its Hilbert polynomial equals $h_d(-t-1)$ follows easily from Serre duality (Proposition~\ref{proposition:serre-gorenstein}).
\end{proof}

Denote by 
\begin{equation}
\label{eq:mdx}
\cM_d(X) := \cM_{X,-K_X}(h_d)
\end{equation}
the moduli space of Gieseker semistable sheaves on $X$ with Hilbert polynomial $h_d(t)$ (with respect to the anticanonical polarization of $X$).
We aim at description of these moduli spaces for $d  \in \{2,3,4\}$.

We start by describing their closed points.

\begin{lemma}
\label{lemma:md-points}
Let $\cF$ be a torsion free sheaf on $X$ whose Hilbert polynomial is $h_d(t)$ with $d \in \{2, 3, 4 \}$.
$(i)$ If $d \in \{2, 3\}$ there is a unique closed point $z \in Z_d$ such that $\cF \cong \cE_z$.
\newline
$(ii)$ If $d \in \{3, 4\}$ there is a unique closed point $z \in Z_{6-d}$ such that $\cF \cong \cE_z^\vee \otimes \omega_X^{-1}$.
\end{lemma}
\begin{proof}
First, assume $d = 2$.
Let us show that $\cF \in \cA_2$. 
By Corollary~\ref{corollary:a-as-orthogonals} for this we should check that~$\Ext^\bullet(\cF,\cO_X) = \Ext^\bullet(\fE_{Z_3},\cF) = 0$.
By Lemma~\ref{lemma:hd-stable} the sheaf $\cF$ is stable.
By Lemma~\ref{lemma:hilbert-fz} and~\eqref{eq:hd-inequalities}, we have $\frac1{\ell(Z_3)}h_{\fE_{Z_3}}(t) > h_\cF(t) > h_{\cO_X}(t)$, hence by semistability 
\begin{equation*}
\Hom(\fE_{Z_3},\cF) = \Hom(\cF,\cO_X) = 0.
\end{equation*}
Similarly, $\frac1{\ell(Z_3)}h_{\fE_{Z_3}}(t-1) < h_\cF(t) < h_{\cO_X}(t+1)$, hence
\begin{equation*}
\Hom(\cF,\fE_{Z_3}(K_X)) = \Hom(\cO_X(-K_X),\cF) = 0.
\end{equation*}
By Serre duality on $X$ (note that $\fE_{Z_3}$ is locally free and use Proposition~\ref{proposition:serre-gorenstein}) we deduce 
\begin{equation*}
\Ext^2(\fE_{Z_3},\cF) = \Ext^2(\cF,\cO_X) = 0.
\end{equation*}
Since (again by Serre duality and local freeness of $\cO_X$ and $\fE_{Z_3}$) we can have 
nontrivial $\Ext^p(\fE_{Z_3},\cF)$ and~$\Ext^p(\cF,\cO_X)$ only for $p \in \{0,1,2\}$,
it remains to check that 
\begin{equation*}
\chi(\fE_{Z_3},\cF) = \chi(X,\cF(K_X)) = 0.
\end{equation*}
The second equality here just follows from $h_\cF(-1) = h_2(-1) = 0$,
and the first follows from Lemma~\ref{lemma:chi-ez-f} and $h_2(0) + h_2(-1) - 2 = 0$.

Thus, we have shown that $\cF \in \cA_2 \cong \bD(Z_2)$. 
Since $Z_2$ is a zero-dimensional scheme, any object in~$\bD(Z_2)$ is isomorphic 
to an iterated extension of shifts of structure sheaves of closed points in~$Z_2$.
By Lemma~\ref{lemma:e-z-small} these sheaves correspond to sheaves~$\cE_z \in \cA_2$, hence $\cF$ is an extension of shifts of those.
Since $\cF$ is a pure sheaf of rank equal to~1, 
we conclude that $\cF \cong \cE_z$ for a closed point $z \in Z_2$.

The case $d = 3$ is treated similarly. 
%
The same argument with~$\cA_2$ and~$\cA_3$ replaced by~$\cA_2^\vee \otimes \omega_X^{-1}$ and~$\cA_3^\vee \otimes \omega_X^{-1}$ 
and~\eqref{eq:dbx} replaced by a twist of~\eqref{eq:dbx-dual} proves part~$(ii)$.
\end{proof}

Below we consider families of objects parameterized by a scheme $S$.
Let $p_S \colon M \to S$ be a morphism of schemes.
For each geometric point $s \in S$ we denote by $M_s$ the scheme-theoretic fiber of $M$ over $s$ 
and by~$i_{M_s} \colon M_s \hookrightarrow M$ its embedding.
The embedding $s \hookrightarrow S$ is denoted simply by $i_s$.

\begin{lemma}
\label{lemma:fiberwise}
Assume $\cF \in \bD^-(M)$ and $M$ is proper over $S$, and in $(ii)$ and $(iii)$ that $M$ is flat over $S$.
\begin{enumerate}
\item[$(i)$] 
If $i_{M_s}^*(\cF) \in \bD^{\le p_0}(M_s)$ for some integer $p_0$ and all geometric points $s \in S$, then $\cF \in \bD^{\le p_0}(M)$.
In particular, if $i_{M_s}^*(\cF) = 0$ for all geometric points $s \in S$, then $\cF = 0$.
\item[$(ii)$] 
If for all geometric points $s \in S$ the pullback $i_{M_s}^*(\cF)$ is a pure sheaf 
such that $H^k(M_s, i_{M_s}^*\cF) = 0$ for all~$k \ne 0$ and $\dim H^0(M_s,i_{M_s}^*\cF) = r$ for a constant $r$,
then $p_{S*}(\cF)$ is a vector bundle of rank~$r$ on~$S$.
\item[$(iii)$] 
If for any geometric point $s \in S$ there is a point $m \in M_s$ such that $i_{M_s}^*(\cF) \cong \cO_m$, then 
there is a unique section $\varphi \colon S \to M$ of the projection $p_S$ and a line bundle $\cL \in \Pic(S)$ such that $\cF \cong \varphi_*\cL$.
\end{enumerate}
\end{lemma}
\begin{proof}
$(i)$ 
Let $\cH^k(\cF)$ be the top nonzero cohomology sheaf of $\cF$; by assumption it is a coherent sheaf on $M$.
Let $s \in S$ be a geometric point in the image of the support of~$\cH^k(\cF)$ under the map $p_S \colon M \to S$.
The spectral sequence
\begin{equation}
\label{eq:spectral}
L_qi_{M_s}^*\cH^p(\cF) \Rightarrow \cH^{p-q}(i_{M_s}^*\cF)
\end{equation}
implies that $\cH^k(i_{M_s}^*\cF) \cong L_0i_{M_s}^*\cH^k(\cF) \ne 0$, hence $k \le p_0$.

$(ii)$ Set $\cG := p_{S*}(\cF)$.
By base change we have an isomorphism 
\begin{equation*}
H^\bullet(M_s,i_{M_s}^*(\cF)) \cong i_s^* (\cG).
\end{equation*}
By assumption, the left hand side is a vector space of dimension $r$ in degree $0$.
By~$(i)$ we have
$\cG \in \bD^{\le 0}(S)$.
Moreover the spectral sequence~\eqref{eq:spectral} with $M = S$ and $\cF = \cG$
\begin{equation*}
L_q i_s^* \cH^p(\cG) \Rightarrow \cH^{p-q}(i_s^*(\cG)) \cong H^{p-q}(M_s, i_{M_s}^*(\cF))
\end{equation*}
shows that $L_0i_s^*\cH^0(\cG) \cong H^0(M,i_{M_s}^*\cF)$ and $L_1i_s^*\cH^0(\cG) = 0$.
Therefore, by Serre's criterion the sheaf~$\cH^0(\cG)$ is a vector bundle of rank~$r$.
Looking again at the spectral sequence, we see that the canonical map $\cG \to \cH^0(\cG)$ induces an isomorphism $i_s^*\cG \to i_s^*(\cH^0(\cG))$ for each $s \in S$, 
hence by part $(i)$ we have an isomorphism $\cG \cong \cH^0(\cG)$.
Thus $p_{S*}(\cF) = \cG$ is a vector bundle of rank $r$.

$(iii)$ 
By part $(i)$ we have $\cF \in \bD^{\le 0}(M)$ and by part~$(ii)$ we know that $p_{S*}(\cF)$ is a line bundle.
Denote it by $\cL$.
Then replacing the object $\cF$ by~$\cF \otimes p_S^*\cL^{-1}$, we may assume that $p_{S*}(\cF) \cong \cO_S$.
We have by adjunction a map $\cO_{M} = p_S^*\cO_S \to \cF \to \cH^0(\cF)$ 
that restricts to the fiber $M_s$ over a geometric point~\mbox{$s \in S$} as the natural map $\cO_{M_s} \to \cO_m$.
Therefore, it is fiberwise surjective, hence by part $(i)$ applied to its cone it is surjective on $M$.
So, $\cH^0(\cF) \cong \cO_\Gamma$ is the structure sheaf of a closed subscheme $\Gamma \subset M$.

By the assumption for each geometric point $s \in S$ we have $i_s^*\cO_\Gamma \cong \cO_m$. 
Therefore the map $p_S\vert_\Gamma \colon \Gamma \to S$ is finite and flat of degree~1.
By $(ii)$ we have $p_{S*}\cO_\Gamma \cong \cO_S$, so it follows that the map $p_S\vert_\Gamma \colon \Gamma \to S$ is an isomorphism.
Therefore, $\Gamma$ is the image of a section $\varphi \colon S \to M$ of the morphism $p_S$.

The above argument proves $\cH^0(\cF) \cong \cO_\Gamma \cong \varphi_*\cO_S$.
Restricting the triangle $\tau^{\le -1}\cF \to \cF \to \cH^0(\cF)$ to an arbitrary fiber $M_s$ we deduce that $i_{M_s}^*(\tau^{\le -1}\cF) = 0$ for any $s \in S$, hence $\tau^{\le -1}\cF = 0$ by part $(i)$.
Therefore $\cF \cong \varphi_*\cO_S$.
\end{proof}

Now we return to a sextic del Pezzo surface $X$ and the moduli spaces $\cM_d(X)$ defined by~\eqref{eq:mdx}.
Recall the sheaves $\cE_{Z_d}$ on $X \times Z_d$ constructed in Proposition~\ref{proposition:fm-zd}.

\begin{theorem}
\label{theorem:moduli}
Let $X$ be a sextic du Val del Pezzo surface over an algebraically closed field $\kk$. 
The moduli spaces $\cM_2(X)$, $\cM_3(X)$, and $\cM_4(X)$ are fine moduli spaces. 
Moreover,
\newline$(i)$
$\cM_2(X) \cong \cM_4(X) \cong Z_2$ and the sheaves $\cE_{Z_2}$ and $\cE_{Z_2}^\vee \otimes \omega_X^{-1}$ are the corresponding universal families;
\newline$(ii)$
$\cM_3(X) \cong Z_3$ and the sheaves $\cE_{Z_3}$ and $\cE_{Z_3}^\vee \otimes \omega_X^{-1}$ are two universal families for this moduli problem.
\end{theorem}
\begin{proof}
Let $S$ be an arbitrary base scheme and $\cF$ a coherent sheaf on $X \times S$ flat over $S$ and such that 
for any geometric point $s \in S$ the restriction $\cF_s = i_{X \times s}^*(\cF)$ is a semistable (i.e., torsion free) sheaf with Hilbert polynomial~$h_d(t)$ for $d \in \{2,3\}$.
By Lemma~\ref{lemma:md-points} it follows that 
\begin{equation*}
\cF_s \cong \cE_z 
\end{equation*}
for a unique closed point $z \in Z_d$.

Consider the semiorthogonal decomposition
\begin{equation}\label{eq:dbx-s}
\bD(X \times S) = \langle \cA_{1S}, \cA_{2S}, \cA_{3S} \rangle = \langle \bD(S), \bD(Z_2 \times S), \bD(Z_3 \times S) \rangle,
\end{equation}
obtained by the base change $S \to \Spec(\kk)$ (\cite[Theorem~5.6]{kuznetsov2011base}) from~\eqref{eq:dbx-dz} 
(the second equality follows from \cite[Theorem~6.4]{kuznetsov2011base} and Proposition~\ref{proposition:fm-zd}).
Since by Proposition~\ref{proposition:fm-zd} the embedding functors of~\eqref{eq:dbx-dz} are the Fourier--Mukai functors given by the sheaves $\cE_{Z_d}$,
the embedding functors of~\eqref{eq:dbx-s} are given by the pullbacks $\cE_{Z_d} \boxtimes \cO_S$ of $\cE_{Z_d}$ via the maps $X \times Z_d \times S \to X \times Z_d$.

For any $\cF \in \bD(X \times S)$, a geometric point $s \in S$, and each $d' \in \{1,2,3\}$ we have by~\cite[(11)]{kuznetsov2011base} an isomorphism
\begin{equation*}
i_{Z_{d'} \times s}^*(\alpha_{d'S}(\cF)) \cong \alpha_{d'}(i_{X \times s}^*(\cF)) \in \bD(Z_{d'}),
\end{equation*}
where $\alpha_{d'} \colon \bD(X) \to \bD(Z_{d'})$ and $\alpha_{d'S} \colon \bD(X \times S) \to \bD(Z_{d'} \times S)$ are 
the projection functors of the semiorthogonal decompositions~\eqref{eq:dbx-dz} and~\eqref{eq:dbx-s} respectively.
Using this isomorphism for the sheaf~$\cF$ we started with, and taking into account Lemma~\ref{lemma:e-z-small}, we conclude that
\begin{equation*}
i_{Z_{d'} \times s}^*(\alpha_{d'S}(\cF)) \cong \alpha_{d'}(\cF_s) \cong \alpha_{d'}(\cE_z) \cong 
\begin{cases}
\cO_z, & \text{if $d' = d$},\\
0, & \text{otherwise}.
\end{cases}
\end{equation*} 
By Lemma~\ref{lemma:fiberwise}$(i)$ it follows that $\alpha_{d'S}(\cF) = 0$ when $d' \ne d$, hence 
\begin{equation*}
\cF \cong \Phi_{\cE_{Z_d} \boxtimes \cO_S}(\cF_d)
\end{equation*}
for some object $\cF_d \in \bD(Z_d \times S)$.
Moreover, the object $\cF_d$ is such that for any geometric point~$s \in S$ we have  $i_{Z_d \times s}^*(\cF_d) \cong \cO_z$ for some closed point~$z \in Z_d$.
By Lemma~\ref{lemma:fiberwise}$(iii)$ it follows that $\cF_d \cong \Gamma_{f_d*}\cL$, where $f_d \colon S \to Z_d$ is a morphism, $\Gamma_{f_d} \colon S \to Z_d \times S$ is its graph, and $\cL$ is a line bundle on $S$.
Therefore
\begin{equation*}
\cF \cong \Phi_{\cE_{Z_d} \boxtimes \cO_S}(\Gamma_{f_d*}\cL) 
\cong (\id_X \times f_d)^*\cE_{Z_d} \otimes p_S^*\cL.
\end{equation*}
This precisely means that the moduli functor $\cM_d(X)$ we are interested in is represented by the scheme~$Z_d$, 
and that $\cE_{Z_d}$ is a universal sheaf.

This proves the statement about~$\cM_2(X)$ and the first statement about~$\cM_3(X)$.
The same argument applied to the dual semiorthogonal decomposition~\eqref{eq:dbx-dual} instead of~\eqref{eq:dbx-dz} 
proves the statement about~$\cM_4(X)$ and the second statement about~$\cM_3(X)$.
\end{proof}

\begin{corollary}\label{corollary:tau-m3x}
There is an automorphism $\sigma \colon \cM_3(X) \to \cM_3(X)$ such that $\cE_{Z_3}^\vee \cong \sigma^*\cE_{Z_3} \otimes \omega_X$.
\end{corollary}
\begin{proof}
Both $\cE_{Z_3}$ and $\cE_{Z_3}^\vee \otimes \omega_X^{-1}$ are universal families on $\cM_3(X) \times X$, hence they differ only by an automorphism of $\cM_3(X)$ and a twist by a line bundle on it.
But since $\cM_3(X)$ is zero-dimensional, it has no non-trivial line bundles.
\end{proof}

\subsection{Hilbert scheme interpretation}\label{subsection:hilbert}

Now consider the polynomial 
\begin{equation}\label{eq:hdprime}
h'_d(t) = dt + 1 \in \ZZ[t].
\end{equation} 
This is the Hilbert polynomial of a rational normal curve of degree $d$.

\begin{lemma}
Let $C \subset X \subset \P^6$ be a subscheme with Hilbert polynomial $h_C(t) = h'_d(t)$ and $1 \le d \le 3$.
Then $C$ is a connected arithmetically Cohen--Macaulay curve.
\end{lemma}
\begin{proof}
In cases $d = 1$ and $d = 2$ this is standard (see, e.g., \cite[Lemma~2.1.1]{kuznetsov2016hilbert}).
In case $d = 3$ the only other possibility for $C$ (see~\cite[\S1]{lehn2015twisted} and references therein) would be a union of a plane cubic curve with a point (possibly embedded).
But~$X$ is an intersection of quadrics by~\cite[Theorem~4.4]{hidaka1981normal} and contains no planes, so this is impossible.
\end{proof}

\begin{lemma}\label{lemma:oc-decomposition}
$(i)$ If $C \subset X$ is a curve with Hilbert polynomial $h'_d(t)$ with $d \in \{2,3\}$ then there is a unique closed point $z \in Z_d$ and an exact sequence
\begin{equation*}
0 \to \cE_z^\vee \to \cO_X \to \cO_C \to 0.
\end{equation*}
Conversely, any nonzero morphism $\cE_z^\vee \to \cO_X$ is injective with cokernel isomorphic to $\cO_C$.

$(ii)$ Similarly, if $L \subset X$ is a line, i.e., a curve with Hilbert polynomial $h'_1(t)$ then there are unique closed points $z_2 \in Z_2$ and $z_3 \in Z_3$ and an exact sequence
\begin{equation*}
0 \to \cE_{z_2} \to \cE_{z_3} \to \cO_L \to 0.
\end{equation*}
Conversely, any nonzero morphism $\cE_{z_2} \to \cE_{z_3}$ is injective with cokernel isomorphic to $\cO_L$.
\end{lemma}
\begin{proof}
For the first part of~$(i)$, by Lemma~\ref{lemma:md-points}$(ii)$ it is enough to show 
that $I_C \otimes \omega_X^{-1}$ is stable with Hilbert polynomial~$h_{6-d}(t)$.
Stability is clear by Lemma~\ref{lemma:hd-stable}, and the Hilbert polynomial evidently equals 
\begin{equation*}
h_{\cO_X}(t+1) - h'_d(t+1) = (3(t+1)(t+2) + 1) - (d(t+1) + 1) = (3t + 6 - d)(t + 1) = h_{6-d}(t).
\end{equation*}

For the first part of~$(ii)$ first note that by Serre duality 
\begin{equation*}
\Ext^\bullet(\cO_L,\cO_X) \cong \Ext^\bullet(\cO_X,\cO_L(K_X))^\vee = H^\bullet(L,\cO_L(-1)) = 0 
\end{equation*}
since $L \cong \P^1$ and $L \cdot K_X = -1$.
Thus by Theorem~\ref{theorem:dbx} the structure sheaf $\cO_L$ is contained in the subcategory~$\langle \cA_2, \cA_3 \rangle$ of $\bD(X)$.
Next, again by Serre duality
\begin{equation*}
\Ext^i(\cO_L,\fE_{Z_2}) \cong \Ext^{2-i}(\fE_{Z_2},\cO_L(K_X))^\vee.
\end{equation*}
Since $\fE_{Z_2}$ is a vector bundle and $L$ is a curve, the right hand side is zero unless $i \in \{1, 2\}$.
On the other hand, $\fE_{Z_2}$ is globally generated by Proposition~\ref{proposition:e-z-big} and~\eqref{eq:cezd}, while $\cO_L(K_X) \cong \cO_L(-1)$ has no global sections, hence for $i = 2$ the right hand side is also zero.
On the other hand, by Lemma~\ref{lemma:chi-ez-f} we have
\begin{equation*}
\chi(\fE_{Z_2},\cO_L(K_X)) = 2\chi(X,\cO_L(K_X)) + \chi(X,\cO_L(2K_X)) - 3r(\cO_L(K_X)) = 2 \cdot h'_1(-1) + h'_1(-2) = -1,
\end{equation*}
and we conclude that $\Ext^\bullet(\cO_L, \fE_{Z_2}) = \kk[-1]$.
This means that there is a unique closed point $z_2 \in Z_2$ and a unique extension
\begin{equation*}
0 \to \cE_{z_2} \to \cF \to \cO_L \to 0
\end{equation*}
such that $\cF \in \cA_3$. 
It remains to note that $\cF$ is a sheaf of rank 1, hence there is a unique $z_3 \in Z_3$ such that $\cF \cong \cE_{z_3}$.

The converse statements are evident.
\end{proof}

Consider the Hilbert scheme
\begin{equation}
\label{def:fd}
F_d(X) := \Hilb_{X,-K_X}(h'_d)
\end{equation} 
of subschemes of $X$ with Hilbert polynomial $h'_d(t)$.
Thus, $F_1(X)$ is the Hilbert scheme of lines, $F_2(X)$ is the Hilbert scheme of conics, and $F_3(X)$ is the Hilbert scheme of generalized twisted cubic curves on $X$.

Recall the notation $p_2$, $p_3$, $p_{23}$, $p_{X2}$, and $p_{X3}$ introduced before Lemma~\ref{lemma:ps-ez}.
By Lemma~\ref{lemma:ps-ez} the sheaves $p_{2*}\cE_{Z_2}$, $p_{3*}\cE_{Z_3}$ and $p_{23*}(p_{X2}^*\cE_{Z_2}^\vee \otimes p_{X3}^*\cE_{Z_3})$ 
are locally free of ranks $2$, $3$ and $1$ on $Z_2$, $Z_3$ and~$Z_2 \times Z_3$ respectively.

\begin{proposition}\label{proposition:fd-zd}
We have natural isomorphisms of Hilbert schemes
\begin{equation*}
F_1(X) \cong Z_2 \times Z_3,
\qquad
F_2(X) \cong \P_{Z_2}(p_{2*}\cE_{Z_2}),
\qquand 
F_3(X) \cong \P_{Z_3}(p_{3*}\cE_{Z_3}).
\end{equation*}
\end{proposition}
\begin{proof}
Assume $d \in \{2,3\}$.
Let $\cC \subset X \times S$ be a flat $S$-family of subschemes in $X$ with Hilbert polynomial~$h'_d(t)$.
Consider the decomposition of the structure sheaf $\cO_\cC \in \bD(X \times S)$ with respect to the semiorthogonal decomposition 
\begin{equation*}
\bD(X \times S) = \langle \cA_{3S}^\vee, \cA_{2S}^\vee, \cA_{1S}^\vee \rangle,
\end{equation*}
obtained from the decomposition of Proposition~\ref{proposition:dbx-dual} 
by the base change $S \to \Spec(\kk)$ via~\cite[Theorem~5.6]{kuznetsov2011base}.
The argument of Theorem~\ref{theorem:moduli} together with the result of Lemma~\ref{lemma:oc-decomposition}$(i)$ shows 
that there is a morphism $f_d \colon S \to Z_d$, a line bundle $\cF_d$ on $S$, and an exact sequence
\begin{equation*}
0 \to (\id_X \times f_d)^*(\cE_{Z_d}^\vee) \otimes p_S^*(\cF_d) \to \cO_{X \times S} \to \cO_\cC \to 0
\end{equation*}
The left arrow corresponds to a global section of the bundle $(\id_X \times f_d)^*(\cE_{Z_d}) \otimes p_S^*(\cF_d^\vee)$, i.e., to a morphism from $p_S^*\cF_d$ to $(\id_X \times f_d)^*(\cE_{Z_d})$.
We have 
\begin{equation*}
\Hom(p_S^*\cF_d,(\id_X \times f_d)^*(\cE_{Z_d})) \cong
\Hom(\cF_d,p_{S*}((\id_X \times f_d)^*(\cE_{Z_d}))) \cong
\Hom(\cF_d,f_d^*p_{d*}(\cE_{Z_d})),
\end{equation*}
where the first equality follows from adjunction and the second is the base change for the diagram
\begin{equation*}
\xymatrix@C=5em{
X \times S \ar[r]^-{\id_X \times f_d} \ar[d]_{p_S} & X \times Z_d \ar[d]^{p_d} \\
S \ar[r]^-{f_d} & Z_d
}
\end{equation*}
The flatness over $S$ of the cokernel $\cO_\cC$ of the morphism $(\id_X \times f_d)^*(\cE_{Z_d}^\vee) \otimes p_S^*(\cF_d) \to \cO_{X \times S}$ 
is equivalent to the corresponding morphism $\cF_d \to f_d^*p_{d*}\cE_{Z_d}$ being nonzero for every closed point $s \in S$.
Thus the Hilbert scheme functor $F_d(X)$ is isomorphic to the functor that associates to a scheme $S$ a morphism $f_d \colon S \to Z_d$ and a line subbundle in $f_d^*p_{d*}\cE_{Z_d}$, 
hence is represented by the projective bundle $\P_{Z_d}(p_{d*}\cE_{Z_d})$.

Now assume $d = 1$ and let $\cC \subset X \times S$ be a flat family of lines.
Decomposing $\cO_\cC$ with respect to~\eqref{eq:dbx-s} 
and using the argument of Theorem~\ref{theorem:moduli} together with the result of Lemma~\ref{lemma:oc-decomposition}$(ii)$,
we conclude that there are morphisms $f_2 \colon S \to Z_2$ and $f_3 \colon S \to Z_3$,
line bundles $\cF_2$ and $\cF_3$ on $S$, and an exact sequence
\begin{equation}\label{eq:o-lines}
0 \to (\id_X \times f_{2})^*(\cE_{Z_2}) \otimes p_S^*(\cF_2) \to (\id_X \times f_{3})^*(\cE_{Z_3}) \otimes p_S^*(\cF_3) \to \cO_\cC \to 0.
\end{equation}
Moreover, the pushforward to $S$ of this sequence gives (again by base change) an exact sequence
\begin{equation*}
0 \to f_{2}^*(p_{2*}\cE_{Z_2}) \otimes \cF_2 \to f_{3}^*(p_{3*}\cE_{Z_3}) \otimes \cF_3 \to \cO_S \to 0.
\end{equation*}
Since by Lemma~\ref{lemma:ps-ez} we know that $f_{d}^*(p_{d*}\cE_{Z_d})$ is a vector bundle of rank $d$ on $S$, the comparison of determinants gives  a relation
\begin{equation}\label{eq:f23-relation}
\cF_3^3 \otimes \cF_2^{-2} \cong f_{2}^*(\det(p_{2*}\cE_{Z_2})) \otimes f_{3}^*(\det(p_{3*}\cE_{Z_3})^{-1}).
\end{equation}
Denoting by $f_{23} \colon S \to Z_2 \times Z_3$ the map induced by $f_2$ and $f_3$, we deduce by adjunction and base change 
\begin{equation*}
\Hom(f_{2}^*(p_{2*}\cE_{Z_2}) \otimes \cF_2, f_{3}^*(p_{3*}\cE_{Z_3}) \otimes \cF_3) \cong
\Hom(\cF_2 \otimes \cF_3^\vee,f_{23}^*p_{23*}(p_{X2}^*\cE_{Z_2}^\vee \otimes p_{X3}^*\cE_{Z_3}))
\end{equation*}
and note that $f_{23}^*p_{23*}(p_{X2}^*\cE_{Z_2}^\vee \otimes p_{X3}^*\cE_{Z_3})$ is a line bundle (again by Lemma~\ref{lemma:ps-ez}).
The flatness over~$S$ of the cokernel $\cO_\cC$ of the morphism of~\eqref{eq:o-lines} is equivalent to the pointwise injectivity 
of the corresponding morphism $\cF_2 \otimes \cF_3^\vee \to f_{23}^*p_{23*}(p_{X2}^*\cE_{Z_2}^\vee \otimes p_{X3}^*\cE_{Z_3})$,
i.e., to this map being an isomorphism.
On the other hand, any line bundle on $S$ can be written as the tensor product of line bundles $\cF_2 \otimes \cF_3^\vee$ satisfying~\eqref{eq:f23-relation} in a unique way.
This shows that the Hilbert scheme functor is isomorphic to the functor that associates to a scheme $S$ a pair of morphisms $f_2 \colon S \to Z_2$ and $f_3 \colon S \to Z_3$,
hence is represented by $Z_2 \times Z_3$.
\end{proof}

\begin{remark}
\label{remark:f4}
With the same argument one can prove that $p_{2*}(\cE_{Z_2}^\vee \otimes \omega_X^{-1})$ is a rank-4 vector bundle on~$Z_2$ 
and that~$F_4(X) \cong \P_{Z_2}(p_{2*}(\cE_{Z_2}^\vee \otimes \omega_X^{-1}))$.
\end{remark}

\section{Families of sextic del Pezzo surfaces}\label{section:families}

Starting from this section we assume that $\kk$ is an arbitrary field of characteristic distinct from~2 and~3, unless something else is specified explicitly.
We keep in mind the notation introduced in previous sections.

\subsection{Semiorthogonal decomposition}\label{subsection:derived-family}

The main result of this section is a description of the derived category of a du Val family (as defined below) of sextic del Pezzo surfaces.

\begin{definition}\label{definition:family-dp6}
A family $f \colon \cX \to S$ is a {\sf du Val family} of sextic del Pezzo surfaces, if $f$ is a flat projective morphism such that 
for every geometric point $s \in S$ the fiber $\cX_s$ of $\cX$ over $s$ is a sextic du Val del Pezzo surface (Definition~\ref{definition:dp6}), 
i.e., a normal integral surface with at worst du Val singularities such that $-K_{\cX_s}$ is an ample Cartier divisor and $K_{\cX_s}^2 = 6$.
\end{definition}

Note that this definition is much more general than the notion of a \emph{good family} used in~\cite{addington2016cubic}, 
since we do not assume any transversality.
Note also that by definition all fibers of $f \colon \cX \to S$ are Gorenstein, 
hence the relative dualizing complex $\omega_{\cX/S}^\bullet$ when restricted to any fiber of $f$ is an invertible sheaf, 
hence by Lemma~\ref{lemma:fiberwise} it is an invertible sheaf on the total space $\cX$.

If $f \colon \cX \to S$ is a du Val family of sextic del Pezzo surfaces then for any base change $S' \to S$ the induced family $f' \colon \cX' = \cX \times_S S' \to S'$ is still a du Val family.
So, du Val families of sextic del Pezzo surfaces form a stack over $(\mathrm{Sch}/\kk)$.
In Appendix~\ref{appendix:moduli-stack} we prove that this stack is smooth of finite type over~$\kk$ (but not separated), so in most arguments of this section one may safely assume 
that the base $S$ of the family is smooth and then deduce the necessary results for any family by a base change argument.

The main result of this section (and of the paper) is the following

\begin{theorem}
\label{theorem:sod-dp6}
Assume $f \colon \cX \to S$ is a du Val family of sextic del Pezzo surfaces.
Then there is an~$S$-linear semiorthogonal decomposition \textup{(}compatible with any base change\textup{)}
\begin{equation}
\label{eq:sod-dp6}
\bD(\cX) = \langle \bD(S), \bD(\cZ_2,\beta_{\cZ_2}), \bD(\cZ_3,\beta_{\cZ_3}) \rangle,
\end{equation}
where $\cZ_2 \to S$ and $\cZ_3 \to S$ are finite flat morphisms of degree $3$ and $2$ respectively, 
$\beta_{\cZ_2}$ and $\beta_{\cZ_3}$ are Brauer classes of order $2$ and $3$ respectively on them,
and the last two components are the twisted derived categories.
\end{theorem}

The embedding functors of the components of~\eqref{eq:sod-dp6} are defined in~\eqref{eq:phid-definition},
and a more precise version of the semiorthogonal decomposition is stated in~\eqref{eq:dbcx-again}.

Base change compatibility means that for a du Val family $\cX' \to S'$ of sextic del Pezzo surfaces obtained from a family $\cX \to S$ by a base change, 
the decomposition of the theorem coincides with the decomposition obtained from~\eqref{eq:sod-dp6} by~\cite[Theorem~5.6]{kuznetsov2011base}.

The proof of the theorem takes all Section~\ref{subsection:derived-family}, and in Section~\ref{subsection:properties-family} we discuss some properties of this semiorthogonal decomposition.

Let $f \colon \cX \to S$ be a du Val family of sextic del Pezzo surfaces.
For $d \in \{2,3,4\}$ let $\cM_d(\cX/S)$ denote the relative moduli stack of semistable sheaves on fibers of $\cX$ over $S$ with Hilbert polynomial $h_d(t)$ defined in~\eqref{eq:hd}.

\begin{proposition}
\label{proposition:md-zd}
For $d \in \{2,3,4\}$ the stack $\cM_d(\cX/S)$ is a $\Gm$-gerbe over its coarse module space 
\begin{equation*}
\cZ_d(\cX/S) := (\cM_d(\cX/S))_{\mathrm{coarse}}
\end{equation*}
with an obstruction given by a Brauer class $\beta_{\cZ_d(\cX/S)}$ of order
\begin{equation}
\label{eq:brauer-order}
\operatorname{ord} \beta_{\cZ_d(\cX/S)} = 
\begin{cases}
2, & \text{if $d = 2$ or $d = 4$},\\
3, & \text{if $d = 3$}.
\end{cases}
\end{equation}
Moreover, there is an isomorphism of the coarse moduli spaces $\sigma_{2,4} \colon \cZ_4(\cX/S) \xrightarrow{\ \sim\ } \cZ_2(\cX/S)$ and
an automorphism $\sigma_{3,3} \colon \cZ_3(\cX/S) \xrightarrow{\ \sim\ } \cZ_3(\cX/S)$, 
such that 
\begin{equation*}
\beta_{\cZ_4(\cX/S)} = \sigma_{2,4}^*(\beta_{\cZ_2(\cX/S)}^{-1}) = \sigma_{2,4}^*(\beta_{\cZ_2(\cX/S)})
\qquand
\beta_{\cZ_3(\cX/S)} = \sigma_{3,3}^*(\beta_{\cZ_3(\cX/S)}^{-1}).
\end{equation*}
\end{proposition}
\begin{proof}
By Lemma~\ref{lemma:hd-stable} all sheaves parameterized by the moduli stack $\cM_d(\cX/S)$ are strictly stable.
Therefore, by~\cite[Theorem~4.3.7]{huybrechts1997geometry}, the coarse moduli space ${\cZ_d(\cX/S)}$ exists and 
there is a quasiuniversal family $\cE_{\cZ_d(\cX/S)}$ on the fiber product $\cX \times_S \cZ_d{(\cX/S)}$ which has a module structure over
a sheaf of Azumaya algebras $\cB_{\cZ_d(\cX/S)}$ (defined up to a Morita equivalence) 
of order equal to the greatest common divisor of the values of the Hilbert polynomial $h_d(t)$ for $t \in \ZZ$.

This means that $\cM_d(\cX/S)$ is a $\Gm$-gerbe over $\cZ_d(\cX/S)$.
Its obstruction class is given by the Brauer class $\beta_{\cZ_d(\cX/S)}$ of the Azumaya algebra $\cB_{\cZ_d(\cX/S)}$, 
and a simple computation of the greatest common divisor of the values of the Hilbert polynomial $h_d(t)$ gives~\eqref{eq:brauer-order}.

By Theorem~\ref{theorem:moduli} the sheaf $\cE^\vee_{\cZ_2(\cX/S)} \otimes \omega_{\cX/S}^{-1}$ on $\cX \times_S \cZ_2(\cX/S)$ provides a family of stable sheaves with Hilbert polynomial $h_4(t)$
and the sheaf $\cE^\vee_{\cZ_4(\cX/S)} \otimes \omega_{\cX/S}^{-1}$ on $\cX \times_S \cZ_4(\cX/S)$ provides a family of stable sheaves with Hilbert polynomial $h_2(t)$.
Therefore, the moduli stacks $\cM_2(\cX/S)$ and $\cM_4(\cX/S)$ are isomorphic.
Denoting by $\sigma_{2,4}$ the induced isomorphism of the coarse moduli spaces, we conclude that 
the family $\sigma_{2,4}^*(\cE^\vee_{\cZ_2(\cX/S)} \otimes \omega_{\cX/S}^{-1})$ is a quasiuniversal family on $\cX \times_S \cZ_4(\cX/S)$.
It follows that the pullback of the opposite Azumaya algebra $\cB_{\cZ_2(\cX/s)}$ to $\cZ_4(\cX/S)$ is Morita-equivalent 
to the Azumaya algebra $\cB_{\cZ_4(\cX/S)}$, hence $\beta_{\cZ_4(\cX/S)} = \sigma_{2,4}^*(\beta_{\cZ_2(\cX/S)}^{-1})$.
Since by~\eqref{eq:brauer-order} the order of $\beta_{\cZ_2(\cX/S)}$ is~2, 
this can be also written as~$\sigma_{2,4}^*(\beta_{\cZ_2(\cX/S)}^{-1})$.

The second isomorphism $\sigma_{3,3}$ is constructed in the same way, and with its construction we also get an isomorphism of quasiuniversal families and an equality of Brauer classes.
\end{proof}

When there is no risk of confusion, we abbreviate $\cM_d(\cX/S)$ and $\cZ_d(\cX/S)$ to $\cM_d$ and $\cZ_d$.
Both the stack $\cM_d$ and the scheme $\cZ_d$ are proper over $S$.
We denote by
\begin{equation*}
f_d \colon \cZ_d \to S
\end{equation*}
the natural projection.
We replace the Azumaya algebra $\cB_{\cZ_d}$ by its Brauer class $\beta_{\cZ_d} \in \Br(\cZ_d)$ and 
consider the quasiuniversal family as a~$\beta^{-1}_{\cZ_d}$-twisted family of sheaves on $\cX \times_S \cZ_d$:
\begin{equation*}
\cE_{\cZ_d} \in \coh(\cX \times_S \cZ_d, \beta^{-1}_{\cZ_d}).
\end{equation*}
We will use these sheaves to construct Fourier--Mukai functors in~\eqref{eq:phid-definition}.
Note that the relation between the quasiuniversal families discussed in Proposition~\ref{proposition:md-zd} in terms of twisted universal families means that
there are line bundles $\cL_{2,4}$ and $\cL_{3,3}$ on $\cZ_4(\cX/S)$ and $\cZ_3(\cX/S)$ respectively, such that
\begin{equation*}
\cE_{\cZ_4(\cX/S)} = \sigma_{2,4}^*(\cE_{\cZ_2(\cX/S)}^{\vee} \otimes \omega_{\cX/S}^{-1}) \otimes \cL_{2,4}
\qquand
\cE_{\cZ_3(\cX/S)} = \sigma_{3,3}^*(\cE_{\cZ_3(\cX/S)}^{\vee} \otimes \omega_{\cX/S}^{-1}) \otimes \cL_{3,3}.
\end{equation*}


\begin{lemma}
\label{lemma:md-zd-base-change}
Let $\phi \colon S' \to S$ be a base change and denote $\cX' := \cX \times_S S'$.
We have
\begin{equation}
\label{eq:md-zd-base-change}
\cM_d(\cX'/S') \cong \cM_d(\cX/S) \times_S S'
\qquand
\cZ_d(\cX'/S') \cong \cZ_d(\cX/S) \times_S S'.
\end{equation} 
Moreover, isomorphisms~\eqref{eq:md-zd-base-change} are compatible with the universal families, i.e., 
\begin{equation}\label{eq:ced-base-change}
\beta_{\cZ'_d} = \phi_{\cZ_d}^*(\beta_{\cZ_d})
\qquand
\cE_{\cZ'_d} \cong \phi_{\cX \times_S \cZ_d}^*\cE_{\cZ_d}.
\end{equation}
where we set $\cZ'_d := \cZ_d(\cX'/S')$ and denote the morphisms induced by the base change
by $\phi_\cX \colon \cX' \to \cX$, $\phi_{\cZ_d} \colon \cZ'_d \to \cZ_d$, 
and $\phi_{\cX \times_S \cZ_d} \colon \cX' \times_{S'} \cZ'_d \to \cX \times_S \cZ_d$, respectively.
\end{lemma}
\begin{proof}
The first isomorphism in~\eqref{eq:md-zd-base-change} is clear from the definition of a relative moduli space, and the second follows from the GIT construction of the coarse moduli space.
Since the pullback of a universal family is a $\phi_{\cZ_d}^*(\beta_{\cZ_d})$-twisted family of stable sheaves on fibers of $\cX'$ over $S'$, the equality of the Brauer classes and 
the isomorphism of universal families follow.
\end{proof}

Considering base changes to geometric points of $S$ and using Theorem~\ref{theorem:moduli}, we obtain

\begin{corollary}
\label{corollary:S-s}
For any geometric point $s \in S$ there are isomorphisms
\begin{equation}
\label{eq:m-z-e-s}
\cM_d(\cX/S)_s \cong \cM_d(X),
\qquad
\cZ_d(\cX/S)_s \cong Z_d
\qquand
i_{Z_d}^*(\cE_{\cZ_d}) \cong \cE_{Z_d},
\end{equation}
where $X = \cX_s$ is the fiber of $\cX$ over $s$ and $i_{Z_d} \colon Z_d \cong \cZ_d(\cX/S)_s \hookrightarrow \cZ_d(\cX/S)$ 
is the natural embedding.
\end{corollary}

In the next Lemma we consider the coarse moduli spaces $\cZ_2$ and $\cZ_3$.
Analogous result for $\cZ_4$ follows via the isomorphism $\sigma_{2,4}$ of Proposition~\ref{proposition:md-zd}.

\begin{lemma}
\label{lemma:zd-flat}
The maps $f_2 \colon \cZ_2 \to S$ and $f_3 \colon \cZ_3 \to S$ are finite flat maps of degree $3$ and $2$ respectively.
The relative dualizing complexes $\omega_{\cZ_2/S}^\bullet$ and $\omega_{\cZ_3/S}^\bullet$ are line bundles.
\end{lemma}
\begin{proof}
Since flatness can be verified on an \'etale covering (note that formation of schemes $\cZ_d$ 
is compatible with base changes by~\eqref{eq:md-zd-base-change})
and the moduli stack of du Val sextic del Pezzo surfaces is smooth by Theorem~\ref{theorem:moduli-stack-dp6}, 
we may assume that $S$ is smooth over $\kk$, hence reduced.
Since $\cZ_d$ is proper over $S$, it is enough to show that the (scheme-theoretic) fiber of $\cZ_d$ 
over any geometric point $\bs \in S$ is zero-dimensional of length~$6/d$.
But as it was mentioned in Corollary~\ref{corollary:S-s}, the fiber $(\cZ_d)_s$ is identified with the zero-dimensional scheme~$Z_d$ 
associated with the surface $X = \cX_s$, and hence its length is indeed~$6/d$, see Corollary~\ref{corollary:dbx-z2-z3}.

Since $f_d$ is flat, to check that the relative dualizing complex is a line bundle, 
it is enough to check that each fiber of $f_d$ is a Gorenstein scheme, 
which holds true by~\eqref{eq:m-z-e-s} and Corollary~\ref{corollary:dbx-z2-z3}.
\end{proof}

For convenience we also define $\cZ_1 = S$, set $f_1 \colon \cZ_1 \to S$ to be the identity map, 
set $\beta_{\cZ_1}$ to be the trivial Brauer class, and $\cE_{\cZ_1} = \cO_\cX$.
Let $p \colon \cX \times_S \cZ_d \to \cX$ and $p_d \colon \cX \times_S \cZ_d \to \cZ_d$ be the projections.
Then we have a fiber square
\begin{equation}
\label{diagram:x-s-z}
\vcenter{\xymatrix@C=5em{
\cX \times_S \cZ_d \ar[r]^{p_d} \ar[d]_{p} 
& \cZ_d \ar[d]^{f_d} \\
\cX \ar[r]_-f & S
}}
\end{equation}
In what follows we work with Fourier--Mukai functors between twisted derived categories.
Since the Brauer classes we consider come from Azumaya algebras, one can consider those twisted derived categories as derived categories of \emph{Azumaya varieties}, as defined in~\cite[Appendix~A]{kuznetsov2006hyperplane};
for instance, as explained in \emph{loc.\ cit.}, we have all standard functors between these varieties and all standard functorial isomorphisms.

\begin{lemma}
\label{lemma:ezd-tor-ext}
For each $d \in \{1,2,3\}$ the sheaf $\cE_{\cZ_d}$ is flat over $\cZ_d$ and has finite $\Ext$-amplitude over $\cX$.
\end{lemma}
\begin{proof}
For $d \in \{2,3\}$ the sheaf $\cE_{\cZ_d}$ is flat over $\cZ_d$ since it comes from a quasiuniversal family for a moduli problem.
Since $p$ is a finite map (Lemma~\ref{lemma:zd-flat}), to check that $\cE_{\cZ_d}$ has finite $\Ext$-amplitude over $\cX$, 
it is enough (\cite[Lemma~10.40]{kuznetsov2006hyperplane}) to show that $p_*\cE_{\cZ_d}$ has finite $\Ext$-amplitude on $\cX$.
But this is in fact a locally free sheaf by Lemma~\ref{lemma:fiberwise} and Proposition~\ref{proposition:e-z-big}.

In case $d = 1$ both properties are evident.
\end{proof}

For each~$d \in \{1,2,3\}$ we consider the Fourier--Mukai functor whose kernel is the universal family $\cE_{\cZ_d}$, considered as a $\beta_{\cZ_d}^{-1}$-twisted sheaf on $\cX  \times_S \cZ_d$:
\begin{equation}
\label{eq:phid-definition}
\Phi_d = \Phi_{\cE_{\cZ_d}} := p_*(\cE_{\cZ_d} \otimes p_d^*(-)) \colon \bD(\cZ_d,\beta_{\cZ_d}) \to \bD(\cX).
\end{equation}

\begin{lemma}
\label{lemma:phi-d-shriek}
The functor $\Phi_d$ is $S$-linear, preserves boundedness and perfectness, and its right adjoint functor is
\begin{equation}
\label{eq:phid-shriek}
\Phi_d^! \colon \bD(\cX) \to \bD(\cZ_d,\beta_{\cZ_d}),\qquad
\Phi_d^!(\cG) \cong p_{d*}\RCHom(\cE_{\cZ_d},p^!\cG).
\end{equation}
\end{lemma}
\begin{proof}
This holds by Lemma~\ref{lemma:ezd-tor-ext} and \cite[Lemma~2.4]{kuznetsov2006hyperplane}.
\end{proof}

Our goal is to show that the functors $\Phi_d$ are fully faithful and that 
\begin{equation}
\label{eq:dbcx-again}
\bD(\cX) = \langle \Phi_1(\bD(\cZ_1,\beta_{\cZ_1})), \Phi_2(\bD(\cZ_2,\beta_{\cZ_2})), \Phi_3(\bD(\cZ_3,\beta_{\cZ_3})) \rangle
\end{equation}
is an $S$-linear semiorthogonal decomposition (this is a more precise version of~\eqref{eq:sod-dp6}).

It is convenient to rewrite the functors $\Phi_d^!$ in a Fourier--Mukai form.

\begin{lemma}
\label{lemma:phid-shriek-other}
We have an isomorphism of functors
\begin{equation*}
\Phi_d^!(\cG) \cong p_{d*}(\cE_{\cZ_d}^\vee \otimes p_d^*\omega_{\cZ_d/S} \otimes p^*(\cG)),
\end{equation*}
where $\cE_{\cZ_d}^\vee := \RCHom(\cE_{\cZ_d},\cO_{\cX \times_S \cZ_d})$ is a coherent sheaf flat over $\cZ_d$.
The kernel~$\cE_{\cZ_d}^\vee \otimes p_d^*\omega_{\cZ_d/S}$ of this functor is compativle with base changes, i.e.,
if $\phi \colon S' \to S$ is a base change, then
\begin{equation*}
\phi_{\cX \times_S \cZ_d}^*(\cE_{\cZ_d}^\vee \otimes p_d^*\omega_{\cZ_d/S}) \cong \cE_{\cZ'_d}^\vee \otimes p_d^{\prime*}\omega_{\cZ'_d/S'}.
\end{equation*}
\end{lemma}
\begin{proof}
For the first part it is enough to show that 
\begin{equation}
\label{eq:isomorphism-functors}
\RCHom(\cE_{\cZ_d},p^!\cG) \cong \cE_{\cZ_d}^\vee \otimes p_d^*\omega_{\cZ_d/S} \otimes p^*(\cG).
\end{equation}
For this we use an argument of Neeman from~\cite[Theorem~5.4]{neeman1996grothendieck}.

First, the functor in the right hand side of~\eqref{eq:isomorphism-functors} 
commutes with arbitrary direct sums since the pullback and the tensor product functors do.
The functor in the left hand side is right adjoint to the functor~$p_*(\cE_{\cZ_d} \otimes (-))$.
The latter functor preserves perfectness by an argument analogous to that of Lemma~\ref{lemma:phi-d-shriek},
hence the former commutes with direct sums by~\cite[Theorem~5.1]{neeman1996grothendieck}.

Further, if $\cG$ is a perfect complex, then $p^!(\cG) \cong p^*(\cG) \otimes p_d^*\omega^\bullet_{\cZ_d/S}$.
By Lemma~\ref{lemma:zd-flat} the dualizing complex~$\omega^\bullet_{\cZ_d/S}$ is a line bundle $\omega_{\cZ_d/S}$, 
hence $p^!(\cG)$ is a perfect complex, and~\eqref{eq:isomorphism-functors} in this case follows.

Now the Neeman's trick comes.
There is a natural transformation from the functor in the right hand side of~\eqref{eq:isomorphism-functors} 
to the functor in the left hand side.
The subcategory of objects on which this transformation is an isomorphism is 
a triangulated subcategory of the unbounded derived category of $\cX$ 
which contains all perfect complexes and is closed under arbitrary direct sums, hence is the whole category.
This proves~\eqref{eq:isomorphism-functors} for all $\cG$.

The fact that $\cE_{\cZ_d}^\vee$ is a sheaf flat over $\cZ_d$ follows from the isomorphisms
in the proof of Proposition~\ref{proposition:md-zd} since $\omega_{\cX/S}$ is a line bundle.
Using Lemma~\ref{lemma:md-zd-base-change} we deduce the required base change isomorphism.
\end{proof}

Given a base change~$\phi \colon S' \to S$, we abbreviate the Fourier--Mukai functor $\Phi_{\cE_{\cZ'_d}} \colon \bD(\cZ'_d) \to \bD(\cX')$ as~$\Phi'_d$.
We have the following property.

\begin{lemma}
\label{lemma:phid-linear}
The functors $\Phi_d$ and $\Phi_d^!$ are $S$-linear and compatible with base changes, i.e.,
\begin{align*}
\Phi'_d \circ \phi_{\cZ_d}^* &\cong \phi_{\cX}^* \circ \Phi_d,
&
\phi_{\cX*} \circ \Phi'_d &\cong \Phi_d \circ \phi_{\cZ_d*},
\\
\Phi^{\prime!}_d \circ \phi_{\cX}^* &\cong \phi_{\cZ_d}^* \circ \Phi_d^!,
&
\phi_{\cZ_d*} \circ \Phi^{\prime!}_d &\cong \Phi_d^! \circ \phi_{\cX*}.
\end{align*}
\end{lemma}
\begin{proof}
For a Fourier--Mukai functor compatibilities with base changes are proved in~\cite[Lemma~2.42]{kuznetsov2006hyperplane}
(the assumption of finiteness of $\Tor$-dimension of the base change morphism is only used in the second half of that lemma).
It remains to note that both $\Phi_d$ and $\Phi_d^!$ are Fourier--Mukai functors (the first by definition~\eqref{eq:phid-definition} and the second by Lemma~\ref{lemma:phid-shriek-other}).
\end{proof}

Now we are ready to prove the theorem.

\begin{proof}[Proof of Theorem~\textup{\ref{theorem:sod-dp6}}]
Take $d_1,d_2 \in \{1,2,3\}$ and consider the diagram
\begin{equation*}
\xymatrix{
& \cZ_{d_1} \times_S \cX \times_S \cZ_{d_2} 
\ar[dl]_{\pr_{1,2}} \ar[d]^{\pr_{1,3}} \ar[dr]^{\pr_{2,3}} \\
\cZ_{d_1} \times_S \cX & \cZ_{d_1} \times_S \cZ_{d_2} & \cX \times_S \cZ_{d_2} 
}
\end{equation*}
A standard computation shows that the composition of functors $\Phi_{d_2}^! \circ \Phi_{d_1}$ is a Fourier--Mukai functor given by the object
\begin{equation}
\label{eq:convolution}
\pr_{1,3*}(\pr_{2,3}^*(\cE_{\cZ_{d_2}}^\vee \otimes p_{d_2}^*\omega_{\cZ_{d_2}/S}) \otimes \pr_{1,2}^*(\cE_{\cZ_{d_1}})) \in \bD(\cZ_{d_1} \times_S \cZ_{d_2}, \beta_{\cZ_{d_1}}^{-1} \boxtimes \beta_{\cZ_{d_2}}).
\end{equation}

Let us prove that $\Phi_d$ is fully faithful for each $d \in \{1,2,3\}$.
For this it is enough to check that the object~\eqref{eq:convolution} in case $d_1 = d_2 = d$ is isomorphic 
(via the unit of adjunction morphism) to the structure sheaf of the diagonal.
Since the unit of the adjunction is induced by a morphism of kernels \cite{anno2012adjunction}, 
by Lemma~\ref{lemma:fiberwise} it is enough to check that for any geometric point $s \in S$ we have
\begin{equation}
\label{eq:isomorphism-diagonal}
i_{Z_d \times Z_d}^*(\pr_{1,3*}(\pr_{2,3}^*(\cE_{\cZ_{d}}^\vee \otimes p_{d}^*\omega_{\cZ_{d}/S}) \otimes \pr_{1,2}^*(\cE_{\cZ_{d}}))) \cong \delta_*\cO_{Z_d} \in \bD(Z_d \times Z_d)
\end{equation}
where $i_{Z_d \times Z_d} \colon Z_d \times Z_d \to \cZ_d \times_S \cZ_d$ is the natural embedding 
and $\delta \colon Z_d \to Z_d \times Z_d$ is the diagonal.
Using base change and isomorphisms of Lemma~\ref{lemma:md-zd-base-change} and Lemma~\ref{lemma:phid-shriek-other}, 
we can rewrite the left hand side of~\eqref{eq:isomorphism-diagonal} as the Fourier--Mukai kernel of the functor
\begin{equation*}
\bD(Z_d) \xrightarrow{\ \Phi_{\cE_{Z_d}}\ } \bD(X) \xrightarrow{\ \Phi_{\cE_{Z_d}}^!\ } \bD(Z_d),
\end{equation*}
where the functor $\Phi_{\cE_{Z_d}}$ is identified in Proposition~\ref{proposition:fm-zd} 
with the equivalence $\bD(Z_d) \cong \cA_d \subset \bD(X)$, and~$\Phi_{\cE_{Z_d}}^!$ is its right adjoint.
This composition is isomorphic to the identity, since the functor $\Phi_{\cE_{Z_d}}$ is fully faithful, 
hence we have~\eqref{eq:isomorphism-diagonal}.
Therefore, the functor $\Phi_d$ is fully faithful.

Next, let us prove that the subcategories $\Phi_d(\bD(\cZ_d,\beta_{\cZ_d})) \subset \bD(\cX)$ for $1 \le d \le 3$ are semiorthogonal.
For this it is enough to check that the object~\eqref{eq:convolution} in case $d_1 < d_2$ is zero.
Again, using Lemma~\ref{lemma:fiberwise} and base change isomorphisms of Lemma~\ref{lemma:md-zd-base-change} 
and Lemma~\ref{lemma:phid-shriek-other} we reduce to the case of $S = \Spec(\kk)$ with algebraically closed $\kk$.
In the latter case, the required vanishing follows from semiorthogonality of the subcategories 
$\Phi_{\cE_{Z_{d_i}}}(\bD(Z_{d_i})) = \cA_{d_i}$ in $\bD(X)$.

Finally, let us prove that the subcategories $\Phi_d(\bD(\cZ_d,\beta_{\cZ_d})) \subset \bD(\cX)$ for $1 \le d \le 3$ generate $\bD(\cX)$.
Take any $\cG \in \bD(\cX)$ and set 
\begin{equation*}
\cG_2 := \Cone(\Phi_3\Phi_3^!\cG \to \cG),
\qquad
\cG_1 := \Cone(\Phi_2\Phi_2^!\cG_2 \to \cG_2),
\qquand
\cG_0 := \Cone(\Phi_1\Phi_1^!\cG_1 \to \cG_1).
\end{equation*}
From full faithfulness and semiorthogonality it easily follows that $\Phi_1^!(\cG_0) = \Phi_2^!(\cG_0) = \Phi_3^!(\cG_0) = 0$.
Then Lemma~\ref{lemma:phid-linear} implies that for any geometric point $s \in S$ setting $\cG_{0s} = i_{X}^*(\cG_0) \in \bD(X)$ to be the restriction of $\cG_0$ to the fiber $X = \cX_s$, 
we have $\Phi_d^!(\cG_{0s}) = 0$ for all $d$.
By semiorthogonal decomposition~\eqref{eq:dbx-dz} this means that $\cG_{0s} = 0$.
Hence by Lemma~\ref{lemma:fiberwise}$(i)$ we have $\cG_0 = 0$.
Thus we have a chain of morphisms
\begin{equation*}
\cG =: \cG_3 \to \cG_2 \to \cG_1 \to \cG_0 = 0
\end{equation*}
with $\Cone(\cG_{d} \to \cG_{d-1}) \in \Phi_d(\bD(\cZ_d,\beta_{\cZ_d}))$, which proves the required semiorthogonal decomposition.

This semiorthogonal decomposition is $S$-linear since by Lemma~\ref{lemma:phi-d-shriek} its embedding functors~$\Phi_d$ are $S$-linear,
and its compatibility with base changes follows from~\cite[Theorem~6.4]{kuznetsov2011base} 
together with Lemma~\ref{lemma:phid-linear} and Lemma~\ref{lemma:md-zd-base-change}.
\end{proof}

\subsection{Some properties of the semiorthogonal decomposition}\label{subsection:properties-family}

Now we list some properties of the semiorthogonal decomposition of Theorem~\ref{theorem:sod-dp6}.

\begin{proposition}\label{proposition:sod-dp6-perfect-unbounded}
The components of the semiorthogonal decomposition~\eqref{eq:sod-dp6} are admissible and their projection functors have finite cohomological amplitude.
Moreover, the functors $\Phi_d$ preserve perfectness and induce a semiorthogonal decomposition
\begin{equation}
\label{eq:sod-dpx-perf}
\bD^{\mathrm{perf}}(\cX) = \langle \Phi_1(\bD^{\mathrm{perf}}(S)), \Phi_2(\bD^{\mathrm{perf}}(\cZ_2,\beta_{\cZ_2})), \Phi_3(\bD^{\mathrm{perf}}(\cZ_3,\beta_{\cZ_3})) \rangle.
\end{equation}
\end{proposition}
\begin{proof}
The functor $\Phi_d$ has finite cohomological amplitude because the sheaf $\cE_{\cZ_d}$ is flat over $\cZ_d$ by Lemma~\ref{lemma:ezd-tor-ext}
(in fact, since $p \colon \cX \times_S \cZ_d \to \cX$ is finite, the functor $\Phi_d$ is exact).
The functor $\Phi_d^!$ has finite cohomological amplitude because the sheaf $\cE_{\cZ_d}$ has finite $\Ext$-amplitude over $\cZ_d$ by Lemma~\ref{lemma:ezd-tor-ext}
(in fact, since $p_d \colon \cX \times_S \cZ_d \to \cZ_d$ is flat of relative dimension 2, the cohomological amplitude of the functor $\Phi_d^!$ equals $(0,2)$).
Thus $\Phi_d \circ \Phi_d^!$ has finite cohomological amplitude.
Since the projection functors of~\eqref{eq:sod-dp6} can be expressed as combinations of these functors
(see the proof of Theorem~\ref{theorem:sod-dp6}), 
it follows that the projection functors also have finite cohomological amplitude.

To show that the components of~\eqref{eq:sod-dp6} are admissible, note that we have two more decompositions
\begin{equation}
\label{eq:more-sod-cx}
\begin{aligned}
\bD(\cX) &= \langle \bD(\cZ_3,\beta_{\cZ_3}) \otimes \omega_{\cX/S}, \bD(S), \bD(\cZ_2, \beta_{\cZ_2}) \rangle,
\\
\bD(\cX) &= \langle \bD(\cZ_2, \beta_{\cZ_2}) \otimes \omega_{\cX/S}, \bD(\cZ_3,\beta_{\cZ_3}) \otimes \omega_{\cX/S}, \bD(S) \rangle.
\end{aligned}
\end{equation} 
Indeed, they can be established by the same argument as~\eqref{eq:sod-dp6} 
starting with the semiorthogonal decompositions~\eqref{eq:dbx-other} instead of~\eqref{eq:dbx}.
Since $\omega_{\cX/S}$ is invertible, the two above decompositions together with~\eqref{eq:sod-dp6} 
prove admissibility of all components (each of them appears on the left in one decomposition and on the right in another).

The functors $\Phi_d$ preserve perfectness by Lemma~\ref{lemma:phi-d-shriek}.
It remains to prove~\eqref{eq:sod-dpx-perf}.
For this note that the functors $\Phi_d$ have left adjoints $\Phi_d^*$ because the components of~\eqref{eq:sod-dp6} are admissible.
Since the functors $\Phi_d$ commute with arbitrary direct sums, it follows from~\cite[Theorem~5.1]{neeman1996grothendieck}
that the functors $\Phi_d^*$ preserve perfectness.
Since the projection functors of~\eqref{eq:sod-dp6} can be expressed as combinations of $\Phi_d \circ \Phi_d^*$,
they preserve perfectness, hence a restriction of~\eqref{eq:sod-dp6} gives~\eqref{eq:sod-dpx-perf}.
%
%
%
\end{proof}

Since the morphism $f \colon \cX \to S$ is Gorenstein, the relative dualizing complex $\omega_{\cX/S}$ is (up to a shift) a line bundle.
Therefore, the relative duality functor 
\begin{equation*}
\rD_{\cX/S}(\cF) := \RCHom(\cF , f^*\omega_S^\bullet)
\end{equation*}
is an anti-autoequivalence of the category $\bD(\cX)$ (see Section~\ref{subsection:duality}).

\begin{proposition}
\label{proposition:sod-dual}
The relative duality functor gives a semiorthogonal decomposition
\begin{equation}
\label{eq:sod-dp6-dual}
\bD(\cX) = \langle \rD_{\cX/S}(\bD(\cZ_3,\beta_{\cZ_3})), \rD_{\cX/S}(\bD(\cZ_2,\beta_{\cZ_2})), \rD_{\cX/S}(\bD(S)) \rangle,
\end{equation}
whose components are equivalent to $\bD(\cZ_3,\beta_{\cZ_3})$, $\bD(\cZ_2,\beta_{\cZ_2})$, and $\bD(S)$ respectively.
Moreover, this decomposition is right mutation-dual to~\eqref{eq:sod-dp6}.
\end{proposition}
\begin{proof}
Since $\rD_{\cX/S}$ is an anti-autoequivalence, \eqref{eq:sod-dp6-dual} is a semiorthogonal decomposition.
Consider also the right mutation-dual decomposition
\begin{equation*}
\bD(\cX) = \langle \LL_1(\LL_2(\Phi_3(\bD(\cZ_3,\beta_{\cZ_3}))), \LL_1(\Phi_2(\bD(\cZ_2,\beta_{\cZ_2}))), \Phi_1(\bD(S)) \rangle,
\end{equation*}
where $\LL_1$ and $\LL_2$ are the mutation functors through the first and the second components of~\eqref{eq:dbcx-again} 
(they are well-defined because the components of the decomposition are admissible).
By base change together with Proposition~\ref{proposition:dbx-dual} and Lemma~\ref{lemma:fiberwise}, these two decompositions coincide.
In particular, the components of~\eqref{eq:sod-dp6-dual} are equivalent 
to $\bD(\cZ_3,\beta_{\cZ_3})$, $\bD(\cZ_2,\beta_{\cZ_2})$, and $\bD(S)$ respectively.
\end{proof}

\subsection{Geometric implications}

The following regularity criterion is very useful.

\begin{proposition}
\label{proposition:smoothness}
Let $\cX \to S$ be a du Val family of sextic del Pezzo surfaces.
The total space $\cX$ of the family is regular if and only if all three schemes $S$, $\cZ_2$ and $\cZ_3$ are regular.
\end{proposition}
\begin{proof}
Consider decompositions~\eqref{eq:dbcx-again} and~\eqref{eq:sod-dpx-perf}.
If $\cX$ is regular then $\bD^\perf(\cX) = \bD(\cX)$, hence it follows that $\bD^\perf(S) = \bD(S)$ and $\bD^\perf(\cZ_d,\beta_{\cZ_d}) = \bD(\cZ_d,\beta_{\cZ_d})$ for $d = 2,3$,
which means that $S$ and $\cZ_d$ are regular. 
The other implication is analogous.
\end{proof}

\begin{remark}
\label{remark:smoothness}
Note also that a du Val family of sextic del Pezzo surfaces $\cX \to S$ is smooth over $S$ 
if and only if the maps $\cZ_2 \to S$ and $\cZ_3 \to S$ are smooth (i.e.\ unramified).
Indeed, this follows from Corollary~\ref{corollary:S-s} 
and the description of the fibers~$Z_d$ of $\cZ_d \to S$ in Corollary~\ref{corollary:dbx-z2-z3}.
\end{remark}

One of the consequences of the regularity criterion is the following.

\begin{corollary}\label{corollary:birational-isomorphism-zd}
Let $\cX \to S$ and $\cX' \to S$ be two du Val families of sextic del Pezzo surfaces with regular total spaces $\cX$ and $\cX'$.
Assume there is a dense open subset $S_0 \subset S$ such that for some $d \in \{2,3\}$ there is 
an $S_0$-isomorphism $\varphi_0 \colon \cZ_d(\cX_0/S_0) \xrightarrow{\ \sim\ }\cZ_d(\cX'_0/S_0)$, where $\cX_0 = \cX \times_S S_0$ and~$\cX'_0 = \cX' \times_S S_0$.
Then~$\varphi_0$ extends to an isomorphism $\varphi \colon \cZ_d(\cX/S) \xrightarrow{\ \sim\ } \cZ_d(\cX'/S)$ over $S$.
Moreover, if we have an equality of Brauer classes $\beta_{\cZ_d(\cX_0/S_0)} = \varphi_0^*\beta_{\cZ_d(\cX'_0/S_0)}$ then it extends to $\beta_{\cZ_d(\cX/S)} = \varphi^*\beta_{\cZ_d(\cX'/S)}$.
\end{corollary}
\begin{proof}
For the first part note, that $\cZ_d(\cX/S)$ is regular by Proposition~\ref{proposition:smoothness}, and in particular normal.
Hence it is isomorphic to the normal closure of $S$ in the field of rational functions on $\cZ_d(\cX_0/S_0)$.
The same argument works for $\cZ_d(\cX'/S)$, hence $\varphi_0$ extends to an isomorphism $\varphi$.

The second claim is evident because the restriction morphisms $\Br(\cZ_d(\cX/S)) \to  \Br(\cZ_d(\cX_0/S_0))$ and~$\Br(\cZ_d(\cX'/S)) \to  \Br(\cZ_d(\cX'_0/S_0))$ 
of the Brauer groups are injective.
\end{proof}

We also have a Hilbert scheme interpretation for the semiorthogonal decomposition.
Let $F_d(\cX/S)$ be the relative Hilbert scheme of subschemes in the fibers of $\cX$ over $S$ with Hilbert polynomial $h'_d(t)$ defined by~\eqref{eq:hdprime}.
Thus, $F_1(\cX/S)$ is the relative Hilbert scheme of lines, $F_2(\cX/S)$ is the relative Hilbert scheme of conics, and $F_3(\cX/S)$ is the relative Hilbert scheme of twisted cubic curves.

\begin{proposition}
\label{proposition:fd-md}
For each $1 \le d \le 3$ the scheme $F_d(\cX/S)$ is flat over $S$, and
\begin{itemize}
\item 
$F_1(\cX/S) \cong \cZ_2 \times_S \cZ_3$,
\item 
$F_2(\cX/S)$ is an \'etale locally trivial $\P^1$-bundle over $\cZ_2$, and 
\item 
$F_3(\cX/S)$ is an \'etale locally trivial $\P^2$-bundle over $\cZ_3$.
\end{itemize}
Moreover, $F_2(\cX/S)$ and $F_3(\cX/S)$ are Severi--Brauer varieties over~$\cZ_2$ and~$\cZ_3$
associated with the Brauer classes $\beta_{\cZ_2}$ and $\beta_{\cZ_3}$, respectively.
\end{proposition}

\begin{proof}
Assume $d \in \{2,3\}$.
The construction of the morphism $F_d(X) \to Z_d$ from Proposition~\ref{proposition:fd-zd} works well 
in an arbitrary du Val family of sextic del Pezzo surfaces and provides a morphism~\mbox{$F_d(\cX/S) \to \cZ_d(\cX/S)$}.
Moreover, it identifies $F_d(\cX/S)$ with the projectivization of the (twisted) vector bundle $p_{d*}\cE_{\cZ_d}$, 
where~$\cE_{\cZ_d}$ is the (twisted) universal sheaf on the product $\cX \times_S \cZ_d$ 
and~$p_{d} \colon \cX \times_S \cZ_d \to \cZ_d$ is the projection.
This is equivalent to the statement of the proposition.

For $d = 1$ also a relative version of the argument of Proposition~\ref{proposition:fd-zd} works.
\end{proof}

\begin{remark}
\label{remark:f4-family}
Taking into account Remark~\ref{remark:f4} and Proposition~\ref{proposition:md-zd} one can prove with the same argument
that the relative Hilbert scheme $F_4(\cX/S)$ is an \'etale locally trivial $\P^3$-bundle over $\cZ_2$ 
with the same Brauer class~$\beta_{\cZ_2}$ as~$F_2(\cX/S)$.
\end{remark}

\begin{corollary}
The total space of $\cX$ is regular if and only if $S$, $F_2(\cX/S)$ and $F_3(\cX/S)$ are regular.
\end{corollary}
\begin{proof}
By Proposition~\ref{proposition:fd-md} if $d \in \{2,3\}$, the morphism $F_d(\cX/S) \to \cZ_d$ is smooth, 
hence the Hilbert scheme~$F_d(\cX/S)$ is regular if and only if $\cZ_d$ is regular.
So, Proposition~\ref{proposition:smoothness} applies.
\end{proof}

\begin{corollary}
\label{corollary:birational-isomorphism-fd}
Let $\cX \to S$ and $\cX' \to S$ be two du Val families of sextic del Pezzo surfaces with regular total spaces $\cX$ and $\cX'$.
Assume for $d \in \{2,3\}$ there is a birational $S$-isomorphism $\psi \colon \xymatrix@1{F_d(\cX/S) \ar@{-->}[r]^-{\sim} & F_d(\cX'/S)}$.
Then it induces a biregular isomorphism $\varphi \colon \cZ_d(\cX/S) \xrightarrow{\ \sim\ } \cZ_d(\cX'/S)$ over~$S$ 
and we have $\beta_{\cZ_d(\cX/S)} = \varphi^*\beta_{\cZ_d(\cX'/S)}$.
\end{corollary}
\begin{proof}
Recall that $\cZ_d(\cX/S)$ is the base of the maximal rationally connected fibration for the morphism~$F_d(\cX/S) \to S$.
Therefore, the birational isomorphism $\psi$ of Hilbert schemes induces 
a birational isomorphism $\varphi_0$ over~$S$ of $\cZ_d(\cX/S)$ and $\cZ_d(\cX'/S)$.
Using Corollary~\ref{corollary:birational-isomorphism-zd} we deduce that it extends 
to an isomorphism $\varphi \colon \cZ_d(\cX/S) \xrightarrow{\ \sim\ } \cZ_d(\cX'/S)$.
Furthermore, since $F_d(\cX/S)$ is a Severi--Brauer variety associated with the Brauer class $\beta_{\cZ_d(\cX/S)}$ of order~$2$ or~$3$, 
the birational isomorphism of Hilbert schemes implies that either $\beta_{\cZ_d(\cX/S)} = \varphi^*\beta_{\cZ_d(\cX'/S)}$
or~$d = 3$ and~$\beta_{\cZ_d(\cX/S)} = \varphi^*\beta_{\cZ_d(\cX'/S)}^{-1}$.
In the latter case we replace~$\varphi$ by its composition with the involution $\sigma_{3,3}$
defined in Proposition~\ref{proposition:md-zd}.
\end{proof}

We would like to thank the referee for pointing out the following interesting consequence of Proposition~\ref{proposition:fd-md} 
which might be useful for applications in birational geometry.

\begin{proposition}
Let $\cX \to S$ be a generically smooth du Val family of sextic del Pezzo surfaces with regular total space~$\cX$.

\noindent$(1)$ The Brauer class $\beta_{\cZ_2}$ vanishes if and only if the family $\cX \to S$ has a rational $3$-multisection.

\noindent$(2)$ The Brauer class $\beta_{\cZ_3}$ vanishes if and only if the family $\cX \to S$ has a rational $2$-multisection.
\end{proposition}
\begin{proof}
Note that the assumption implies that $S$, $\cZ_2$, and $\cZ_3$ are all regular (Proposition~\ref{proposition:smoothness}).
Furthermore, by base change to an open subscheme in $S$ we may assume that $\cX \to S$ is smooth
(hence $\cZ_d \to S$ are smooth as well by Remark~\ref{remark:smoothness}).

If $\beta_{\cZ_2}$ vanishes then the map $F_2(\cX) \to \cZ_2$ has a rational section.
Moreover, one can assume that such a rational section is not contained in the rational $2$-multisection of $F_2(\cX) \to \cZ_2$
that parameterizes singular conics on fibers of $\cX$.
This means that there is a birational morphism $\cZ'_2 \to \cZ_2$ and a subscheme~$\cC \subset \cX \times_S \cZ'_2$ 
which is a family of conics on fibers of $\cX$ with the general fiber over~$\cZ'_2$ a union of three smooth conics
from distinct linear systems on a smooth sextic del Pezzo surface.
Note that smooth conics from different linear systems intersect at a point, 
hence the projection $\cC \to \cX$ fails to be an isomorphism onto its image over a subscheme of $\cX$ 
which is a rational 3-multisection of $\cX \to S$.

Similarly, if $\beta_{\cZ_3}$ vanishes then the map $F_3(\cX) \to \cZ_3$ has a nice rational section,
hence there is a birational morphism $\cZ'_3 \to \cZ_3$ and a subscheme~$\cC \subset \cX \times_S \cZ'_3$ 
which is a family of rational cubics on fibers of $\cX$ with the general fiber over~$\cZ'_3$ a union of two smooth cubics
from distinct linear systems on a smooth sextic del Pezzo surface.
Smooth cubics from different linear systems intersect at two points, 
hence the projection $\cC \to \cX$ fails to be an isomorphism onto its image over a subscheme of $\cX$ 
which is a rational 2-multisection of $\cX \to S$.

Conversely, assume a 2-multisection of $\cX \to S$ exists, i.e., there is a proper subscheme $T \subset \cX$ 
such that the map $T \to S$ is generically finite of degree~2. In the Hilbert scheme $F_3(\cX/S)$ consider 
the subscheme $F_3^T(\cX/S)$ that parameterizes cubic curves containing the fibers of~$T$ over~$S$.
Since on a smooth del Pezzo surface in each linear system of cubic curves there is either a unique curve 
containing a given pair of points (if the points do not lie on a line), or a pencil of such cubic curves (otherwise),
we conclude that $F_3^T(\cX/S) \subset F_3(\cX/S)$ is 
either a rational section of $F_3(\cX/S) \to \cZ_3$ or is generically a sub-$\P^1$-bundle. 
In both cases it follows that the Brauer class $\beta_{\cZ_3}$ vanishes generically over~$\cZ_3$, 
and since~$\cZ_3$ is regular, $\beta_{\cZ_3} = 0$.

Similarly, if a 3-multisection $T$ of $\cX \to S$ exists, we consider the subscheme $F_4^T(\cX/S) \subset F_4(\cX/S)$
that parameterizes quartic curves containing the fibers of~$T$ over~$S$.
Since on a smooth del Pezzo surface in each linear system of quartic curves there is either a unique curve 
containing a given triple of points (if the points do not lie on a conic), 
or a pencil of such curves (if the points lie on a conic, but not on a line),
or a 2-dimensional linear system (otherwise),
we conclude that $F_4^T(\cX/S) \subset F_4(\cX/S)$ is 
either a rational section of $F_4(\cX/S) \to \cZ_4 = \cZ_2$ (in the first case),
or is generically a sub-$\P^2$-bundle (in the third case), 
while in the second case the subscheme~$F_2^T(\cX/S) \subset F_2(\cX/S)$ provides a rational section of $F_2(\cX/S) \to \cZ_2$.
In all cases it follows that the Brauer class $\beta_{\cZ_2}$ vanishes generically over~$\cZ_2$, 
and since~$\cZ_2$ is regular, $\beta_{\cZ_2} = 0$.
\end{proof}

\section{Standard families}\label{section:special-families}

In this section we discuss some standard families of sextic del Pezzo surfaces and their special features.
Throughout this section the base field $\kk$ is an arbitrary field of characteristic coprime to 2 and 3.

\subsection{Linear sections of $\P^2 \times \P^2$}\label{subsection:p2p2}

The simplest way to construct a sextic del Pezzo surface is by considering an intersection of $\P^2 \times \P^2$ 
with a linear subspace of codimension 2 in the Segre embedding.

We denote by $W_1$ and $W_2$ a pair of vector spaces of dimension 3 and let 
\begin{equation*}
\P(W_1) \times \P(W_2) \hookrightarrow \P(W_1 \otimes W_2)
\end{equation*}
be the Segre embedding. 
To give its linear section of codimension 2, we need a two-dimensional subspace~$K \subset W_1^\vee \otimes W_2^\vee$.
We denote by $K^\perp \subset W_1 \otimes W_2$ its codimension 2 annihilator, and set
\begin{equation}\label{eq:xk}
X_K := (\P(W_1) \times \P(W_2)) \cap \P(K^\perp) \subset \P(W_1 \otimes W_2), 
\end{equation}
to be the corresponding linear section.

\begin{lemma}\label{lemma:xk-duval}
Assume the base field $\kk$ is algebraically closed.
The intersection $X_K$ defined by~\textup{\eqref{eq:xk}} is a sextic du Val del Pezzo surface if and only if 
\begin{itemize}
\item the line $\P(K)$ does not intersect the dual Segre variety $\P(W_1^\vee) \times \P(W_2^\vee) \subset \P(W_1^\vee \otimes W_2^\vee)$, and
\item the line $\P(K)$ is not contained in the discriminant cubic hypersurface $\cD_{W_1,W_2} \subset \P(W_1^\vee \otimes W_2^\vee)$.
\end{itemize}
Furthermore, $X_K$ is smooth if and only if the line $\P(K)$ is transversal to $\cD_{W_1,W_2}$.
\end{lemma}
\begin{proof}
First, let us show that the condition is necessary.
If $\P(K)$ intersects $\P(W_1^\vee) \times \P(W_2^\vee)$ then $K$ contains a bilinear form $b \in W_1^\vee \otimes W_2^\vee$ of rank 1, i.e., $b = \varphi_1 \otimes \varphi_2$, where $\varphi_i \in W_i^\vee$.
The zero locus of~$b$ on $\P(W_1) \times \P(W_2)$ is equal to $(\P(\varphi_1^\perp) \times \P(W_2)) \cup (\P(W_1) \times \P(\varphi_2^\perp))$. 
Consequently, $X_K$ is its hyperplane section, hence is a union of two cubic scrolls, hence is not an integral surface.

Next, assume that $\P(K)$ is contained in $\cD_{W_1,W_2}$. 
Then there are three possibilities: either
\begin{itemize}
\item 
$K \subset w_1^\perp \otimes W_2^\vee$ for some $w_1 \in W_1$, or
\item 
$K \subset W_1^\vee \otimes w_2^\perp$ for some $w_2 \in W_2$, or
\item 
$K \subset \Ker(W_1^\vee \otimes W_2^\vee \to U_1^\vee \otimes U_2^\vee)$ for some 2-dimensional subspaces $U_1 \subset W_1$ and $U_2  \subset W_2$.
\end{itemize}
In the first case we have $\{w_1\} \times \P(W_2) \subset X_K$, in the second case $\P(W_1) \times \{w_2\} \subset X_K$, and in the third case $\P(U_1) \times \P(U_2) \subset X_K$,
so in all these cases the surface $X_K$ is not integral.

Now, let us show that the conditions are sufficient.
So, assume $K \subset W_1^\vee \otimes W_2^\vee$ is such that the line~$\P(K)$ is not contained in $\cD_{W_1,W_2}$ and does not intersect $\P(W_1^\vee) \times \P(W_2^\vee)$.
Let $b_0 \in K$ be a bilinear form of rank 3 (it exists since $\P(K)$ is not contained in $\cD_{W_1,W_2}$).
Then $b_0$ identifies $W_2$ with $W_1^\vee$ and under this identification $b_0$ corresponds to the identity in $\P(W_1^\vee \otimes W_1) \cong \P(\End(W_1))$, 
so its zero locus is isomorphic to the flag variety $\Fl(1,2;W_1) \subset \P(W_1) \times \P(W_1^\vee)$. 
The subspace $K$ is then determined by an operator $b \in \End(W_1)$ defined up to a scalar multiple of the identity, and the condition that $\P(K)$ does not intersect $\P(W_1^\vee) \otimes \P(W_2^\vee)$
can be rephrased by saying that the pencil $\{b + t\id\}$ does not contain operators of rank~1.
From the Jordan Theorem it is clear that there are only three types of such $b$:
\begin{itemize}
\item[$(0)$]
$b$ is diagonal with three distinct eigenvalues;
\item[$(1)$]
$b$ has two Jordan blocks of sizes 2 and 1 with distinct eigenvalues; 
\item[$(2)$]
$b$ has one Jordan block of size 3.
\end{itemize}
It is easy to see that in case (0) the surface $X_b$ is smooth (hence of type 0), 
in case (1) it has one $A_1$ singularity (and is of type 2), and in case (3) it has one $A_2$ singularity (and is of type 4).

It remains to note that case (0) happens precisely when the line $\P(K)$ is transversal to $\cD_{W_1,W_2}$.
\end{proof}

Consider the universal family of codimension 2 linear sections of $\P(W_1) \times \P(W_2)$ that are sextic du Val del Pezzo surfaces.
According to Lemma~\ref{lemma:xk-duval} it can be described as follows.
Consider the Grassmannian~$\Gr(2,W_1^\vee \otimes W_2^\vee)$ parameterizing all two-dimensional subspaces $K \subset W_1^\vee \otimes W_2^\vee$, 
and its open subset parameterizing subspaces satisfying the conditions of Lemma~\ref{lemma:xk-duval}:
\begin{equation*}
S := \{ K \in \Gr(2,W_1^\vee \otimes W_2^\vee) \quad \mid \quad 
\P(K) \not\subset \cD_{W_1,W_2} \quand \P(K) \cap (\P(W_1^\vee) \times \P(W_2^\vee)) = \varnothing \}.
\end{equation*}
Let $\cK \subset W_1^\vee \otimes W_2^\vee \otimes \cO_S$ be the tautological rank 2 bundle, 
let $\cK^\perp \subset W_1 \otimes W_2 \otimes \cO_S$ be its rank~7 annihilator, and let 
\begin{equation}
\label{eq:cx-p2p2}
\cX := (\P(W_1) \times \P(W_2)) \times_{\P(W_1 \otimes W_2)} \P_S(\cK^\perp)
\end{equation}
be the corresponding du Val family of sextic del Pezzo surfaces.

\begin{proposition}
\label{proposition:dp-p2p2}
Let $\cX \to S$ be the du Val family of sextic del Pezzo surfaces defined by~\textup{\eqref{eq:cx-p2p2}}. 
Then 
\begin{equation*}
\cZ_2 = \P_S(\cK) \times_{\P(W_1^\vee \otimes W_2^\vee)} \cD_{W_1,W_2},
\qquad 
\cZ_3 = S \sqcup S,
\end{equation*}
and the Brauer classes $\beta_{\cZ_2}$ and $\beta_{\cZ_3}$ are both trivial.
\end{proposition}
\begin{proof}
Using homological projective duality for $\P(W_1) \times \P(W_2)$, see Theorem~\ref{theorem:hpd-p2p2}, 
we obtain from~\cite[Theorem~6.27]{kuznetsov2007hpd} a semiorthogonal decomposition
\begin{equation*}
\bD(\cX) = \langle \Phi_{\cE_2}(\bD(\P_S(\cK) \times_{\P(W_1^\vee \otimes W_2^\vee)} \YY_2)), \bD(S) \otimes \cO_\cX(1,0), \bD(S) \otimes \cO_\cX(0,1), \bD(S) \otimes \cO_\cX(1,1) \rangle,
\end{equation*}
where $\YY_2 \to \cD_{W_1,W_2}$ is the resolution of singularities defined in~\eqref{eq:def-yy2}, and
$\cE_2$ is the derived pullback of the sheaf $\EE_2$ defined by~\eqref{eq:hpd-kernel-xx3} with respect to the natural map 
\begin{equation}\label{eq:map-cx2}
\cX \times_S (\P_S(\cK) \times_{\P(W_1^\vee \otimes W_2^\vee)} \YY_2) \to \cX_2 \times_{\P(W_1^\vee \otimes W_2^\vee)} \YY_2
\end{equation} 
(we will discuss this map below), where $\cX_2$ is the universal hyperplane section of $\P(W_1) \times \P(W_2)$.
Mutating the last component to the far left (note that $K_\cX = \cO_\cX(-1,-1)$ in the relative Picard group~$\Pic(\cX/S)$), we get
\begin{equation}
\label{eq:sod-cx-p2p2}
\bD(\cX) = \langle \bD(S) \otimes \cO_\cX, \Phi_{\cE_2}(\bD(\P_S(\cK) \times_{\P(W_1^\vee \otimes W_2^\vee)} \YY_2)), \bD(S) \otimes \cO_\cX(1,0), \bD(S) \otimes \cO_\cX(0,1) \rangle.
\end{equation}
We claim that this decomposition agrees with the general decomposition~\eqref{eq:dbcx-again} of a du Val family of sextic del Pezzo surfaces.

Indeed, the last two components of~\eqref{eq:sod-cx-p2p2} can be considered as the derived category of~$S \sqcup S$ (the trivial double covering of $S$) embedded via the Fourier--Mukai functor with kernel
\begin{equation*}
\cE_3 := \cO_\cX(1,0) \sqcup \cO_\cX(0,1) \in \bD(\cX \sqcup \cX) = \bD(\cX \times_S (S \sqcup S)),
\end{equation*}
and the second component is the derived category of 
\begin{equation*}
\P_S(\cK) \times_{\P(W_1^\vee \otimes W_2^\vee)} \YY_2 = \P_S(\cK) \times_{\P(W_1^\vee \otimes W_2^\vee)} \cD_{W_1,W_2}
\end{equation*}
(recall that $\YY_2$ maps birationally onto $\cD_{W_1,W_2}$ and $\P_S(\cK)$ by definition of $S$ 
does not touch the indeterminacy locus $\P(W_1^\vee) \times \P(W_2^\vee)$ of that birational isomorphism),
which is a flat degree~3 covering of~$S$ (since $\cD_{W_1,W_2}$ is a hypersurface of degree~3), 
and it is embedded into~$\bD(\cX)$ via the Fourier--Mukai functor with kernel~$\cE_2$.

Let us check that both $\cE_2$ and $\cE_3$ are flat families of torsion-free rank 1 sheaves 
on fibers of~$\cX$ over~$S$ with Hilbert polynomials $h_2(t)$ and $h_3(t)$ respectively, parameterized by the schemes
\begin{equation*}
\cZ'_2 := \P_S(\cK) \times_{\P(W_1^\vee \otimes W_2^\vee)} \YY_2
\qquand
\cZ'_3 := S \sqcup S.
\end{equation*}
For the second family flatness is clear and Hilbert polynomial computation is straightforward, so we skip it.
For the first family we note that the map~\eqref{eq:map-cx2} is flat.
Indeed, by~\eqref{eq:cx-p2p2} we have
\begin{equation*}
\cX \times_S (\P_S(\cK) \times_{\P(W_1^\vee \otimes W_2^\vee)} \YY_2) =
\Big((\P(W_1) \times \P(W_2)) \times_{\P(W_1 \otimes W_2)} \P_S(\cK^\perp)\Big) \times_S \Big(\P_S(\cK) \times_{\P(W_1^\vee \otimes W_2^\vee)} \YY_2\Big),
\end{equation*}
while 
\begin{equation*}
\cX_2 \times_{\P(W_1^\vee \otimes W_2^\vee)} \YY_2 = 
(\P(W_1) \times \P(W_2)) \times_{\P(W_1 \otimes W_2)} Q \times_{\P(W_1^\vee \otimes W_2^\vee)} \YY_2,
\end{equation*}
where $Q \subset \P(W_1 \otimes W_2) \times \P(W_1^\vee \otimes W_2^\vee)$ is the incidence quadric, so the map~\eqref{eq:map-cx2} is induced 
by the natural map $\P_S(\cK^\perp) \times_S \P_S(\cK) \to Q = \Fl(1,8;W_1^\vee \otimes W_2^\vee)$.
This map factors as 
\begin{equation*}
\P_S(\cK^\perp) \times_S \P_S(\cK) \hookrightarrow 
\P_{\Gr(2,W_1^\vee \otimes W_2^\vee)}(\cK^\perp) \times_{\Gr(2,W_1^\vee \otimes W_2^\vee)} \P_{\Gr(2,W_1^\vee \otimes W_2^\vee)}(\cK) = 
\Fl(1,2,8;W_1^\vee \otimes W_2^\vee) \to Q.
\end{equation*}
The first map is an open embedding, while the last map is a $\P^6$-bundle, so the composition is flat (and even smooth).

Recall the embedding of the Weil divisor $\YY_2 \times \P(W_2) \hookrightarrow \cX_2 \times_{\P(W_1^\vee \otimes W_2^\vee)} \YY_2$ defined by~\eqref{diagram:yy-cx3}.
Its preimage under the map~\eqref{eq:map-cx2} is a Weil divisor in $\cX \times_S (\P_S(\cK) \times_{\P(W_1^\vee \otimes W_2^\vee)} \YY_2)$ 
such that for any geometric point~$(K,b,w_1)$ of $\cZ'_2 := \P_S(\cK) \times_{\P(W_1^\vee \otimes W_2^\vee)} \YY_2$ its fiber is the Weil divisor 
\begin{equation*}
L := X_K \times_{\P(W_1)} \{w_1\} \subset X_K,
\end{equation*}
i.e., a line on the sextic del Pezzo surface $X_K$ contracted by the projection $X_K \to \P(W_1)$ to the point~$\{w_1\} \in \P(W_1)$.
Thus, by~\eqref{eq:hpd-kernel-xx3}, the sheaf we are interested in is the twisted ideal $\cI_L \otimes \cO_{\P(W_1)}(1)$.
In particular, it is torsion-free, and its Hilbert polynomial equals~$h_2(t)$.

Therefore, the families of sheaves $\cE_2$ and $\cE_3$ induce maps 
\begin{equation*}
\mu_2 \colon \cZ'_2 \to \cZ_2(\cX/S)
\qquand 
\mu_3 \colon \cZ'_3 \to \cZ_3(\cX/S)
\end{equation*}
such that the Brauer classes $\mu_2^*(\beta_{\cZ_2})$ and $\mu_3^*(\beta_{\cZ_3})$ are trivial 
and $\cE_2$ and $\cE_3$ are isomorphic (up to twists by line bundles on~$\cZ'_2$ and $\cZ'_3$) 
to the pullbacks of the universal bundles $\cE_{\cZ_2}$ and $\cE_{\cZ_3}$, respectively.
This means that for $d \in \{2,3\}$ we have isomorphisms of Fourier--Mukai functors
\begin{equation*}
\Phi_{\cE_{d}} \cong \Phi_{\cE_{\cZ_d}} \circ \mu_{d*} \circ T_d \colon \bD(\cZ'_d) \to \bD(\cX),
\end{equation*}
where $T_d$ is a line bundle twist in $\bD(\cZ'_d)$.

Since both $\Phi_{\cE_{\cZ_d}}$ and $\Phi_{\cE_d}$ are fully faithful, 
so is the composition $\mu_{d*} \circ T_d \colon \bD(\cZ'_d) \to \bD(\cZ_d)$.
Moreover, comparing semiorthogonal decompositions given by~\eqref{eq:dbcx-again} and~\eqref{eq:sod-cx-p2p2}, we conclude that $\mu_{d*} \circ T_d$ is essentially surjective, i.e., an equivalence of categories.
Since $T_d$ is also an equivalence, we conclude that $\mu_{d*}$ is an equivalence, hence $\mu_d$ is an isomorphism.
Thus $\cZ'_d \cong \cZ_d(\cX/S)$ and the Brauer classes $\beta_{\cZ_2}$ and $\beta_{\cZ_3}$  on~$\cZ_2(\cX/S)$ and~$\cZ_3(\cX/S)$ vanish.
\end{proof}

\begin{remark}
\label{remark:extended-family}
Considering the family $\bar\cX \subset \P(W_1) \times \P(W_2) \times \bar{S}$ of \emph{all} codimension 2 linear sections of~$\P(W_1) \times \P(W_2)$ over $\bar{S} := \Gr(2,W_1^\vee \otimes W_2^\vee)$ 
and applying the semiorthogonal decomposition of Theorem~\ref{theorem:hpd-p2p2}, we obtain
\begin{equation*}
\bD(\bar\cX) = \langle \bD(\bar{S}), \bD(\bar{\cZ}_2), \bD(\bar{\cZ}_3) \rangle,
\end{equation*}
where $\bar{\cZ_2} = \P_S(\cK) \times_{\P(W_1^\vee \otimes W_2^\vee)} \YY_2$ and $\bar{\cZ_3} = \bar{S} \sqcup \bar{S}$.
Note that the map $\bar{\cZ_2} \to \bar{S}$ is not flat --- its non-flat locus $\bar{S} \setminus S$ is equal 
to the locus of non-integral fibers of the family $\bar{\cX} \to \bar{S}$.
\end{remark}

The statement of Proposition~\ref{proposition:dp-p2p2} can be inverted as follows.

\begin{lemma}
\label{lemma:reconstructing-family}
Let $\cX \to S$ be a du Val family of sextic del Pezzo surfaces.
If $\cZ_3(\cX/S) = S \sqcup S$ and the Brauer class $\beta_{\cZ_3}$ is trivial, 
then Zariski locally over $S$ the family $\cX \to S$ can be represented as a family of codimension $2$ linear sections of\/ $\P^2 \times \P^2$.
In particular, if $\cX$ is regular then $\beta_{\cZ_2}$ is trivial.
\end{lemma}
\begin{proof}
By Proposition~\ref{proposition:fd-md} there is a pair of rank 3 vector bundles $\cW_1$ and $\cW_2$ on $S$ such that
\begin{equation*}
F_3(\cX/S) = \P_S(\cW_1) \sqcup \P_S(\cW_2).
\end{equation*}
Moreover, the restrictions of the universal sheaf $\cE_{\cZ_3}$ on $\cX \times_S \cZ_3 = \cX \sqcup \cX$ to the two components are line bundles defining regular maps $\cX \to \P_S(\cW_1)$ and $\cX \to \P_S(\cW_2)$ respectively.
Thus, the induced map 
\begin{equation*}
\cX \to \P_S(\cW_1) \times_S \P_S(\cW_2)
\end{equation*}
is a closed embedding.
Zariski locally the bundles $\cW_1$ and $\cW_2$ are trivial, hence we obtain the required local presentation of $\cX$.
The Brauer class $\beta_{\cZ_2}$ vanishes Zariski locally by Proposition~\ref{proposition:dp-p2p2}
and~$\cZ_2$ is regular by Proposition~\ref{proposition:smoothness}, hence $\beta_{\cZ_2} = 0$.
\end{proof}

A slightly more general family of sextic del Pezzo surfaces can be obtained by replacing the projective spaces $\P(W_1)$ and $\P(W_2)$ 
by a pair of projectively dual (i.e., corresponding to mutually inverse Brauer classes) Severi--Brauer surfaces.
In this way one can obtain a family of sextic del Pezzo surfaces with a nontrivial Brauer class $\beta_{\cZ_3}$ 
(however, this class will be ``constant in a family'').
Another possible generalization, is to consider a double covering $\tS \to S$ and a Brauer class $\beta$ on $\tS$ of order~3 
such that~$\beta = \sigma^*(\beta^{-1})$, 
where $\sigma \colon \tS \to \tS$ is the involution of the double covering 
(cf.~the isomorphism of Brauer classes in Proposition~\ref{proposition:md-zd}).
Then one can apply the Weil restriction of scalars to obtain an \'etale locally trivial fibration over $S$ 
with fibers $\P^2 \times \P^2$, and then consider its linear section of codimension~2.

\subsection{Hyperplane sections of $\P^1 \times \P^1 \times \P^1$}
\label{subsection:p1p1p1}

Another simple way to construct a sextic del Pezzo surface is by considering a hyperplane section of $\P^1 \times \P^1 \times \P^1$ in the Segre embedding.

We denote by $V_1$, $V_2$, and $V_3$ three vector spaces of dimension 2 and let 
\begin{equation*}
\P(V_1) \times \P(V_2) \times \P(V_3) \hookrightarrow \P(V_1 \otimes V_2 \otimes V_3)
\end{equation*}
be the Segre embedding. 
For a trilinear form $b \in V_1^\vee \otimes V_2^\vee \otimes V_3^\vee$ we denote by
\begin{equation}\label{eq:xb}
X_b := (\P(V_1) \times \P(V_2) \times \P(V_3)) \cap \P(b^\perp) \subset \P(V_1 \otimes V_2 \otimes V_3) 
\end{equation}
the corresponding hyperplane section, where $b^\perp \subset V_1 \otimes V_2 \otimes V_3$ is the annihilator hyperplane of $b$.

Recall that the group $\rG := (\PGL(V_1) \times \PGL(V_2) \times \PGL(V_3)) \rtimes \fS_3$ acts on $\P(V_1^\vee \otimes V_2^\vee \otimes V_3^\vee)$ with four orbits.
The orbits closures are:
\begin{itemize}
\item $\rO_3 = \P(V_1^\vee) \times \P(V_2^\vee) \times \P(V_3^\vee)$;
\item $\overline\rO_4 = (\P(V_1^\vee) \times \P(V_2^\vee \otimes V_3^\vee)) \cup (\P(V_2^\vee) \times \P(V_1^\vee \otimes V_3^\vee)) \cup (\P(V_3^\vee) \times \P(V_1^\vee \otimes V_2^\vee))$;
\item $\overline\rO_6 = (\P(V_1) \times \P(V_2) \times \P(V_3))^\vee$ is the projectively dual quartic hypersurface; and
\item $\overline\rO_7 = \P(V_1^\vee \otimes V_2^\vee \otimes V_3^\vee)$
\end{itemize}
(we use the dimensions of the orbits as indices).
The equation of the quartic hypersurface $\overline\rO_6$ is given by the Cayley's hyperdeterminant, see~\cite[14.1.7]{gkz1994}.

\begin{lemma}\label{lemma:xb-duval}
Assume the base field $\kk$ is algebraically closed.
The hyperplane section $X_b$ defined by~\textup{\eqref{eq:xb}} is a sextic du Val del Pezzo surface if and only if $b \in \P(V_1^\vee \otimes V_2^\vee \otimes V_3^\vee) \setminus \overline\rO_4$.
Furthermore, $X_b$ is smooth if and only if $b \in \P(V_1^\vee \otimes V_2^\vee \otimes V_3^\vee) \setminus \overline\rO_6$.
\end{lemma}
\begin{proof}
Indeed, choosing a representative of each orbit, it is easy to see that 
$X_b$ is smooth, if $b \in \rO_7$; 
has one $A_1$ singularity (and is of type 1), if $b \in \rO_6$;
is a union of a smooth quadric and a quartic scroll, if~$b \in \rO_4$; 
and is a union of three quadrics, if $b \in \rO_3$.
\end{proof}

Consider the universal family of hyperplane sections of $\P(V_1) \times \P(V_2) \times \P(V_3)$ that are sextic du Val del Pezzo surfaces.
By Lemma~\ref{lemma:xb-duval} it can be described as follows.
Consider the open subset 
\begin{equation*}
S := \rO_6 \cup \rO_7 = \P(V_1^\vee \otimes V_2^\vee \otimes V_3^\vee) \setminus \overline\rO_4 \subset \P(V_1^\vee \otimes V_2^\vee \otimes V_3^\vee).
\end{equation*}
Let $\cL \subset V_1^\vee \otimes V_2^\vee \otimes V_3^\vee \otimes \cO_S$ be the tautological line bundle, let $\cL^\perp \subset V_1 \otimes V_2 \otimes V_3 \otimes \cO_S$ be its rank 7 annihilator, and let 
\begin{equation}\label{eq:cx-p1p1p1}
\cX := (\P(V_1) \times \P(V_2) \times \P(V_3)) \times_{\P(V_1 \otimes V_2 \otimes V_3)} \P_S(\cL^\perp)
\end{equation}
be the corresponding du Val family of sextic del Pezzo surfaces.

Denote by ${\cD_{V_1,V_2,V_3}} \to \P(V_1^\vee \otimes V_2^\vee \otimes V_3^\vee)$ the double covering branched along the Cayley quartic hypersurface~$\overline\rO_6$.

\begin{proposition}\label{proposition:dp-p1p1p1}
Let $\cX \to S$ be the du Val family of sextic del Pezzo surfaces defined by~\textup{\eqref{eq:cx-p1p1p1}}. 
Then 
\begin{equation*}
\cZ_2 = S \sqcup S \sqcup S,
\qquad 
\cZ_3 = S \times_{\P(V_1^\vee \otimes V_2^\vee \otimes V_3^\vee)} {\cD_{V_1,V_2,V_3}},
\end{equation*}
and the Brauer classes $\beta_{\cZ_2}$ and $\beta_{\cZ_3}$ are both trivial.
\end{proposition}
\begin{proof}
The proof is parallel to that of Proposition~\ref{proposition:dp-p2p2}.

Using homological projective duality for $\P(V_1) \times \P(V_2) \times \P(V_3)$, see Theorem~\ref{theorem:hpd-p1p1p1}, 
we obtain from~\cite[Theorem~6.27]{kuznetsov2007hpd} a semiorthogonal decomposition
\begin{multline*}
\bD(\cX) = \langle \Phi_{\cE_3}(\bD(S \times_{\P(V_1^\vee \otimes V_2^\vee \otimes V_3^\vee)} \YY_3)), \\
\bD(S) \otimes \cO_\cX(1,1,1), \bD(S) \otimes \cO_\cX(2,1,1), \bD(S) \otimes \cO_\cX(1,2,1), \bD(S) \otimes \cO_\cX(1,1,2) \rangle,
\end{multline*}
where $\YY_3 \to \cD_{V_1,V_2,V_3}$ is the resolution of singularities defined in~\eqref{eq:def-yy3} and
where $\cE_3$ is the derived pullback of the sheaf $\EE_3$ defined by~\eqref{eq:hpd-kernel-xxx3} with respect to the natural map 
\begin{equation}\label{eq:map-cx3}
\cX \times_S (S \times_{\P(V_1^\vee \otimes V_2^\vee \otimes V_3^\vee)} \YY_3) \to \cX_3 \times_{\P(V_1^\vee \otimes V_2^\vee \otimes V_3^\vee)} \YY_3
\end{equation} 
(we will discuss this map below).
Mutating the last four components to the far left 
(again, note that~\mbox{$K_\cX = \cO_\cX(-1,-1,-1)$} in the relative Picard group~$\Pic(\cX/S)$), we get
\begin{multline}
\bD(\cX) = \langle \bD(S) \otimes \cO_\cX, \bD(S) \otimes \cO_\cX(1,0,0), \bD(S) \otimes \cO_\cX(0,1,0), \bD(S) \otimes \cO_\cX(0,0,1), \\
\Phi_{\cE_3}(\bD(S \times_{\P(V_1^\vee \otimes V_2^\vee \otimes V_3^\vee)} \YY_3)) \rangle.
\end{multline}
We claim that this decomposition agrees with the general decomposition~\eqref{eq:dbcx-again} for a du Val family of sextic del Pezzo surfaces.
Indeed, its second, third and fourth components can be considered as the derived category of $S \sqcup S \sqcup S$ (the trivial triple covering of $S$) embedded via the Fourier--Mukai functor with kernel
\begin{equation*}
\cE_2 := \cO_\cX(1,0,0) \sqcup \cO_\cX(0,1,0) \sqcup \cO_\cX(0,0,1) \in \bD(\cX \sqcup \cX \sqcup \cX) = \bD(\cX \times_S (S \sqcup S \sqcup S)),
\end{equation*}
while the last component is the derived category 
of~$S \times_{\P(V_1^\vee \otimes V_2^\vee \otimes V_3^\vee)} \YY_3 = 
S \times_{\P(V_1^\vee \otimes V_2^\vee \otimes V_3^\vee)} {\cD_{V_1,V_2,V_3}}$
(recall that $\YY_3$ maps birationally onto ${\cD_{V_1,V_2,V_3}}$ and $S$ by definition 
does not touch the image $\overline\rO_4$ (see Remark~\ref{remark:fibers-resolution-dvvv}) 
of the indeterminacy locus of that birational isomorphism),
which is a flat degree~2 covering of~$S$ (since ${\cD_{V_1,V_2,V_3}}$ is flat over the complement of~$\overline\rO_4$), 
and it is embedded into~$\bD(\cX)$ via the Fourier--Mukai functor with kernel~$\cE_3$.

Let us check that both $\cE_2$ and $\cE_3$ are flat families of torsion-free rank 1 sheaves on the fibers of $\cX$ over~$S$ with Hilbert polynomials $h_2(t)$ and $h_3(t)$ respectively, parameterized by 
\begin{equation*}
\cZ'_2 := S \sqcup S \sqcup S
\qquand
\cZ'_3 := S \times_{\P(V_1^\vee \otimes V_2^\vee \otimes V_3^\vee)} \YY_3.
\end{equation*}
For the first flatness is clear and Hilbert polynomial computation is straightforward, so we skip it.
For the second we note that the map~\eqref{eq:map-cx3} is flat.
Indeed, arguing as in the proof of Proposition~\ref{proposition:dp-p2p2} we can check
that the map~\eqref{eq:map-cx3} is obtained by base change from the map 
\begin{equation*}
\P_S(\cL^\perp) \times_S \P_S(\cL) \to Q = \Fl(1,7;V_1^\vee \otimes V_2^\vee \otimes V_3^\vee) 
\end{equation*}
which is an open embedding.

Recall the embedding of the Weil divisor $\YY_3 \times \P(V_3) \hookrightarrow \cX_3 \times_{\P(V_1^\vee \otimes V_2^\vee \otimes V_3^\vee)} \YY_3$ defined by~\eqref{diagram:yy-cx}.
The preimage of $\YY_3 \times \P(V_3)$ under~\eqref{eq:map-cx3} is a Weil divisor in $\cX \times_S (S \times_{\P(V_1^\vee \otimes V_2^\vee \otimes V_3^\vee)} \YY_3)$ 
such that for any closed point $(b,v_1,v_2)$ of $\cZ'_3 = S \times_{\P(V_1^\vee \otimes V_2^\vee \otimes V_3^\vee)} \YY_3$ its fiber is the Weil divisor 
\begin{equation*}
L := X_b \times_{\P(V_1) \times \P(V_2)} \{(v_1,v_2)\} \subset X_b,
\end{equation*}
i.e., a line on the sextic del Pezzo surface $X_b$ contracted by the projection $X_b \to \P(V_1) \times \P(V_2)$ 
to the point $(v_1,v_2) \in \P(V_1) \times \P(V_2)$.
Thus, by~\eqref{eq:hpd-kernel-xxx3}, the sheaf we are interested in is the twisted ideal \mbox{$\cI_L \otimes \cO_{\P(V_1) \times \P(V_2)}(1,1)$}.
In particular, it is torsion-free, and its Hilbert polynomial equals~$h_3(t)$.

The rest of the proof repeats the argument of Proposition~\ref{proposition:dp-p2p2}.
\end{proof}

The statement of Proposition~\ref{proposition:dp-p1p1p1} can be inverted as follows (and the proof repeats the proof of Lemma~\ref{lemma:reconstructing-family}).

\begin{lemma}
Let $\cX \to S$ be a du Val family of sextic del Pezzo surfaces.
If $\cZ_2(\cX/S) = S \sqcup S \sqcup S$ and the Brauer class $\beta_{\cZ_2}$ is trivial, then Zariski locally over $S$ 
the family $\cX \to S$ can be represented as a family of hyperplane sections of\/ $\P^1 \times \P^1 \times \P^1$.
In particular, if $\cX$ is regular then $\beta_{\cZ_3}$ is trivial.
\end{lemma}

\subsection{Blowup families}\label{subsection:blowups}

For each length 3 subscheme $Y \subset \P^2$, consider the blowup 
\begin{equation*}
\widehat{X} := \Bl_Y(\P^2).
\end{equation*}
Unless $Y$ is the infinitesimal neighborhood of a point (i.e., is given by the square of the maximal ideal of a point), 
$\widehat{X}$ is a weak del Pezzo surface of degree 6.
In particular, the anticanonical class of $\widehat{X}$ is nef and big and
the anticanonical model of $\widehat{X}$ (i.e., the image of $\widehat{X}$ under the anticanonical map) is a sextic du Val del Pezzo surface.

This construction can be also performed in a family.
Let 
\begin{equation}\label{eq:hilb-p2}
S := (\P^2)^{[3]} \setminus \P^2
\end{equation}
be the open subset of the Hilbert cube of the plane parameterizing length 3 subschemes in $\P^2$ avoiding infinitesimal neighborhoods of points.
Let $\cY \subset \P^2 \times S$ be the corresponding family of subschemes.
Let
\begin{equation*}
\hcX := \Bl_\cY(\P^2 \times S)
\end{equation*}
be the blowup and $\cX \to S$ its relative anticanonical model.
In particular, we have a morphism
\begin{equation*}
\hpi \colon \hcX \to \cX.
\end{equation*}
Let $S_1 \subset S$ be the divisor parameterizing subschemes in $\P^2$ contained in a line.
It is easy to see that over~$S_1$ there is a~$\P^1$-bundle $\Delta \to S_1$ (formed by strict transforms of lines supporting the subschemes) 
and an embedding $\Delta \hookrightarrow \hcX$, such that~$\cX$ is the contraction of $\Delta$ to $S_1$.
In particular, over $S_0 := S \setminus S_1$ the map $\hpi$ is an isomorphism.
One can check that~$\hcX$ is a small resolution of singularities of $\cX$.

\begin{proposition}\label{proposition:blowup-p2}
Let $S$ be defined by~\eqref{eq:hilb-p2} and let $\cX \to S$ be the relative anticanonical model of the blowup $\hcX = \Bl_\cY(\P^2 \times S) \to S$.
Then
\begin{equation*}
\cZ_2 \cong \cY
\qquand 
\cZ_3 \cong S \,\mathop\cup\limits_{S_1}\, S,
\end{equation*}
and both Brauer classes $\beta_{\cZ_2}$ and $\beta_{\cZ_3}$ are trivial.
\end{proposition}

Here $S \,\mathop\cup\limits_{S_1}\, S$ is the gluing of two copies of the scheme $S$ along the divisor~$S_1$.
In particular, in this example both $\cX$ and $\cZ_3$ are not regular.
We leave the proof of this proposition to the interested reader.

A similar argument applies to blowups of $\P^1 \times \P^1$.

\begin{proposition}\label{proposition:blowup-p1p1}
Let $S := (\P^1 \times \P^1)^{[2]}$ be the Hilbert square of $\P^1 \times \P^1$ and let $\cY \subset (\P^1 \times \P^1) \times S$ be the universal family of subschemes.
Let $\hcX := \Bl_{\cY}(\P^1 \times \P^1 \times S)$ be the blowup and let $\cX \to S$ be its relative anticanonical model.
Let $S_{1,0} \subset S$ and $S_{0,1} \subset S$ be the divisors parameterizing subschemes contained in a horizontal or a vertical ruling of $\P^1 \times \P^1$ respectively.
Then
\begin{equation*}
\cZ_2 \cong S \,\mathop\cup\limits_{S_{1,0}}\, S \,\mathop\cup\limits_{S_{0,1}}\, S
\qquand 
\cZ_3 \cong \cY
\end{equation*}
and both Brauer classes $\beta_{\cZ_2}$ and $\beta_{\cZ_3}$ are trivial.
\end{proposition}

Note that in this example the base $S$ of the family is proper.

\appendix%
\section{Auslander algebras}\label{appendix:auslander}

Let $Z = \Spec(\kk[t]/t^m)$ be a non-reduced zero-dimensional scheme.
The Auslander algebra $\tR_m$ defined below provides a categorical resolution of the derived category $\bD(Z)$, see~\cite[Section~5]{kuznetsov2015categorical} for details.

The algebra $\tR_m$ is defined as the path algebra of a quiver with relations:
\begin{equation}\label{eq:trm}
\tR_m = \kk\left\{\left. 
\xymatrix@1{\bullet \ar@<-.5ex>[r]_-{\beta_1} & \bullet \ar@<-.5ex>[l]_-{\alpha_1} \ar@<-.5ex>[r]_-{\beta_2} & \bullet \ar@<-.5ex>[l]_-{\alpha_2} \ar@{}[r]|-{\dots} & \bullet \ar@<-.5ex>[r]_-{\beta_{m-1}} & \bullet \ar@<-.5ex>[l]_-{\alpha_{m-1}} } 
\ \right|\ 
\text{$\beta_i\alpha_i = \alpha_{i+1}\beta_{i+1}$ for $1 \le i \le m-2$, $\beta_{m-1}\alpha_{m-1} = 0$}\
\right\}.
\end{equation}
Alternatively, it can be written as a matrix algebra
\begin{equation}
\label{eq:trm-matrix}
\tR_m = \bigoplus_{i,j = 0}^{m-1} (\tR_m)_{ij},
\qquad  
(\tR_m)_{ij} = 
\begin{cases}
\kk[t]/t^{m-i}, & \text{if $i \ge j$},\\
t^{j-i}\kk[t]/t^{m-i}, & \text{if $i \le j$},
\end{cases}
\end{equation}
with multiplication induced by the natural maps $(\tR_m)_{ij} \otimes (\tR_m)_{jk} \to (\tR_m)_{ik}$.
We identify the category of representations of the quiver with the category of left modules over its path algebra.

We denote by $\epsilon_i$ the $i$-th vertex idempotent in $\tR_m$ (in terms of~\eqref{eq:trm-matrix} it is the unit in $(\tR_m)_{ii} \cong \kk[t]/t^{m-i}$).
For every left $\tR_m$-module $M$ we have 
\begin{equation*}
M = \bigoplus_{i=0}^{m-1} \epsilon_i M.
\end{equation*}
We call $M_i := \epsilon_i M$ the {\sf $i$-th component} of $M$, and the vector $(\dim M_0,\dim M_1, \dots) \in \ZZ^m$ the {\sf dimension vector} of $M$.

For each $0 \le i \le m-1$ we denote by $S_i$ the {\sf simple module} of $i$-th vertex of the quiver (its dimension vector is $(\underbrace{0, \dots, 0}_{i},1,\underbrace{0,\dots, 0}_{m-1-i})$) and by
\begin{equation*}
P_i = \tR_m \epsilon_i
\end{equation*}
its indecomposable projective cover ({\sf projective module} of $i$-th vertex).

The algebra $\tR_m$ has finite global dimension (it is bounded by $2m - 2$, see~\cite[Proposition~A.14]{kuznetsov2015categorical})
and its derived category $\bD(\tR_m)$ is generated by an exceptional collection (see Lemma~\ref{lemma:auslander-ext-ei})
consisting of representations~$E_i$ with dimension vectors 
\begin{equation*}
\dim(E_i) = (\underbrace{1,\dots,1}_{i+1},\underbrace{0,\dots,0}_{m-1-i})
\end{equation*}
and with $\beta$-arrows acting by zero and $\alpha$-arrows acting by identity.

\begin{remark}
Actually, the algebra $\tR_m$ is quasihereditary and the modules $E_i$ are its standard modules.
\end{remark}

We call the $E_i$ {\sf standard exceptional modules}.
They have nice projective resolutions
\begin{equation}
\label{eq:projective-resolution-ei}
0 \to P_{i+1} \xrightarrow{\ \beta_{i+1}\ } P_i \to E_i \to 0,
\end{equation}
with the maps induced by the right $\beta_{i+1}$-multiplication.
Using these, it is easy to check that $E_0,\dots,E_{m-1}$ is an exceptional collection and to compute its endomorphism algebra.

\begin{lemma}
\label{lemma:auslander-ext-ei}
The collection $E_0,E_1,\dots,E_{m-1}$ is exceptional and $\Ext^\bullet(E_i,E_j) = \kk \oplus \kk[-1]$ for all $i < j$.
Moreover, the multiplication map
\begin{equation*}
\Ext^p(E_i,E_j) \otimes \Ext^q(E_j,E_k) \to \Ext^{p+q}(E_i,E_k),
\qquad i < j < k
\end{equation*}
is an isomorphism when $p = 0$ or $q = 0$.
\end{lemma}
\begin{proof}
Using~\eqref{eq:projective-resolution-ei} we see that $\Ext^\bullet(E_i,E_j)$ is computed by the complex
\begin{equation*}
(E_j)_i \xrightarrow{\ \beta_{i+1}\ } (E_j)_{i+1}.
\end{equation*}
If $i > j$ both spaces are zero and if $i = j$ the first is $\kk$ and the second is zero, hence the collection is exceptional.
Similarly, if $j > i$ both spaces are $\kk$ and the arrow is zero, hence $\Hom(E_i,E_j) = \Ext^1(E_i,E_j) = \kk$.

For the second statement, note that we have an exact sequence
\begin{equation}\label{eq:resolution-simple}
0 \to E_{i-1} \xrightarrow{\ \alpha_i\ } E_i \to S_i \to 0,
\end{equation} 
which shows that for $j \le k$ the natural map $E_j \to E_k$ is injective and its cokernel is an extension of simple modules $S_l$ with $j+1 \le l \le k$.
On the other hand, \eqref{eq:projective-resolution-ei} implies that $\Ext^\bullet(E_i,S_l) = 0$ as soon as~$l \ge i+2$, hence the map 
\begin{equation*}
\Ext^\bullet(E_i,E_j) \to \Ext^\bullet(E_i,E_k)
\end{equation*}
induced by the embedding $E_j \to E_k$ is an isomorphism. This proves the case when $q = 0$.

Similarly, merging~\eqref{eq:resolution-simple} with~\eqref{eq:projective-resolution-ei} we obtain for $l \ge 1$ a projective resolution
\begin{equation}\label{eq:projective-resolution-simples}
0 \to P_l \xrightarrow{\ (-\alpha_{l+1},\beta_l)\ } P_{l+1} \oplus P_{l-1} \xrightarrow{\ (\beta_{l+1},\alpha_l)\ } P_l \to S_l \to  0
\end{equation}
of the simple module $S_l$.
It follows that $\Ext^\bullet(S_l,E_k)$ for $1 \le l < k$ are computed by the complex
\begin{equation*}
(E_k)_l \xrightarrow{\ (0,1)\ } (E_k)_{l+1} \oplus (E_k)_{l-1} \xrightarrow{\ (-1,0)} (E_k)_l,
\end{equation*}
hence for $1 \le l < k$ we have $\Ext^\bullet(S_l,E_k) = 0$.
On the other hand, the cokernel of the embedding $E_i \to E_j$ is an extension of simple modules $S_l$ with $1 \le l \le j$, hence the map 
\begin{equation*}
\Ext^\bullet(E_j,E_k) \to \Ext^\bullet(E_i,E_k)
\end{equation*}
induced by this embedding is an isomorphism. This proves the case when $p = 0$.
\end{proof}

The following characterization of the categories $\bD(\tR_2)$ and $\bD(\tR_3)$ is quite useful.

\begin{proposition}\label{proposition:rm-equivalence}
Assume $\cT$ is a triangulated category admitting a DG enhancement.

\noindent$(a)$
If $\cT$ is generated by an exceptional pair $(\cL_0,\cL_1)$ such that $\Ext^\bullet(\cL_0,\cL_1) \cong \kk \oplus \kk[-1]$ then there is an equivalence of categories $\cT \cong \bD(\tR_2)$ taking $\cL_i$ to $E_i$.

\noindent$(b)$
If $\cT$ is generated by an exceptional triple $(\cL_0,\cL_1,\cL_2)$ 
such that $\Ext^\bullet(\cL_i,\cL_j) \cong \kk \oplus \kk[-1]$ for all~\mbox{$i < j$}
and the multiplication map $\Ext^p(\cL_0,\cL_1) \otimes \Ext^q(\cL_1,\cL_2) \to  \Ext^{p+q}(\cL_0,\cL_2)$ 
is an isomorphism when $p = 0$ or $q = 0$, then there is an equivalence of categories $\cT \cong \bD(\tR_3)$ taking $\cL_i$ to $E_i$.
\end{proposition}
\begin{proof}
Since the category $\cT$ is enhanced, there is an equivalence of $\cT$ with the derived category of the DG algebra $\RHom_\cT(\oplus \cL_i,\oplus \cL_i)$.
By the assumption, its cohomology is isomorphic (as a graded algebra) to the graded algebra $\Ext^\bullet_{\bD(\tR_m)}(\oplus E_i, \oplus E_i)$.
To get the desired equivalence, it remains to check that the latter algebra (considered as a DG algebra with trivial differential) is formal.

But formality is clear, since any higher $A_\infty$-operation $\bm_i$ (with $i \ge 3$) requires at least three non-trivial arguments, 
so one needs the quiver to have at least four vertices to admit such an operation.
So, for $\tR_2$ and $\tR_3$ all higher operations vanish and the algebra is formal.
\end{proof}

\begin{remark}
Let $m \ge 4$ and assume $\cT$ is an enhanced triangulated category generated by an exceptional collection $\cL_0,\cL_1,\dots,\cL_{m-1}$ satisfying the properties of Lemma~\ref{lemma:auslander-ext-ei}.
To establish an equivalence of $\cT$ with~$\bD(\tR_m)$ one should additionally check that higher $A_\infty$-operations in $\cT$ vanish when one of the arguments is contained in the space $\Hom(\cL_i,\cL_{i+1}) = \kk$.
\end{remark}

The endomorphism algebra of the projective module $P_0$ is
\begin{equation*}
\End_{\tR_m}(P_0) = \epsilon_0 \tR_m \epsilon_0 = (\tR_m)_{00} = \kk[t]/t^m.
\end{equation*}
Using this we define an adjoint pair of functors
\begin{equation}
\label{eq:pms}
\begin{aligned}
\pi_{m*} & \colon \tR_m\mmod \to (\kk[t]/t^m)\mmod, \quad & M & \mapsto \Hom_{\tR_m}(P_0,M) = \epsilon_0 M,\\
\pi_m^* & \colon (\kk[t]/t^m)\mmod \to \tR_m\mmod, \quad & N & \mapsto P_0 \otimes_{\kk[t]/t^m} N.
\end{aligned}
\end{equation} 
Following our convention, we denote in the same way their derived functors.

Note that the functor $\pi_{m*} \colon \bD(\tR_m) \to \bD(\kk[t]/t^m)$ preserves boundedness, while its left adjoint $\pi_m^*$ does not.
Recall that $\langle - \rangle^\oplus$ denotes the minimal triangulated subcategory of $\bD^-(\tR_m)$ closed under infinite direct sums.

\begin{proposition}\label{proposition:auslander-localization}
The functor $\pi_m^* \colon \bD^-(\kk[t]/t^m) \to \bD^-(\tR_m)$ is fully faithful, 
and its right adjoint functor~$\pi_{m*} \colon \bD(\tR_m) \to \bD(\kk[t]/t^m)$ preserves boundedness and is essentially surjective. 
We have 
\begin{equation*}
\Ker \pi_{m*} = \langle S_1,\dots,S_{m-1} \rangle^\oplus
\qquand
\Ima \pi_m^* = \langle P_0 \rangle^\oplus = {}^\perp\langle S_1,\dots,S_{m-1} \rangle.
\end{equation*}
Moreover, $\pi_{m*}(S_0) = \kk$ and $\pi_{m*}(P_0) = \kk[t]/t^m$.
\end{proposition}
\begin{proof}
The functor $\pi_m^*$ is fully faithful by~\cite[Theorem~5.23]{kuznetsov2015categorical} 
and $\pi_{m*} \circ \pi_m^* \cong \id$ by~\cite[(48)]{kuznetsov2015categorical}.
Furthermore, $\pi_{m*}$ is exact by definition, hence for any $N \in \bD(\kk[t]/t^m)$, 
taking $M := \tau^{\ge p}\pi_m^*(N)$ (the canonical truncation) with $p \ll 0$, we obtain $N \cong \pi_{m_*}M$, 
hence $\pi_{m*}$ is essentially surjective.

The image of~$\pi_m^*$ by definition equals the subcategory $\langle P_0 \rangle^\oplus \subset \bD^-(\tR_m)$ and
this category is evidently equal to the orthogonal of the simple modules $S_1$, \dots, $S_{m-1}$.
Since the functor $\pi_{m*}$ is the right adjoint of~$\pi_m^*$, 
we have $\Ker \pi_{m*} = (\Ima \pi_m^*)^\perp = P_0^\perp = \langle S_1,\dots,S_{m-1} \rangle^\oplus$.
Finally, applying~\eqref{eq:pms} we easily get~$\pi_{m*}(S_0) = \epsilon_0S_0 = \kk$ 
and $\pi_{m*}(P_0) = \epsilon_0P_0 = \epsilon_0\tR_m\epsilon_0 = (\tR_m)_{00} = \kk[t]/t^m$.
\end{proof}

\section{Moduli stack of sextic du Val del Pezzo surfaces}\label{appendix:moduli-stack}

The moduli stack $\DP$ of sextic du Val del Pezzo surfaces is the fibered category over $(\mathrm{Sch}/\kk)$ whose fiber over a $\kk$-scheme $S$ is 
the groupoid of all du Val $S$-families of sextic del Pezzo surfaces $f \colon \cX \to S$ (in the sense of Definition~\ref{definition:family-dp6}).
A morphism from $f' \colon \cX' \to S'$ to $f \colon \cX \to S$ is a fiber product diagram
\begin{equation*}
\xymatrix{
\cX' \ar[r] \ar[d]_{f'} & \cX \ar[d]^f \\ S' \ar[r] & S
}
\end{equation*}
The main result of this section was communicated by Jenya Tevelev.

\begin{theorem}\label{theorem:moduli-stack-dp6}
The moduli stack $\DP$ is a smooth Artin stack of finite type over $\kk$.
\end{theorem}
\begin{proof}
By Hilbert scheme argument, the stack $\DP$ is an Artin stack of finite type.
So, by deformation theory it is enough to check that deformations of a sextic du Val del Pezzo surface $X$ over an algebraically closed field are unobstructed, i.e., that $\Ext^p(\Omega_X,\cO_X) = 0$ for $p \ge 2$.
On the other hand, since $X$ has isolated hypersurface singularities, 
the sheaf $\Omega_X$ has a locally free resolution of length~1, 
hence~$\CExt^p(\Omega_X,\cO_X) = 0$ for $p \ge 2$ and $\CExt^1(\Omega_X,\cO_X)$ has zero-dimensional support.
By the local-to-global spectral sequence this means that 
\begin{equation*}
\Ext^{\ge 2}(\Omega_X,\cO_X) = H^{\ge 2}(X,T_X).
\end{equation*}
Since $X$ is a surface, we just have to check that $H^2(X,T_X) = 0$.
This holds by~\cite[Proposition~3.1]{hacking2010smoothable}.
\end{proof}

\begin{remark}
The stack $\DP$ is not separated ---
one can construct two du Val families $\cX \to \AA^1$ and~$\cX' \to \AA^1$ that are isomorphic over $\AA^1 \setminus \{0\}$, but differ by a small birational transformation over the whole base.
The simplest example is to consider the blowup $\hcX = \Bl_{\cY}(\PP^2 \times \AA^1) \to \AA^1$, 
where $\cY = \Big(\{(1,0,0)\} \times \AA^1\Big) \sqcup \Big(\{(0,1,0)\} \times \AA^1\Big) \sqcup \Big\{ ((1,1,t),t) \mid t \in \AA^1 \Big\}$,
and define $\cX \to \AA^1$ as the family of relative anticanonical models of $\hcX$.
Then $\cX \to \AA^1$ is smooth over~$\AA^1 \times \{0\}$ and degenerates to a singular del Pezzo surface (of type 1) at point $t = 0$.
On the other hand, the family $\cX \times_{\AA^1} (\AA^1 \setminus \{0\}) \to \AA^1 \setminus \{0\}$ is isomorphic to the trivial family, hence can be extended to a trivial family $\cX' \to \AA^1$.
Thus, the families~$\cX$ and $\cX'$ are isomorphic over $\AA^1 \setminus \{0\}$, but have different fibers over $0$.
\end{remark}

\section{Homological projective duality for $\P^2 \times \P^2$ and $\Fl(1,2;3)$}\label{appendix:hpd-p2p2}

We refer to~\cite{kuznetsov2007hpd,icm2014,thomas2015notes} for the definition and a review of homological projective duality.
In this section we construct the homologically projective dual variety for $\P^2 \times \P^2$ and the flag variety~$\Fl(1,2;3)$ with respect to a certain symmetric rectangular Lefschetz decomposition.

Denote by $W_1$ and $W_2$ two three-dimensional vector spaces, consider the product
\begin{equation*}
\XX_2 := \P(W_1) \times \P(W_2),
\end{equation*}
its Segre embedding $\XX_2 \hookrightarrow \P(W_1 \otimes W_2) =: \P(\WW)$,
and the standard exceptional collection on ${\XX_2}$:
\begin{equation*}
\bD({\XX_2}) = \langle  \cO_{\XX_2}, \cO_{\XX_2}(1,0), \cO_{\XX_2}(2,0), \cO_{\XX_2}(0,1), \cO_{\XX_2}(1,1), \cO_{\XX_2}(2,1), \cO_{\XX_2}(0,2), \cO_{\XX_2}(1,2), \cO_{\XX_2}(2,2) \rangle. 
\end{equation*}
We modify this collection slightly to turn it into a symmetric rectangular Lefschetz decomposition with respect to the line bundle~$\cO_{\XX_2}(1,1)$.
For this we mutate $\cO_{\XX_2}(2,0)$ and $\cO_{\XX_2}(0,2)$ to the far left.
An easy computation shows that the result is the following exceptional collection
\begin{equation*}
\bD({\XX_2}) = \langle \cO_{\XX_2}(0,-1), \cO_{\XX_2}(-1,0), \cO_{\XX_2}, \cO_{\XX_2}(1,0), \cO_{\XX_2}(0,1), \cO_{\XX_2}(1,1), \cO_{\XX_2}(2,1), \cO_{\XX_2}(1,2), \cO_{\XX_2}(2,2) \rangle. 
\end{equation*}
Clearly, this is a rectangular Lefschetz collection with respect to $\cO_{\XX_2}(1,1)$ with three blocks equal to 
\begin{equation}\label{eq:xx3-lefschetz}
\cA_{\XX_2}^{\rm{sym}} = \langle \cO_{\XX_2}(0,-1), \cO_{\XX_2}(-1,0), \cO_{\XX_2} \rangle. 
\end{equation}
It is symmetric with respect to the transposition of factors.

The homological projective duality of $\XX_2$ with respect to the standard Lefschetz decomposition
\begin{equation*}
\bD(\XX_2) = \langle \bD(\P(W_1)), \bD(\P(W_1)) \otimes \cO_{\XX_2}(1,1), \bD(\P(W_1)) \otimes \cO_{\XX_2}(2,2) \rangle
\end{equation*}
is described in~\cite{bernardara2016determinantal}.
For this the linear homological projective duality argument~\cite[Section~8]{kuznetsov2007hpd} is used.
Indeed, the scheme~${\XX_2}$ can be considered as a projectivization of a vector bundle
\begin{equation*}
{\XX_2} \cong \P_{\P(W_1)}(\cW_2),
\qquad 
\cW_2 := W_2\otimes \cO_{\P(W_1)}(-1) \subset \WW \otimes \cO_{\P(W_1)}.
\end{equation*}
Consequently, by~\cite[Corollary~8.3]{kuznetsov2007hpd} the homological projectively dual of ${\XX_2}$ with respect to the Lefschetz decomposition with the first block $\bD(\P(W_1))$ is
\begin{equation}
\label{eq:def-yy2}
{\YY_2} := \P_{\P(W_1)}(\cW_2^\perp),
\end{equation}
where $\cW_2^\perp := \Ker(\WW^\vee \otimes \cO_{\P(W_1)} \to \cW_2^\vee) \cong W_2^\vee \otimes \Omega_{\P(W_1)}(1)$ is a rank 6 vector bundle on~$\P(W_1)$.
In the next theorem we show that the result of homological projective duality with respect to~\eqref{eq:xx3-lefschetz} is the same.

\begin{theorem}\label{theorem:hpd-p2p2}
The variety ${\YY_2}$ is homologically projectively dual to the variety ${\XX_2}$ with respect to the Lefschetz decomposition of $\bD(\XX_2)$ with first block~\eqref{eq:xx3-lefschetz}.
\end{theorem}
\begin{proof}
Let ${\cX_2} \subset {\XX_2} \times \P(\WW^\vee)$ be the universal hyperplane section of ${\XX_2}$.
By~\cite[Theorem~8.2]{kuznetsov2007hpd}, there is a semiorthogonal decomposition
\begin{equation}
\label{eq:dbxx3-1}
\bD({\cX_2}) = \langle i_*\phi^*(\bD({\YY_2})), \bD(\P(W_1) \times \P(\WW^\vee)) \otimes \cO_{\XX_2}(1,1), \bD(\P(W_1) \times \P(\WW^\vee)) \otimes \cO_{\XX_2}(2,2) \rangle,
\end{equation}
where the morphisms $i$ and $\phi$ are defined by the Cartesian diagram
\begin{equation}
\label{diagram:yy-cx3}
\vcenter{\xymatrix{
{\YY_2} \times \P(W_2) \ar[r]^-i \ar[d]_\phi & {\cX_2} \ar[d]^{p_{\cX_2}} \\
{\YY_2} \ar[r]^-{p_{\YY_2}} & \P(W_1) \times \P(\WW^\vee) \ar[r]^-{p_2} & \P(\WW^\vee)
}}
\end{equation}
and the map $p_{\YY_2}$ is induced by the embedding $\cW_2^\perp \hookrightarrow \WW^\vee \otimes \cO_{\P(W_1)}$.
We modify~\eqref{eq:dbxx3-1} by a sequence of mutations to change it to the form we need.

First, using the standard exceptional collection $\bD(\P(W_1)) = \langle \cO_{\P(W_1)}(-1), \cO_{\P(W_1)}, \cO_{\P(W_1)}(1) \rangle$,
decomposition \eqref{eq:dbxx3-1} can be rewritten as
\begin{multline*}
\bD({\cX_2}) = \langle i_*\phi^*(\bD({\YY_2})), \\ 
\bD(\P(\WW^\vee)) \otimes \cO_{\XX_2}(0,1), 
\bD(\P(\WW^\vee)) \otimes \cO_{\XX_2}(1,1), 
\bD(\P(\WW^\vee)) \otimes \cO_{\XX_2}(2,1), \\
\bD(\P(\WW^\vee)) \otimes \cO_{\XX_2}(1,2), 
\bD(\P(\WW^\vee)) \otimes \cO_{\XX_2}(2,2), 
\bD(\P(\WW^\vee)) \otimes \cO_{\XX_2}(3,2)
\rangle.
\end{multline*}
Mutating the last component to the far left and taking into account that $\omega_{\cX_2} \cong \cO_{\XX_2}(-2,-2)$ up to a line bundle pulled back from $\P(\WW^\vee)$, we obtain a semiorthogonal decomposition
\begin{multline*}
\bD({\cX_2}) = \langle \bD(\P(\WW^\vee)) \otimes \cO_{\XX_2}(1,0),
i_*\phi^*(\bD({\YY_2})), \\ 
\bD(\P(\WW^\vee)) \otimes \cO_{\XX_2}(0,1), 
\bD(\P(\WW^\vee)) \otimes \cO_{\XX_2}(1,1), \\
\bD(\P(\WW^\vee)) \otimes \cO_{\XX_2}(2,1), 
\bD(\P(\WW^\vee)) \otimes \cO_{\XX_2}(1,2), 
\bD(\P(\WW^\vee)) \otimes \cO_{\XX_2}(2,2)
\rangle.
\end{multline*}

Next, mutating the second component one step to the left, we get
\begin{multline}\label{eq:dbxx3-2}
\bD({\cX_2}) = \langle \Phi(\bD({\YY_2})), \\
\bD(\P(\WW^\vee)) \otimes \cO_{\XX_2}(1,0),
\bD(\P(\WW^\vee)) \otimes \cO_{\XX_2}(0,1), 
\bD(\P(\WW^\vee)) \otimes \cO_{\XX_2}(1,1), \\
\bD(\P(\WW^\vee)) \otimes \cO_{\XX_2}(2,1), 
\bD(\P(\WW^\vee)) \otimes \cO_{\XX_2}(1,2), 
\bD(\P(\WW^\vee)) \otimes \cO_{\XX_2}(2,2)
\rangle.
\end{multline}
where $\Phi = \LL_{\bD(\P(\WW^\vee)) \otimes \cO_{\XX_2}(1,0)} \circ i_* \circ \phi^* \colon \bD({\YY_2}) \to \bD({\cX_2})$.

This almost proves the result. 
The only small thing left is to show that the functor $\Phi$ is a Fourier--Mukai functor 
whose kernel is supported on the fiber product ${\YY_2} \times_{\P(\WW^\vee)} {\cX_2}$.
For this we note that by~\eqref{diagram:yy-cx3} the functor $i_* \circ \phi^*$ is a Fourier--Mukai functor 
with kernel $j_*\cO_{{\YY_2} \times \P(W_2)}$, 
where $j \colon \YY_2 \times \P(W_2) \to \YY_2 \times_{\P(\WW^\vee)} \cX_2$ is the embedding 
induced by the commutative square~\eqref{diagram:yy-cx3}.
Moreover, using notation from the diagram~\eqref{diagram:yy-cx3}, 
the projection functor onto $\bD(\P(\WW^\vee)) \otimes \cO_{\XX_2}(1,0)$ is given by
\begin{equation*}
\cF \mapsto p_{\cX_2}^*(p_2^*p_{2*}(p_{\cX_2*}\cF \otimes \cO_{\P(W_1)}(-1)) \otimes \cO_{\P(W_1)}(1)),
\end{equation*}
and its composition with $i_* \circ \phi^*$ is given by
\begin{equation*}
\cF \mapsto p_{\cX_2}^*(p_2^*p_{2*}(p_{\YY_2*}\cF \otimes \cO_{\P(W_1)}(-1)) \otimes \cO_{\P(W_1)}(1)).
\end{equation*}
In other words, this composition is a Fourier--Mukai functor with kernel $p_{\YY_2}^*(\cO_{\P(W_1)}(-1)) \boxtimes p_{\cX_2}^*(\cO_{\P(W_1)}(1))$ on $\YY_2 \times_{\P(\WW^\vee)} \cX_2$.
Therefore, the functor $\Phi$ fits into a distinguished triangle
\begin{equation*}
\Phi_{p_{\YY_2}^*(\cO_{\P(W_1)}(-1)) \boxtimes p_{\cX_2}^*(\cO_{\P(W_1)}(1))} \to \Phi_{j_*\cO_{{\YY_2} \times \P(W_2)}} \to \Phi.
\end{equation*}
It follows from~\eqref{diagram:yy-cx3} that 
\begin{equation*}
j^*(p_{\YY_2}^*(\cO_{\P(W_1)}(-1)) \boxtimes p_{\cX_2}^*(\cO_{\P(W_1)}(1))) \cong \cO_{{\YY_2} \times \P(W_2)}, 
\end{equation*}
hence by~\cite{anno2012adjunction} the above triangle of functors is induced by (a rotation of the twist of) 
the triangle associated with the standard exact sequence
\begin{equation*}
0 \to \cI_{{\YY_2} \times \P(W_2), {\YY_2} \times_{\P(\WW^\vee)} {\cX_2}} \to \cO_{{\YY_2} \times_{\P(\WW^\vee)} {\cX_2}} \to j_*\cO_{{\YY_2} \times \P(W_2)} \to 0.
\end{equation*}
It finally proves, that up to an irrelevant shift, the functor $\Phi$ is indeed a Fourier--Mukai functor with the kernel
\begin{equation}\label{eq:hpd-kernel-xx3}
\EE_2 := \cI_{{\YY_2} \times \P(W_2), {\YY_2} \times_{\P(\WW^\vee)} {\cX_2}} \otimes (p_{\YY_2}^*(\cO_{\P(W_1)}(-1)) \boxtimes p_{\cX_2}^*(\cO_{\P(W_1)}(1))) \in \bD({\YY_2} \times_{\P(\WW^\vee)} \cX_2),
\end{equation}
and thus completes the proof of the theorem.
\end{proof}

\begin{remark}
We could, of course, exchange the role of $W_1$ and $W_2$ in the construction.
Then we would get a slightly different homological projectively dual variety as the result, that is
\begin{equation*}
\YY_2' = \P_{\P(W_2)}(W_1^\vee \otimes \Omega_{\P(W_2)}(1)).
\end{equation*}
Note that the natural maps $\YY_2 \to \P(\WW^\vee)$ and $\YY_2' \to \P(\WW^\vee)$ are birational onto the same discriminant cubic hypersurface $\cD_{W_1,W_2} \subset \P(\WW^\vee) = \P(W_1^\vee \otimes W_2^\vee)$
and provide two small resolutions of singularities of~$\cD_{W_1,W_2}$, related to each other by a flop, identifying their derived categories.
So, the homological projectively dual of $\XX_2$ as a category is unambiguously defined, but has two different geometric models~$\YY_2$ and~$\YY_2'$, breaking down its inner symmetry.
\end{remark}

One can also use the above result to construct a symmetric homological projective duality 
for the flag variety $\Fl(W) = \Fl(1,2;W)$ of a three-dimensional vector space $W$.
Note that $\Fl(W)$ is a smooth hyperplane section of $\P(W) \times \P(W^\vee)$ corresponding 
to the natural bilinear pairing $b_0$ between $W$ and~$W^\vee$.
Since~$b_0$ is nondegenerate, it does not lie in the image $\cD_{W,W^\vee}$ of $\YY_2$ in $\P(W \otimes W^\vee)$, 
hence by homological projective duality \cite[Theorem~6.3]{kuznetsov2007hpd} we obtain a semiorthogonal decomposition
\begin{equation*}
\bD(\Fl(W)) = \langle \cO_{\Fl(W)}(0,-1), \cO_{\Fl(W)}(-1,0), \cO_{\Fl(W)}, \cO_{\Fl(W)}(1,0), \cO_{\Fl(W)}(0,1), \cO_{\Fl(W)}(1,1) \rangle. 
\end{equation*}
Clearly, this is a Lefschetz collection with respect to $\cO_{\Fl(W)}(1,1)$ with two blocks equal to 
\begin{equation}\label{eq:flw-lefschetz}
\cA_{\Fl(W)}^{\rm{sym}} = \langle \cO_{\Fl(W)}(0,-1), \cO_{\Fl(W)}(-1,0), \cO_{\Fl(W)} \rangle. 
\end{equation}
On the other hand, the linear projection with center at $b_0$ defines a regular map $\YY_2 \to \P((W \otimes W^\vee)/b_0)$.

\begin{corollary}\label{corollary:hpd-fl3}
The variety ${\YY_2}$ defined by~\eqref{eq:def-yy2} is homological projectively dual to the flag variety $\Fl(W)$ with respect to the Lefschetz decomposition of $\bD(\Fl(W))$ with first block~\eqref{eq:flw-lefschetz}.
\end{corollary}
\begin{proof}
Follows from Theorem~\ref{theorem:hpd-p2p2} by~\cite[Theorem~3.4]{carocci2015homological}. 
(see also~\cite[Proposition~A.10]{kuznetsov2018cones}).
\end{proof}

\section{Homological projective duality for $\P^1 \times \P^1 \times \P^1$}\label{appendix:hpd-p1p1p1}

In this section we construct the homologically projective dual variety for $\P^1 \times \P^1 \times \P^1$ with respect to a certain symmetric rectangular Lefschetz decomposition.

Denote by $V_1$, $V_2$, and $V_3$ three two-dimensional vector spaces, consider the product
\begin{equation*}
{\XX_3} := \P(V_1) \times \P(V_2) \times \P(V_3),
\end{equation*}
its Segre embedding $\XX_3  \hookrightarrow \P(V_1 \otimes V_2 \otimes V_3) =: \P(\VV)$,
and the standard exceptional collection on ${\XX_3}$:
\begin{equation*}
\bD({\XX_3}) = \langle  \cO_{\XX_3}, \cO_{\XX_3}(1,0,0), \cO_{\XX_3}(0,1,0), \cO_{\XX_3}(1,1,0), \cO_{\XX_3}(0,0,1), \cO_{\XX_3}(1,0,1), \cO_{\XX_3}(0,1,1), \cO_{\XX_3}(1,1,1) \rangle. 
\end{equation*}
We modify this collection slightly to turn it into a symmetric rectangular Lefschetz decomposition with respect to the line bundle $\cO_{\XX_3}(1,1,1)$.
For this we mutate $\cO_{\XX_3}(1,1,0)$, $\cO_{\XX_3}(1,0,1)$ and $\cO_{\XX_3}(0,1,1)$ to the far right.
An easy computation shows that what we get is an exceptional collection
\begin{equation*}
\bD({\XX_3}) = \langle  
\cO_{\XX_3}, \cO_{\XX_3}(1,0,0), \cO_{\XX_3}(0,1,0), \cO_{\XX_3}(0,0,1), 
\cO_{\XX_3}(1,1,1), \cO_{\XX_3}(2,1,1), \cO_{\XX_3}(1,2,1), \cO_{\XX_3}(1,1,2) 
\rangle. 
\end{equation*}
Clearly, this is a rectangular Lefschetz collection with respect to $\cO_{\XX_3}(1,1,1)$ with two blocks equal to 
\begin{equation}
\label{eq:xx-lefschetz}
\cA_{\XX_3}^{\rm{sym}} = \langle  \cO_{\XX_3}, \cO_{\XX_3}(1,0,0), \cO_{\XX_3}(0,1,0), \cO_{\XX_3}(0,0,1) \rangle. 
\end{equation}
It is symmetric with respect to the action of the group $\fS_3$ by permutations of factors.

As in Appendix~\ref{appendix:hpd-p2p2}, we use a small modification of linear homological projective duality.
Again, the scheme $\XX_3$ can be represented as a projectivization of a vector bundle
\begin{equation*}
{\XX_3} \cong \P_{\P(V_1) \times \P(V_2)}(\cV_3),
\qquad 
\cV_3 := V_3\otimes \cO_{\P(V_1) \times \P(V_2)}(-1,-1) \subset \VV \otimes \cO_{\P(V_1) \times \P(V_2)}.
\end{equation*}
Consequently, by~\cite[Corollary~8.3]{kuznetsov2007hpd} the homological projectively dual of ${\XX_3}$ with respect to the Lefschetz decomposition with the first block $\bD(\P(V_1) \times \P(V_2))$ is
\begin{equation}\label{eq:def-yy3}
{\YY_3} := \P_{\P(V_1) \times \P(V_2)}(\cV_3^\perp),
\end{equation}
where $\cV_3^\perp := \Ker(\VV^\vee \otimes \cO_{\P(V_1) \times \P(V_2)} \to \cV_3^\vee) \cong V_3^\vee \otimes \Omega_{\P(V_1 \otimes V_2)}(1) \vert_{\P(V_1) \times \P(V_2)}$ is a rank 6 vector bundle on~$\P(V_1) \times \P(V_2)$.

\begin{theorem}\label{theorem:hpd-p1p1p1}
The variety ${\YY_3}$ is homologically projectively dual to the variety ${\XX_3}$ with respect to the Lefschetz decomposition of $\bD(\XX_3)$ with first block~\eqref{eq:xx-lefschetz}.
\end{theorem}
\begin{proof}
Let ${\cX_3} \subset {\XX_3} \times \P(\VV^\vee)$ be the universal hyperplane section of ${\XX_3}$.
By~\cite[Theorem~8.2]{kuznetsov2007hpd}, there is a semiorthogonal decomposition
\begin{equation}\label{eq:dbxx1}
\bD({\cX_3}) = \langle i_*\phi^*(\bD({\YY_3})), \bD(\P(V_1) \times \P(V_2) \times \P(\VV^\vee)) \otimes \cO_{\XX_3}(1,1,1) \rangle,
\end{equation}
where the morphisms $i$ and $\phi$ are defined by the commutative diagram
\begin{equation}\label{diagram:yy-cx}
\vcenter{\xymatrix{
{\YY_3} \times \P(V_3) \ar[r]^-i \ar[d]_\phi & {\cX_3} \ar[d]^{p_{\cX_3}} \\
{\YY_3} \ar[r]^-{p_{\YY_3}} & \P(V_1) \times \P(V_2) \times \P(\VV^\vee) \ar[r] & \P(\VV^\vee)
}}
\end{equation}
and the map $p_{\YY_3}$ is induced by the embedding $\cV_3^\perp \hookrightarrow \VV^\vee \otimes \cO_{\P(V_1) \times \P(V_2)}$.
We modify~\eqref{eq:dbxx1} by a sequence of mutations to change it to the form we need.

First, using the standard exceptional collection 
\begin{equation*}
\bD(\P(V_1) \times \P(V_2)) = \langle \cO_{\P(V_1) \times \P(V_2)}, \cO_{\P(V_1) \times \P(V_2)}(1,0), \cO_{\P(V_1) \times \P(V_2)}(0,1), \cO_{\P(V_1) \times \P(V_2)}(1,1) \rangle
\end{equation*}
in $\bD(\P(V_1) \times \P(V_2))$, 
we rewrite \eqref{eq:dbxx1} as
\begin{multline*}
\bD({\cX_3}) = \langle i_*\phi^*(\bD({\YY_3})), \bD(\P(\VV^\vee)) \otimes \cO_{\XX_3}(1,1,1), \\ \bD(\P(\VV^\vee)) \otimes \cO_{\XX_3}(2,1,1), \bD(\P(\VV^\vee)) \otimes \cO_{\XX_3}(1,2,1), \bD(\P(\VV^\vee)) \otimes \cO_{\XX_3}(2,2,1) \rangle.
\end{multline*}
Mutating the last component to the far left, and taking into account that $\omega_{\cX_3} \cong \cO_{\XX_3}(-1,-1,-1)$ up to a line bundle pulled back from $\P(\VV^\vee)$, we get a semiorthogonal decomposition
\begin{multline*}
\bD({\cX_3}) = \langle \bD(\P(\VV^\vee)) \otimes \cO_{\XX_3}(1,1,0), i_*\phi^*(\bD({\YY_3})), \bD(\P(\VV^\vee)) \otimes \cO_{\XX_3}(1,1,1), \\ \bD(\P(\VV^\vee)) \otimes \cO_{\XX_3}(2,1,1), \bD(\P(\VV^\vee)) \otimes \cO_{\XX_3}(1,2,1) \rangle.
\end{multline*}

Next, mutating the second component to the left, we get
\begin{multline*}
\bD({\cX_3}) = \langle \Phi(\bD({\YY_3})), \bD(\P(\VV^\vee)) \otimes \cO_{\XX_3}(1,1,0), \bD(\P(\VV^\vee)) \otimes \cO_{\XX_3}(1,1,1), \\ \bD(\P(\VV^\vee)) \otimes \cO_{\XX_3}(2,1,1), \bD(\P(\VV^\vee)) \otimes \cO_{\XX_3}(1,2,1) \rangle.
\end{multline*}
where $\Phi = \LL_{\bD(\P(\VV^\vee)) \otimes \cO_{\XX_3}(1,1,0)} \circ i_* \circ \phi^* \colon \bD({\YY_3}) \to \bD({\cX_3})$.

Finally, mutating the second component to the far right and using the pullback 
of the standard exact sequence $0 \to \cO_{\P(V_3)} \to V_3 \otimes \cO_{\P(V_3)}(1) \to \cO_{\P(V_3)}(2) \to 0$, we obtain
\begin{multline}\label{eq:dbxx6}
\bD({\cX_3}) = \langle \Phi(\bD({\YY_3})), \bD(\P(\VV^\vee)) \otimes \cO_{\XX_3}(1,1,1), \\ \bD(\P(\VV^\vee)) \otimes \cO_{\XX_3}(2,1,1), \bD(\P(\VV^\vee)) \otimes \cO_{\XX_3}(1,2,1), \bD(\P(\VV^\vee)) \otimes \cO_{\XX_3}(1,1,2) \rangle.
\end{multline}

As before, this almost proves the result. 
The only small thing left is to show that the functor $\Phi$ is a Fourier--Mukai functor 
whose kernel is supported on the fiber product ${\YY_3} \times_{\P(\VV^\vee)} {\cX_3}$.
The same computations as in the proof of Theorem~\ref{theorem:hpd-p2p2} show that the functor~$\Phi$ fits into a distinguished triangle
\begin{equation*}
\Phi_{p_{\YY_3}^*(\cO_{\P(V_1) \times \P(V_2)}(-1,-1)) \boxtimes p_{\cX_3}^*(\cO_{\P(V_1) \times \P(V_2)}(1,1))} \to \Phi_{j_*\cO_{{\YY_3} \times \P(V_3)}} \to \Phi,
\end{equation*}
with notation introduced in~\eqref{diagram:yy-cx}.
Again, it follows that this triangle is associated (up to a twist and a rotation) with the standard exact sequence
\begin{equation*}
0 \to \cI_{{\YY_3} \times \P(V_3), {\YY_3} \times_{\P(\VV^\vee)} {\cX_3}} \to \cO_{{\YY_3} \times_{\P(\VV^\vee)} {\cX_3}} \to j_*\cO_{{\YY_3} \times \P(V_3)} \to 0,
\end{equation*}
hence 
the functor $\Phi$ is indeed a Fourier--Mukai functor with kernel
\begin{equation}\label{eq:hpd-kernel-xxx3}
\EE_3 := \cI_{{\YY_3} \times \P(V_3), {\YY_3} \times_{\P(\VV^\vee)} {\cX_3}} \otimes (p_{\YY_3}^*(\cO_{\P(V_1) \times \P(V_2)}(-1,-1)) \boxtimes p_{\cX_3}^*(\cO_{\P(V_1) \times \P(V_2)}(1,1)))
\end{equation}
and thus completes the proof of the theorem.
\end{proof}

\begin{remark}
\label{remark:fibers-resolution-dvvv}
The natural map ${\YY_3} \to \P(\VV^\vee)$ is generically finite of degree 2, 
and its branch divisor is the Cayley quartic $\overline\rO_6$ (see Section~\ref{subsection:p1p1p1}).
Moreover, the fiber of this morphism is isomorphic to~$\P^1$ over~$\rO_4$, and to a reducible conic over~$\rO_3$.
In particular, ${\YY_3}$ provides a small resolution of singularities 
of the double cover $\cD_{V_1,V_2,V_3}$ of~$\P(\VV^\vee)$ branched over~$\overline\rO_6$.
\end{remark}

\begin{remark}
We could, of course, exchange the role of $V_i$ in the construction.
Then we would get a slightly different homological projectively dual variety as the result, 
that is for each permutation ${\sigma} \in \fS_3$ we would get
\begin{equation*}
\YY_3^{\sigma} = \P_{\P(V_{{\sigma}(1)}) \times \P(V_{{\sigma}(2)})} 
\Big(V_{{\sigma}(3)}^\vee \otimes \Omega_{\P(V_{{\sigma}(1)} \otimes V_{{\sigma}(2)})}(1)\vert_{\P(V_{{\sigma}(1)}) \times \P(V_{{\sigma}(2)})}\Big).
\end{equation*}
Note that for each ${\sigma}$ the natural map $\YY_3^{\sigma} \to \P(\VV^\vee)$ factors 
through a birational morphism onto the double cover $\cD_{V_1,V_2,V_3}$ of~$\P(\VV^\vee)$ branched over $\overline\rO_6$.
They provide six small resolutions of singularities of~$\cD_{V_1,V_2,V_3}$, 
related to each other by flops, identifying their derived categories.
So, the homological projectively dual of~$\XX_3$ as a category is unambiguously defined, 
but has six different geometric models~$\YY_3^{\sigma}$, ${\sigma} \in \fS_3$, breaking down its inner symmetry.
\end{remark}

\end{document}